\documentclass[a4paper, 12pt]{article}
\usepackage[margin=2.45cm]{geometry}
\usepackage[utf8]{inputenc}
\usepackage[nosumlimits]{amsmath}
\usepackage{amssymb,stmaryrd,amsthm,MnSymbol}
\usepackage{amsfonts,amssymb,amsopn}
 \usepackage[colorlinks=true,
            linkcolor=red,
            filecolor=red,
            citecolor=blue,
            urlcolor=blue,]{hyperref}
\usepackage[dvipsnames]{xcolor}
\usepackage{bbold}

\usepackage{comment}
\usepackage{graphicx}
\usepackage{mathtools}
\usepackage{multicol}
\usepackage{tikz, tikz-cd}
\usepackage{adjustbox}

\usepackage{aligned-overset}
\usepackage{enumerate}
\usepackage[ruled]{algorithm2e}
\usepackage{wrapfig}
\usepackage{placeins}

\usepackage{nicematrix}

\usepackage{caption, subcaption}
\newlength{\twosubht}
\newsavebox{\twosubbox}

\usepackage{natbib}
\let\OLDthebibliography\thebibliography
\renewcommand\thebibliography[1]{
  \OLDthebibliography{#1}
  \setlength{\parskip}{1pt}
  \setlength{\itemsep}{0pt plus 0.3ex}
}
\usepackage{multirow}

\title{Subdivision $k$-Form Spaces within the Finite Element Exterior Calculus Framework}
\author{Robert Piel\footnote{University of Surrey, School of Mathematics and Physics, UK (r.piel@surrey.ac.uk), ORCID: \url{https://orcid.org/0009-0001-8681-2493}} \ and
Werner Bauer\footnote{University of Surrey, School of Mathematics and Physics, UK (w.bauer@surrey.ac.uk), ORCID: \url{https://orcid.org/0000-0002-5040-4287}}
}

\newcommand{\green}{\color{ForestGreen}}

\newcommand{\T}{\mathcal{T}}
\newcommand{\TFE}{{\T_{\mathrm{FE}}}}
\newcommand{\K}{\mathcal{K}}
\newcommand{\V}{\mathcal{V}}
\newcommand{\E}{\mathcal{E}}
\newcommand{\F}{\mathcal{F}}

\newcommand{\simplexname}{\sigma}
\newcommand{\simplex}[3]{\freeQ{\simplexname}{#1}{#2}{#3}}

\newcommand{\simplexsubset}[2]{\freeVec{\Delta}{#1}{#2}}

\newcommand{\vertex}[2]{\freeQ{v}{}{#1}{#2}}
\newcommand{\vertexpos}[2]{\freeQ{\Vec{v}}{}{#1}{#2}}
\newcommand{\edge}[2]{\freeQ{e}{}{#1}{#2}}
\newcommand{\face}[2]{\freeQ{f}{}{#1}{#2}}

\newcommand{\KmKnl}[3]{\K^{#1}\K^{#2}_{#3}}
\newcommand{\VFl}[1]{\V\F_{#1}}

\newcommand{\FFl}[1]{\F\F_{#1}}

\newcommand{\FEl}[1]{\F\E_{#1}}

\newcommand{\FVl}[1]{\F\V_{#1}}

\newcommand{\FFltoL}[2]{\F\F_{#1 \to #2}}

\newcommand{\coarseIdxone}{{ i}}
\newcommand{\coarseIdxtwo}{{ m}}
\newcommand{\fineIdxone}{{ j}}
\newcommand{\fineIdxtwo}{{ n}}

\renewcommand{\l}{{\green\ell}}
\renewcommand{\L}{{\green L}}

\newcommand{\lzero}{{\green 0}}
\newcommand{\lone}{{\green {\l'}  }}
\newcommand{\ltwo}{{\green {\l''} }}

\newcommand{\nr}{{n_r}}
\newcommand{\ns}{{n_s}}
\newcommand{\nrcrit}{{n_r^\mathrm{crit}}}
\newcommand{\nscrit}{{n_s^\mathrm{crit}}}

\newcommand{\constrQ}[5]{ ( {#1}^{#2}_{#3 \vert #4} )_{#5}} 
\newcommand{\constrVec}[4]{ {#1}^{#2}_{#3 \vert #4} }
\newcommand{\freeQ}[4]{ ( {#1}^{#2}_{#3} )_{#4}}
\newcommand{\freeVec}[3]{ {#1}^{#2}_{#3}}

\newcommand{\Sop}[3]{\mathcal{S}^{#1}_{{#2} \to {#3}}}
\newcommand{\Smat}[3]{\mathbf{\mathrm{S}}^{#1}_{#2 \to #3}}
\newcommand{\Smatij}[5]{\left(\mathbf{\mathrm{S}}^{#1}_{#2 \to #3}\right)_{#4 #5}}

\newcommand{\Aop}[3]{\mathcal{A}^{#1}_{{#2} \to {#3}}}
\newcommand{\Amat}[3]{\mathbf{\mathrm{A}}^{#1}_{#2 \to #3}}
\newcommand{\Amatij}[5]{\left(\mathbf{\mathrm{A}}^{#1}_{#2 \to #3}\right)_{#4 #5}}

\newcommand{\primalAop}[4]{ {\mathcal{A}^{#1}_{{#2} \to {#3}|{#4} }}  }
\newcommand{\dualAop}[4]{ {\mathcal{B}^{#1}_{{#2}|\,{#3} \rightarrow {#4}} }}

\newcommand{\Whit}{\mathrm{Whit}}
\newcommand{\Loop}{\mathrm{Loop}}
\newcommand{\AopWhit}[2]{\Aop{\Whit}{#1}{#2}}
\newcommand{\AopLoop}[2]{\Aop{\Loop}{#1}{#2}}

\newcommand{\extd}{\mathbf{d}}
\newcommand{\auxextd}{\overset{\sim}{\vphantom{a}\smash{\mathbf{d}}}} 
\newcommand{\Dmat}[2]{\mathbf{D}^{#1}_{#2}} 
\newcommand{\Dmatij}[4]{\big(\mathbf{D}^{#1}_{#2}\big)_{#3 #4}} 

\newcommand{\FEECspace}[2]{\Lambda^{#1}_{#2}}
\newcommand{\FEECbasis}{\psi}
\newcommand{\FEECbases}{\Psi}
\newcommand{\FEECspaceZero}[2]{\mathring{\Lambda}^{#1}_{#2}}

\newcommand{\SDFspace}[3]{\mathcal{S}\Lambda^{#1}_{#2 \vert #3}}
\newcommand{\SDFbasis}{\phi}
\newcommand{\SDFbases}{\Phi}
\newcommand{\SDFbasesZero}{\mathring{\SDFbases}}
\newcommand{\SDFspaceZero}[3]{\mathring{\mathcal{S}}\Lambda^{#1}_{#2 \vert #3}}

\newcommand{\trace}{\mathrm{tr}}
\newcommand{\Vint}{\mathring{\V}}
\newcommand{\Eint}{\mathring{\E}}
\newcommand{\Fint}{\mathring{\F}}
\newcommand{\Kint}{\mathring{\K}^k}
\newcommand{\BCAop}[3]{\mathring{\mathcal{A}} {}^{#1}_{#2 \to #3}}
\newcommand{\BCAmat}[3]{\mathring{\mathrm{A}} {}^{#1}_{#2 \to #3}}
\newcommand{\BCBEop}[4]{\mathring{\mathcal{B}} {}^{#1}_{#2 \vert #3 \to #4}}

\renewcommand{\powerset}{\mathbb{P}}

\newcommand{\auxincl}{\overset{\sim}{\vphantom{a}\smash{\iota}}} 
\newcommand{\support}{\mathrm{supp}}
\newcommand{\pfrac}[2]{\frac{\partial #1}{\partial #2}}
\newcommand{\M}{\mathcal{M}}
\newcommand{\vvert}[1]{\left\vert #1 \right\vert}
\newcommand{\dimlk}[2]{\big\vert \K^{#2}_{#1} \big\vert}
\newcommand{\sumdimlk}[2]{\vert \K^{#2}_{#1} \vert}

\newcommand{\CG}{\mathrm{CG}}
\newcommand{\NED}{\mathrm{NED}}
\newcommand{\DG}{\mathrm{DG}}

\newcommand{\HcurlM}[1]{H\left(\text{curl}, #1\right)}
\newcommand{\HocurlM}[1]{\mathring{H}\left(\text{curl}, #1\right)}

\newtheorem{theorem}{Theorem}[section]
\newtheorem{definition}[theorem]{Definition}
\newtheorem{lemma}[theorem]{Lemma}
\newtheorem{proposition}[theorem]{Proposition}
\newtheorem{corollary}[theorem]{Corollary}

\newtheorem{remark}[theorem]{Remark}
\newtheorem{assumption}[theorem]{Assumption}
\newtheorem{numberedproof}[theorem]{Proof}

\numberwithin{equation}{section}

\begin{document}

\maketitle

\vspace{-2.0em}

\begin{abstract}
  This paper introduces discrete differential form spaces over two-dimensional manifold meshes that feature enhanced subdivision-induced inter-element regularity compared to conventional finite element (FE) spaces. This increase in smoothness is achieved by pulling back refined subdivision basis functions along a hierarchy of increasingly fine meshes that are generated by a subdivision algorithm. We introduce a framework that casts several known instances of $k$-form subdivision schemes in the language of FE and derive conditions under which the resulting subdivision-induced hierarchy of FE function spaces satisfies a discrete de Rham complex. The paper further illustrates the enforcing of zero boundary conditions by discarding basis functions close to the mesh boundary and shows that this does not compromise the de Rham complex.

  To analyse our novel subdivision $k$-form spaces we solve the Maxwell eigenvalue problem to confirm the absence of spurious modes and to study the accuracy of the computed eigenvalues. Recovering accurately the expected analytic eigenvalue spectrum shows that our novel subdivision $k$-form spaces indeed preserve the de Rham complex, since this test case is known to be challenging for methods not preserving this structure. Further, we numerically investigate the approximation errors of these subdivision spaces for given analytic functions.

  The presented study shows that our method can be employed in two ways. Upon a suitable choice of parameters, the subdivision $k$-form spaces are up to $1.5$ orders of magnitude more accurate in the $L^2$ norm than conventional lowest-order FE spaces with the same number of degrees of freedom. Alternatively, for a given target accuracy, the number of required degrees of freedom can be significantly reduced, resulting in a speed-up by a factor of up to 6 for the discussed test cases.

\end{abstract}

\tableofcontents

\vspace{-0.5em}

\section{Introduction}

Many foundational principles in physics are most naturally described through geometry, often via symmetries and conservation laws, such as conservation of energy or mass. To model these systems consistently, the corresponding partial differential equations (PDEs) should inherently reflect these geometric structures. To obtain consistent numerical solutions of these systems, it is crucial
to preserve these structures during the discretisation process of PDEs to accurately capture the physical behaviour of the underlying continuous systems. Such structure-preserving methods generally
enable long-term stability, prevention of a systematic drift in the stationary solutions, and consistency in the statistical properties, see e.g. \citep{Leimkuhler_Reich_2005,hairer2006,wan2018}.

A crucial structure to be preserved in many numerical simulations in order to guarantee consistency and stability of the approximated solution is the de Rham complex; a structure related to preserving the vector identities on the discrete level. The exterior calculus of differential forms provides a powerful framework to describe this structure, and significant research has been undertaken to devise approaches that discretise this framework and preserve its key features in a suitable way. For example, Discrete Exterior Calculus (DEC) \citep{Hirani.2003,Desbrun.2006} provides discrete analogs of $k$-forms and consistent derivative operators that form a discrete de Rham complex. The DEC idea allows one to derive, for example, consistent discretisations of the equations of geophysical fluid dynamics \citep{ThuburnDEC} that preserve energy over very long integration periods \citep{Ringler.2010}. The principles of DEC can also be applied to address various inconsistencies in existing operational forecasting models \citep{ChrisWernerDEC}. However, DEC relies on dual meshes, cochains and matrix representations rather than function spaces and operators, which can limit its flexibility in numerical analysis or certain applications.

Alternatively, Finite Element Exterior Calculus (FEEC), developed by \cite{Arnold.2006,Arnold.FEEC}, combines the FE method that is based upon the language of functional analysis with the tools from topology and cohomology to obtain stable and convergent FE discretisations. It constructs sequences of FE spaces that preserve the structure of the de Rham complex while 
featuring flexibility in handling complex geometries and arbitrary polynomial degrees. A notable drawback of standard FE methods is their restriction to globally $\mathcal{C}^0$ continuous basis functions, unless higher continuity is explicitly incorporated, as e.g. in the Argyris element \citep{Argyris.1968}. This lack of smoothness poses challenges when discretising, for example,  fourth-order PDEs like the biharmonic equation, or when working with complexes involving second-order differential operators such as the Hessian or the elasticity complex \citep{arnold_complexes_2021}.

This lack of smoothness can be addressed using Isogeometric Analysis (IGA) (see e.g. \citep{Hughes.2005_IGA}), a generalized FE discretisation method that constructs and applies smooth function spaces, namely spline spaces. Spline spaces feature higher smoothness at the cell interfaces than only $\mathcal{C}^0$.
In recent years, a wide variety of spline families have been added to the numerical toolbox of IGA to solve various PDEs.
Among them, B-splines are particularly popular. Introduced into a FE setting in \cite{Hoellig.2003}, they have been used to build discrete de Rham complexes from tensor-product B-spline spaces, as shown in \cite{Sonnendruecker.2011} and \cite{Buffa.2011_IGA_with_splines}. These spaces can achieve the maximum regularity of $\mathcal{C}^{p-1}$ for B-splines of degree $p$. However, traditional spline methods are typically limited to regular (quadrilateral) meshes.

Several generalizations for irregular mesh geometries (such as general polygonal meshes) have been proposed, for instance Powell-Sabin splines \citep{Powell.1977} or T-splines \citep{Sederberg.2003}. While these are effective for scalar functions or 0-forms, extending them to $k$-forms while preserving compatibility is non-trivial and remains so far an unsolved problem.

One promising way to achieve sequences of compatible spline function spaces that approximate $k$-forms on irregular geometries is the so-called \emph{subdivision method}. In this context, a subdivision algorithm constructs a hierarchy of nested meshes that allows for smooth spline function spaces on irregular meshes.
Subdivision splines reduce to classical splines on regular patches and are defined on irregular ones as the limit of an infinite sequence of classical spline patches obtained from a refinement process, see \cite{Peters.2008}. For instance, Loop subdivision \citep{Loop.1987} generalizes the $\mathcal{C}^2$ quartic box spline to arbitrary triangulations. On regular mesh patches, it reproduces the quartic box spline, but drops to $\mathcal{C}^1$ regularity exactly at irregular vertices. This limiting process, while enabling compatibility with arbitrary triangulations like in FEM, makes evaluation and integration of the resulting subdivision spline more challenging. Despite this, the increased smoothness offered by subdivision splines has proven valuable in areas where the smoothness of traditional FE is insufficient. Notably, \cite{Cirak.2000} and \cite{Burkhart.2010} applied subdivision splines to problems such as plate bending (using Loop subdivision) and linear elasticity (using Catmull-Clark subdivision). Building on these, \cite{IGA_to_do_list} surveyed key open questions in the field of subdivision-based IGA, particularly its extension to three-dimensional domains.

Unlike closed-form spline generalizations such as Powell-Sabin or T-splines, \cite{Wang.2006} and \cite{Wang.2008} demonstrated a path to generalizing subdivision to $k$-forms while maintaining compatibility in a DEC sense (i.e. preserving a discrete de Rham complex). The authors enforced a commutative relationship between subdivision and differentiation to construct a 1-form subdivision scheme from given 0- and 2-form schemes, thereby producing spaces that satisfy a discrete de Rham complex. Later, \cite{Goes.2016} explored the potential of these spaces for IGA. To circumvent the challenge of evaluating the smooth limit functions, they approximated the limit by applying a few iterations of the subdivision algorithm and then using Whitney interpolants. While these approximations are only $\mathcal{C}^0$, they retain a certain subdivision-induced inter-element regularity and thus share some important properties of the genuinely smooth limit functions like increased convergence rates under mesh refinement.

In this work, we build on the foundational ideas of \cite{Wang.2006} to develop a FE framework for subdivision-based $k$-form spaces. The idea is to develop a subdivision FE method for irregular, here arbitrary triangular meshes, that features subdivision-induced smoothness and allows us to study the method's convergence and approximation properties in terms of FE analysis. This has not been done so far in the literature; either subdivision has been done for DEC methods or FE-like methods have only been done for regular quadrilateral meshes using tensor-product splines.

The main contribution of this work is a FEEC-compatible reinterpretation of subdivision $k$-form constructions, realizing them as subspaces of standard FE spaces on refined simplicial meshes. To this end, we introduce explicit subdivision operators that enable the refinement and smoothing of discrete differential forms and show that these operators commute with the exterior derivative, thereby preserving the de Rham complex structure across all levels. The resulting hierarchy of approximation spaces admits an explicit parametrization in terms of refinement, corresponding to the introduction of additional degrees of freedom, and smoothing, in the sense of increasing a certain subdivision-induced regularity across element interfaces.

The remainder of the paper is organized as follows. In Section~\ref{sec_subdiv}, we briefly review the FE exterior calculus framework and introduce the subdivision algorithm to construct mesh hierarchies. Section~\ref{sec:subdiv_k_forms} generalises the description of the subdivision algorithm to discrete differential $k$-forms defined on the mesh hierarchy. Based on that, the section defines the subdivision $k$-form spaces and establishes their structural properties, including nestedness, commutativity with exterior differentiation, and preservation of the de Rham complex. We also outline the notion of subdivision-induced regularity. In Section~\ref{sec:implementation}, we discuss implementation aspects, including the construction of FE matrices and the approximation of the physical domain using the mesh hierarchy. Section~\ref{sec:numerics} presents numerical experiments, including convergence studies for projection operators, an investigation of the effect of subdivision-induced regularity on approximation performance, and simulations of the Maxwell eigenvalue problem on a simply connected domain as a structure-sensitive test case. Finally, Section~\ref{sec_conclusion} concludes with a discussion of the results and outlines directions for future work, including hierarchical and adaptive extensions of the proposed framework.

\section{Background: FEEC and the subdivision algorithm}
\label{sec_subdiv}

This section establishes the necessary theoretical background of this work. First, we review some standard notions of Finite Element Exterior Calculus (FEEC). After that, we introduce the notation used in the context of meshes and simplices and their adjacency used throughout this work. This allows us to introduce a model of the subdivision algorithm that enables us to build a mesh hierarchy in a way that is compatible with the FE setting. The formalisation of the subdivision algorithm will provide the foundation upon which we construct the subdivision $k$-form spaces in Section~\ref{sec:subdiv_k_forms}.

\subsection{Finite Element Exterior Calculus}

Compatible finite elements consist of a sequence of finite element spaces that approximate differential forms representing the variables of a PDE at hand and that, in turn, comprise a de Rham complex. As shown by \cite{Arnold.2006,Arnold.FEEC} such Finite Element Exterior Calculus (FEEC) discretisations guarantee convergence and stability of the discretised PDE. In this section, we will recall the definition of differential forms to then discuss the de Rham complex and its significance. Afterwards, we highlight finite-dimensional approximation spaces for differential forms that comprise discrete de Rham complexes.

Let $\M$ be a manifold and let $V$ denote the space of tangent vector fields of $\M$. This work will consider differential $k$-forms (denoted henceforth simply as $k$-forms) $\omega^k$ to be multi-linear, anti-symmetric, covariant tensors of the form
\begin{equation}
    \begin{aligned}
        \omega^k: \;\; \underbrace{V \times \hdots \times V}_{k \; \text{times}} \to \mathbb{R}.
    \end{aligned}
\end{equation}
The sets of $k$-forms over $\M$ constitute vector spaces under point-wise vector addition and scalar multiplication. We denote these spaces by $\Omega^k\big(\M\big)$. We can invoke the metric tensor of $\M$ to obtain well-known objects from vector calculus. For our case of $2$-dimensional manifolds, $0$-, $1$- and $2$-forms correspond to scalar functions, vector fields and density fields, respectively. For a discussion on the relation of differential forms and vector fields, refer to standard textbooks on differential geometry or consult \cite{Werner2016} for a concise overview.

The differential operator acting on $k$-forms is called the exterior derivative $\extd^k: \Omega^k\big(\M\big) \to \Omega^{k+1}\big(\M\big)$. Being a derivative, $\extd^k$ is a linear operator that satisfies the Leibniz rule. Additionally, $\extd^k$ is required to preserve the anti-symmetry of $k$-forms and to satisfy the relation $\extd^{k+1} \circ \extd^{k} = 0$. This implies that the $k$-form spaces $\Omega^k\big(\M\big)$ comprise a differential complex, i.e.
\begin{equation}
    \label{eq:de_rham_complex}
    \begin{tikzcd}
        0 \arrow[r] & \Omega^0\big(\M\big) \arrow[r, "\extd^0"] & \Omega^1\big(\M\big) \arrow[r, "\extd^1"] & \Omega^2\big(\M\big) \arrow[r] & 0.
    \end{tikzcd}
\end{equation}

This particular complex is valid for two-dimensional manifolds $\M$. The vector calculus equivalents of $\extd^0$ and $\extd^1$ are the gradient $\nabla$ and the rotation operator $\nabla \times$, respectively. The latter is a two-dimensional analog of the curl operator which is given by
\begin{equation}
    \nabla \times \begin{bmatrix} u_1 \\ u_2 \end{bmatrix} = \pfrac{u_2}{x^1} - \pfrac{u_1}{x^2}.
\end{equation}
Thus, $\extd^1 \extd^0 f = 0$ for $f \in \Omega^0\big(\M\big)$ recovers the well-known identity $\nabla \times \nabla f = 0$ for functions $f$.
Note that $\extd^{k+1} \circ \extd^{k} = 0$ is equivalent to $\mathrm{im}\big( \extd^{k-1} \big) \subset \mathrm{ker}\big( \extd^{k} \big)$, where $\mathrm{im}$ denotes the image and $\mathrm{ker}$ the kernel of the exterior derivative operator.

A differential complex is called a \textit{de Rham} complex iff the dimensions of the cohomology groups $\mathcal{H}^k \coloneq \mathrm{ker}(\extd^{k})/\mathrm{im}( \extd^{k-1})$ are identical to the $k^{\mathrm{th}}$ Betti number of $\M$. This ensures that the discrete $k$-form spaces faithfully capture the harmonic forms and that the images and kernels of the discrete differential operators are well-behaved.
Extensive research has been undertaken to discretise the exterior calculus of $k$-forms, culminating in the framework of FEEC \citep{Arnold.FEEC}, for example.

Finite element methods usually seek to find weak solutions in Sobolev spaces. In this work, we consider, for example,
$H^1(\M)$, $\HcurlM{\M}$ and~$L^2(\M)$ defined as
\begin{align}
    L^2(\M) &= \big\{ f: \M \to \mathbb{R} \;\; \big\vert \;\; \vert\vert f \vert\vert^2 = \int_\M f^2 \; \mathrm{d}x = (f, f)_\M \; < \infty \big\}, \\
    \HcurlM{\M} &= \big\{ u \in \big[ L^2(\M) \big]^2 \;\; \big\vert \;\; \nabla \times u \in L^2(\M) \big\}, \\
    H^1(\M) &= \big\{ f \in L^2(\M) \;\; \big\vert \;\; \nabla f \in \big[ L^2(\M) \big]^2 \big\}.
\end{align}
These (still infinite dimensional) spaces have less regularity than the ones used in the smooth complex in Eq.~\eqref{eq:de_rham_complex}. The Sobolev spaces comprise the de Rham complex
\begin{equation}
    \label{eq:infinite_dim_de_Rham}
    \begin{tikzcd}[row sep=small, column sep=normal]
        0 \arrow[r] & H^1(\M) \arrow[r, "\extd", "\nabla"'] & \HcurlM{\M} \arrow[r, "\extd", "\nabla \times"'] & L^2(\M) \arrow[r] & 0.
    \end{tikzcd}
\end{equation}
FEEC also constructs families of compatible FE spaces that mimic the complex in Eq.~\eqref{eq:infinite_dim_de_Rham}.
Namely, by approximating these Sobolev spaces by finite dimensional FE spaces denoted as
$\FEECspace{0}{} \subset H^1(\M)$, $\FEECspace{1}{} \subset \HcurlM{\M}$ and $\FEECspace{2}{} \subset L^2(\M)$,
we obtain (discrete) FEEC de Rham complexes as
\begin{equation}
\label{eq:FEEC_de_Rham}
\begin{tikzcd}
    0 \arrow[r] & \FEECspace{0}{} \arrow[r, "\extd^0"] & \FEECspace{1}{} \arrow[r, "\extd^1"] & \FEECspace{2}{} \arrow[r] & 0.
\end{tikzcd}
\end{equation}
In the following, we will often omit the superscript $k$ of $\extd^k$ and simply denote it by $\extd$ if its domain becomes clear from the context. 
For later use, we denote the set of basis functions of $\FEECspace{k}{}$ by $\freeVec{\FEECbases}{k}{}$, i.e.
\begin{equation}
    \label{eq:FEECbases_definition}
    \FEECspace{k}{} = \mathrm{span}\big( \freeVec{\FEECbases}{k}{} \big)
\end{equation}
and a single basis function by $\freeQ{\FEECbasis}{k}{}{\coarseIdxone} \in \freeVec{\FEECbases}{k}{}$.
Further, the dimensions of the spaces $\FEECspace{k}{}$ are denoted by $\mathrm{dim}\big(\FEECspace{k}{}\big)$.
Our construction of the subdivision $k$-form spaces will build upon the lowest-order versions of the spaces $\FEECspace{k}{}$, see Eq.~\eqref{eq:choice_of_feec_space}. This is because we will see that subdivision naturally operates on the lowest order
two-dimensional complex with N\'ed\'elec spaces ($\NED_1$) as the $1$-form space.

The spaces $\FEECspace{k}{}$ will be used in Section~\ref{sec:subdiv_k_forms} to construct the subdivision $k$-form spaces. The de Rham complex structure in Eq.~\eqref{eq:FEEC_de_Rham} will carry over to the subdivision $k$-form spaces through commutativity of subdivision operators and derivatives.

\subsection{Meshes, simplices and adjacency operators}

In the context of this work, we think about subdivision schemes as a map between function spaces on a mesh hierarchy. That is, starting from an initial triangular mesh
$\T_{0}$, the finer meshes $\T_{\l}$ on integer levels $\l \in \{ 0, 1, ..., \L \}$ are iteratively generated via a so-called subdivision algorithm until we reach a finest mesh $\T_{\L}$. Before we can properly introduce subdivision algorithms, we need to establish the following notation and operators.

The triangulation of the domain $\M$ on level $\l$ leaves us with a set of $k$-simplices $\K^k$ for $k \in \{ 0,1,2 \}$
of the mesh $\T_\l$. In general, we denote a $k$-simplex
by $\simplex{k}{\l}{\coarseIdxone} \in \K^k\big(\T_\l\big)$,
where the notation indicates that $\simplex{k}{\l}{\coarseIdxone}$ is the $i^{\text{th}}$ $k$-simplex at mesh level $\l$. For example, we will deal with the following sets of mesh entities of the mesh $\T_{\l}$ at level $\l$:
\begin{align}
    \K^0(\T_{\l}) &= \V\big(\T_{\l}\big) \;\; \text{is the set of all vertices $\vertex{\l}{\coarseIdxone}$ of the mesh } \T_{\l}, \\
    \K^1(\T_{\l}) &= \E\big(\T_{\l}\big) \;\; \text{is the set of all edges $\edge{\l}{\coarseIdxone}$ of the mesh } \T_{\l}, \\
    \K^2(\T_{\l}) &= \F\big(\T_{\l}\big) \;\; \text{is the set of all faces $\face{\l}{\coarseIdxone}$ of the mesh } \T_{\l}.
\end{align}
We typically use the short-hand notation $\K^k_{\l} = \K^k \big(\T_{\l}\big)$ and denote the cardinality of a set by $\vert \cdot \vert$. For example, we write $\V_{\l} = \V\big(\T_{\l}\big)$ for the set of mesh vertices on level $\l$
where $ \vert \V_{\l} \vert$ is the total number of vertices on this mesh level.

The adjacency relations between simplices are denoted consistently with the naming conventions in \cite{Ringler.2010}, i.e. the operators will be denoted by
\begin{equation}
    \KmKnl{p}{q}{\l}: \K^q_{\l} \to \powerset\big(\K^p_{\l}\big)
\end{equation}
with $p,q \in \{0,1,2\}$, meaning that they map $n$-simplices to their adjacent $m$-simplices.
Here, we introduced the power set $\powerset\big(U\big)$ of a set $U$ as $\powerset\big(U\big) \coloneqq \big\{ V\colon \; V \subset U \big\}$.
Some concrete realizations of the adjacency operator at mesh level $\l$ are the following relations
for an arbitrary index $\coarseIdxone$:
\begin{alignat}{2}
    \FFl{\l} &: \F_{\l} \to \powerset\big(\F_{\l}\big), \quad && \face{\l}{\coarseIdxone} \mapsto \{ \face{\l}{\coarseIdxone} \} \\
    \VFl{\l} &: \F_{\l} \to \powerset\big(\V_{\l}\big), \quad && \face{\l}{\coarseIdxone} \mapsto \{ \text{all vertices $\vertex{\l}{\coarseIdxone_\coarseIdxtwo}$ that constitute the face } \face{\l}{\coarseIdxone} \}, \\
    \FEl{\l} &: \E_{\l} \to \powerset\big(\F_{\l}\big), \quad && \edge{\l}{\coarseIdxone_\coarseIdxtwo} \mapsto \{ \text{all faces $\face{\l}{\coarseIdxone_\coarseIdxtwo}$ that contain the edge } \edge{\l}{\coarseIdxone} \}.
\end{alignat}
The notation $\coarseIdxone_\coarseIdxtwo$ indicates that the adjacency operators return sets of $p$-simplices (indexed by $\coarseIdxtwo$) adjacent to the $q$-simplex $\simplex{q}{\l}{\coarseIdxone}$, see Section~\ref{sec_subdivalg_formal} for more information.

The previous adjacency operators accept only single simplices as their arguments. For example, we can map one face to its three vertices. However, we will need them to act on sets of simplices too, allowing us, for example, to map a set of faces to the union of all of their vertices.
To formalise this idea, we extend the above definition of the adjacency operators as follows:
\begin{equation}
    \label{eq:adjacency_for_sets}
    \KmKnl{p}{q}{\l}: \; \powerset\big(\K^q_{\l}\big) \to \powerset\big(\K^p_{\l}\big), \quad
    \simplexsubset{q}{\l} \mapsto \bigcup_{\simplex{q}{\l}{\coarseIdxone} \in \simplexsubset{q}{\l} } \KmKnl{p}{q}{\l} \big( \simplex{q}{\l}{\coarseIdxone} \big).
\end{equation}
Here, the argument of $\KmKnl{p}{q}{\l}$ is a set $\simplexsubset{q}{\l}$ of simplices and the operator
returns the union of the individual contribution of every simplex in the set. The notation here is slightly
abusive as the notations of these technically slightly different adjacency operators coincide.
However, we may drop the powerset notation from now on when the context makes it clear which one we are applying.

So far, these operations are restricted to one individual mesh $\T_{\l}$ of the mesh hierarchy with no mapping across meshes on different levels. Further below we will introduce analogous operators that extend these definitions to the whole mesh hierarchy, i.e. they allow for mappings between different mesh levels $\l$.

\subsection{Subdivision algorithm and mesh subdivision}
\label{sec_subdivalg}

In this subsection, we first give a general description of a subdivision algorithm. This algorithm will be crucial for this work as it provides a means to construct a hierarchy of nested meshes and to move between different refinement levels across this mesh hierarchy. Then,
we formulate the Loop subdivision algorithm \citep{Loop.1987} applied in this work
in a way that makes it suitable for the use in FE methods.

\subsubsection{General description of subdivision}

A subdivision algorithm maps a coarse mesh to a fine mesh by composing two ideas: topological mesh refinement and averaging of the mesh geometry. What these steps look like concretely depends on the chosen subdivision scheme. For the Loop subdivision scheme \citep{Loop.1987} used in this work, the topological refinement splits every coarse triangle into four fine triangles. It keeps the three coarse vertices and adds three fine vertices, one on each formerly coarse edge, see Figure~\ref{fig:subdiv_algorithm}. The formerly coarse vertices are usually called even vertices and the newly added vertices are called odd vertices.

After the topological refinement, the geometric averaging step computes new positions for the fine vertices using a subdivision rule, usually given as a weighted average of select coarse vertices. Both the choice of coarse vertices and their averaging weights depend on the chosen subdivision scheme, the number of adjacent coarse vertices and how close to the boundary the fine vertex is. This leads to new positions for the fine vertices, see Figure~\ref{fig:subdiv_algorithm}.

Traditionally, subdivision algorithms are used to obtain more smooth-looking approximations of surfaces in computer graphics. That means that the finer mesh is the desired end product of the subdivision algorithm and subdivision schemes are designed such that certain properties are guaranteed for the (formal) limit surface $\T_{\infty}$, usually regarding the smoothness.

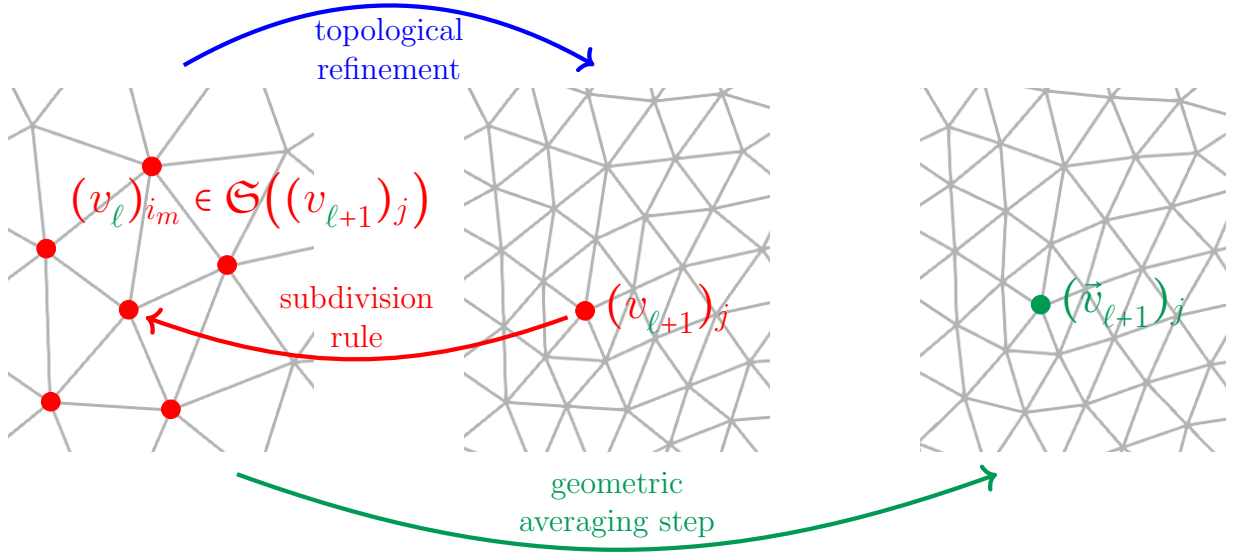
\begin{figure}[t!]
    \centering

    \begin{tikzpicture}
    \node[inner sep=0pt] (initial_mesh) at (0,0)
         [opacity=0.3]{\includegraphics[trim=15cm 9.6cm 7.5cm 10.8cm, clip, width=.25\textwidth]{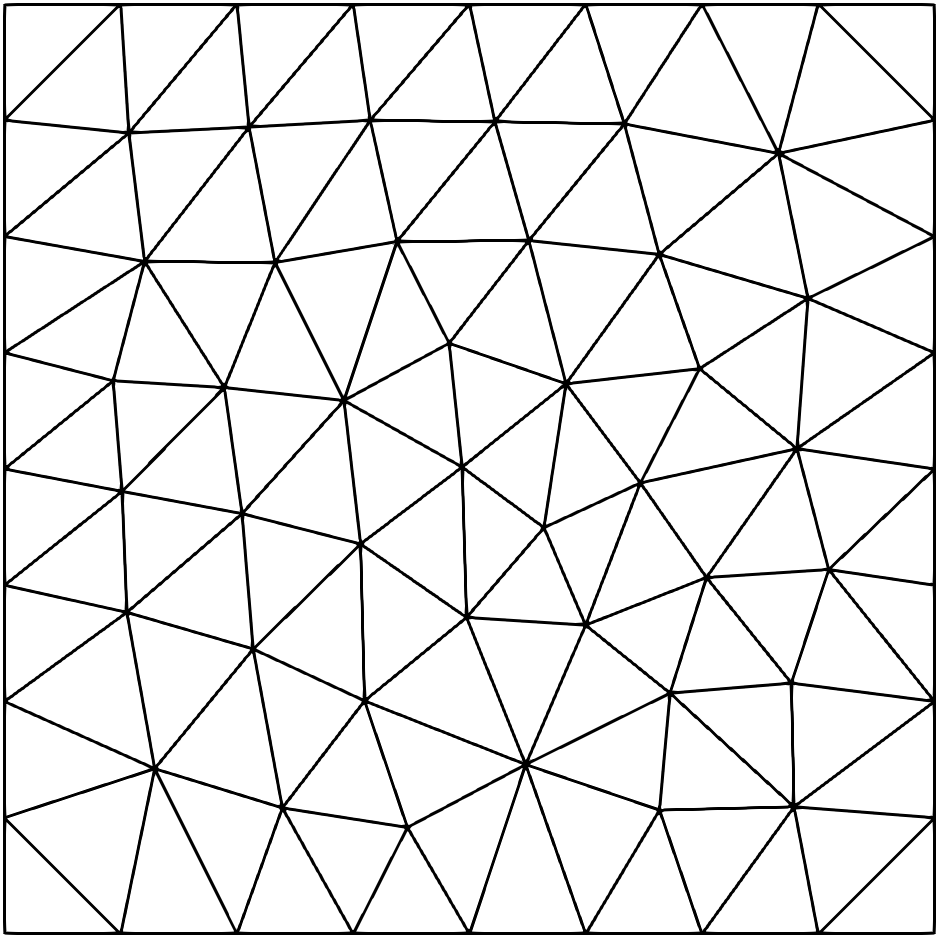}};
    \node[inner sep=0pt] (topological_ref) at (6,0)
         [opacity=0.3]{\includegraphics[trim=15cm 9.6cm 7.5cm 10.8cm, clip, width=.25\textwidth]{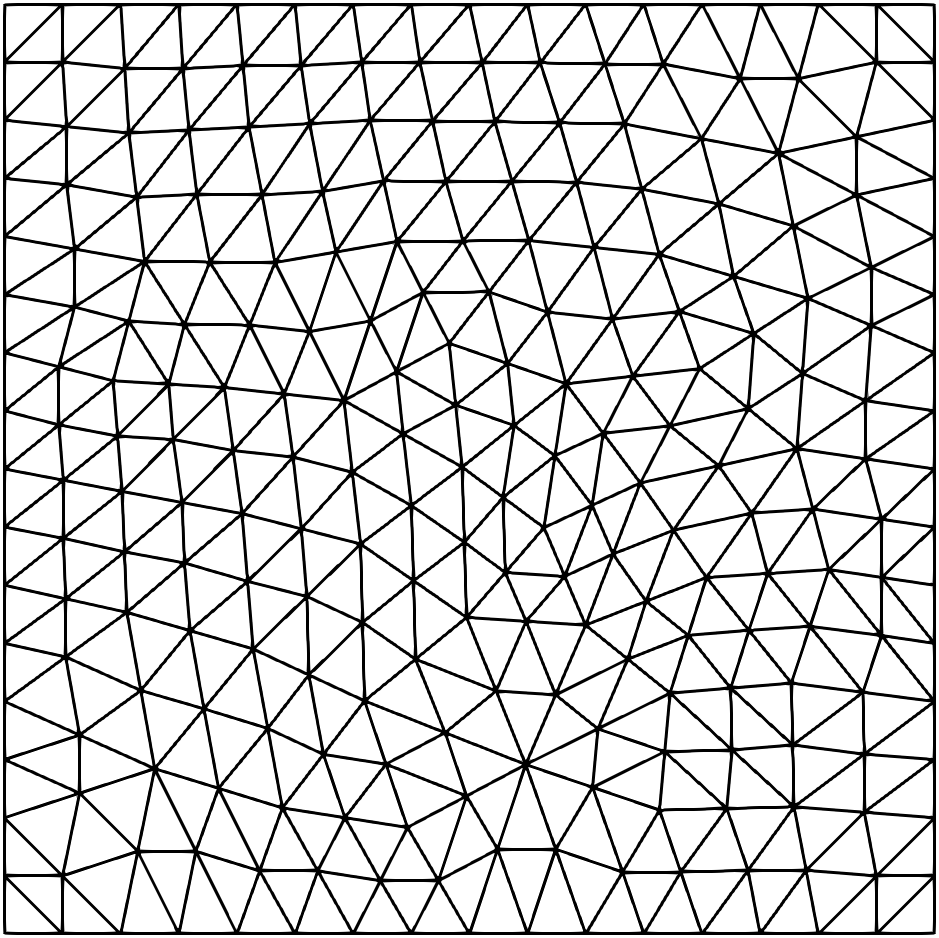}};
    \node[inner sep=0pt] (geometrical_ref) at (12,0)
         [opacity=0.3]{\includegraphics[trim=15cm 9.6cm 7.5cm 10.8cm, clip, width=.25\textwidth]{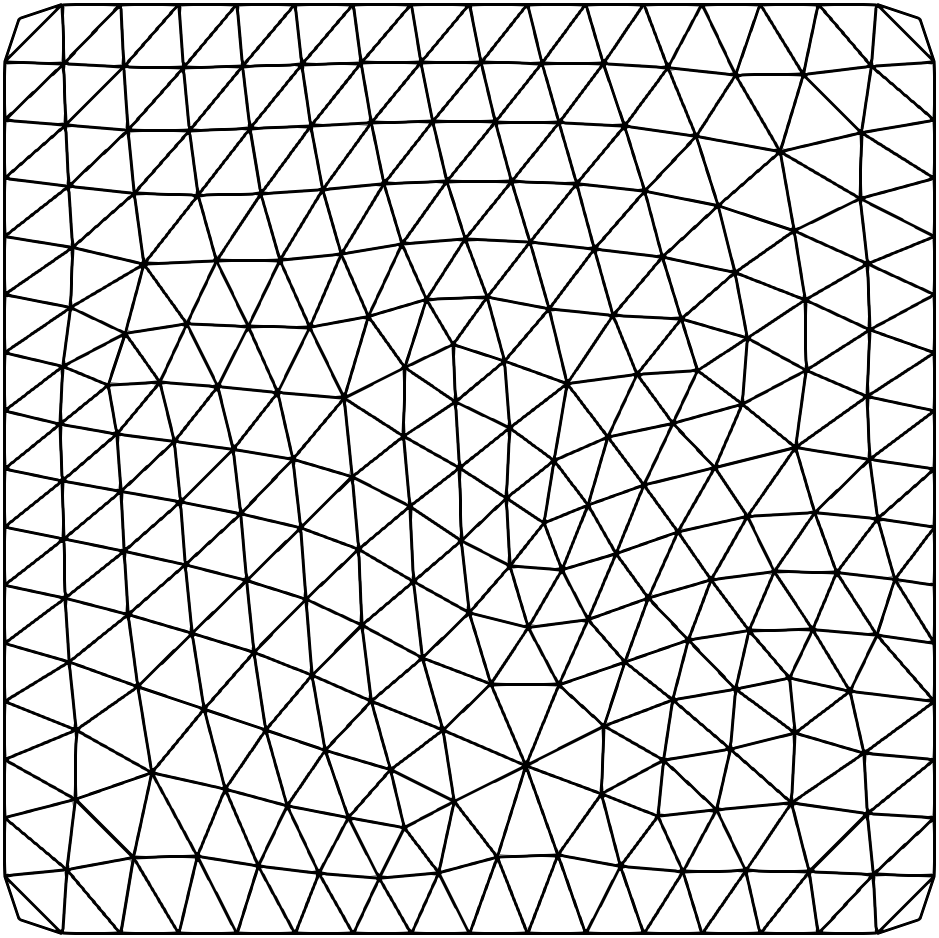}};

    \draw[->, ultra thick, color=blue] ([xshift=0.3cm, yshift=0.3cm] initial_mesh.north) to [bend left = 30] node[below, color=blue, label={[align=center, below]topological\\refinement}] {} ([xshift=-0.3cm, yshift=0.3cm] topological_ref.north);

    \draw[->, ultra thick, color=red] (5.35, -0.63) to [bend left = 20] node[below, color=red, label={[align=center]subdivision\\rule}] {} (-0.2, -0.63);

    \node[circle, draw=red, fill=red, inner sep=0pt, minimum size=7pt, label={[yshift=-5mm,font={\Large \bfseries},text=red]30:
    $\vertex{\l+1}{\fineIdxone}$}] (vlia) at (5.58, -0.545) {};

    \node[circle, draw=red, fill=red, inner sep=0pt, minimum size=7pt] at (-0.425, -0.53) {};
    \node[circle, draw=red, fill=red, inner sep=0pt, minimum size=7pt] at (-0.12, 1.37) {};
    \node[circle, draw=red, fill=red, inner sep=0pt, minimum size=7pt] at (-1.51, 0.28) {};
    \node[circle, draw=red, fill=red, inner sep=0pt, minimum size=7pt] at (-1.45, -1.75) {};
    \node[circle, draw=red, fill=red, inner sep=0pt, minimum size=7pt] at (0.13, -1.85) {};
    \node[circle, draw=red, fill=red, inner sep=0pt, minimum size=7pt] at (0.87, 0.063) {};

    \node[label={[yshift=-5mm,font={\Large \bfseries},text=red]30:$\vertex{\l}{\coarseIdxone_\coarseIdxtwo} \in \mathfrak{S}\big( \vertex{\l+1}{\fineIdxone}\big)$}] at (-1.5, 0.8) {};

    \draw[->, ultra thick, color=ForestGreen] ([xshift=1.0cm, yshift=-0.3cm] initial_mesh.south) to [bend right = 20] node[below, color=ForestGreen, label={[align=center]geometric\\averaging step}] {} ([xshift=-1.0cm, yshift=-0.3cm] geometrical_ref.south);

    \node[circle, draw=ForestGreen, fill=ForestGreen, inner sep=0pt, minimum size=7pt, label={[yshift=-5mm,font={\Large \bfseries},text=ForestGreen]30:$\vertexpos{\l+1}{\fineIdxone}$}] (vlia_pos) at (11.575, -0.463) {};

    \end{tikzpicture}

    \caption{The Loop subdivision algorithm maps a coarse mesh (left) to a fine mesh (right). First, the topology is refined by quadrisecting all faces (centre). The subdivision rule then maps any fine vertex to a set of coarse vertices and assigns a weight to each of them. Finally, the geometric averaging step computes the positions of the fine vertices as a weighted average of the positions of the coarse vertices.}
    \label{fig:subdiv_algorithm}
\end{figure}

\subsubsection{Formal description of a subdivision algorithm}
\label{sec_subdivalg_formal}

Next, we give a detailed introduction to the mesh subdivision algorithm and fix some important nomenclature and notation. The implications of the mesh subdivision schemes on differential $k$-forms defined on the mesh hierarchy will be discussed after we have laid these foundations.

Regarding the notation, we denote simplices on the coarse mesh level $\l$ with an index $\coarseIdxone$ and
on the next finer level $\l +1 $ with $\fineIdxone$. Further, when using $\fineIdxone_\fineIdxtwo$ we
refer to a set (over the subindex $\fineIdxtwo$) of simplices on the next finer mesh level $\l +1$
that are related to the coarse-mesh simplex $\coarseIdxone$ via, for example, the face splitting operation
introduced below in Eq.~\eqref{facesplitting}.
In turn, when using $\coarseIdxone_{\coarseIdxtwo}$ we
refer to a set (over the subindex ${\coarseIdxtwo}$) of simplices on the coarse mesh level $\l $
that are related to the fine-mesh simplices $\fineIdxone$ via, for example, the subdivision rule of Definition~\ref{def:subdivrule}
(see also Figure~\ref{fig:subdiv_algorithm}).

In the case of Loop subdivision, the topological refinement step quadrisects every face as depicted in Figure~\ref{fig:subdiv_algorithm}. The following definition provides a formal description of the face splitting step.

\begin{definition}[Topological refinement step]
  \label{def:refinement_step}
  The topological refinement step splits a coarse face $\face{\l}{\coarseIdxone}$ of a mesh $\T_\l$ into finer faces $\face{\l+1}{\fineIdxone_\fineIdxtwo}, \fineIdxtwo \in \{1,\dots,N\}$ that are part of $\T_{\l+1}$. The face splitting operator $\FFltoL{\l}{\l+1}$ formalises this relationship by
\begin{equation}\label{facesplitting}
    \FFltoL{\l}{\l+1} \big(\face{\l}{\coarseIdxone}\big) = \bigcup_{\fineIdxtwo = 1}^{N} \big\{ \face{\l+1}{\fineIdxone_\fineIdxtwo} \big\} \subset \F\big(\T_{\l+1}\big).
\end{equation}
\end{definition}

For example, Loop subdivision~\citep{Loop.1987} splits every coarse face $\face{\l}{\coarseIdxone}$ into $N=4$ fine faces, while double-iterations of $\sqrt{3}$ subdivision~\citep{Kobbelt.2000} would split it into $N=9$ fine faces.

This splitting operation can be generalised to act on arbitrary sets of faces on level $\l$.
Denoting these sets by $\simplexsubset{2}{\l} \in \F_{\l}$, we will write
\begin{equation}
    \label{facesplitting_union}
    \FFltoL{\l}{\l+1} \big( \simplexsubset{2}{\l} \big) \coloneqq \bigcup_{\face{\l}{\coarseIdxone} \in \simplexsubset{2}{\l}} \FFltoL{\l}{\l+1} \big(\face{\l}{\coarseIdxone}\big),
\end{equation}
where the union is over all faces on level $\l$ that are part of the set $\simplexsubset{2}{\l}$ similarly to Eq.~\eqref{eq:adjacency_for_sets}.

Next, the positions of the vertices of the resulting fine faces are then determined by the geometrical averaging step. The following definitions of subdivision rules and subdivision schemes provide the formal background for this geometric averaging procedure.

\begin{definition}[Subdivision rule]\label{def:subdivrule}
    Let $\T_{\l}$ and $\T_{\l +1}$ be a coarse and a fine mesh, respectively. A subdivision rule $\mathfrak{S}$ is a map that assigns to each fine vertex $\vertex{\l+1}{\fineIdxone} \in \V_{\l +1}$ a set of coarse vertices $\vertex{\l}{\coarseIdxone_\coarseIdxtwo}$ and associated weights
    $(w_\l)_{\coarseIdxone_\coarseIdxtwo}$ for $\coarseIdxtwo = 1, \dots, M$ by
    \begin{equation}
    \begin{split}
        \mathfrak{S}: \; \V_{\l+1} \to \big[ \V_\l \times \mathbb{R} \big]^M, \qquad \vertex{\l+1}{\fineIdxone} \mapsto \Big\{ \big( \vertex{\l}{\coarseIdxone_1}, \; (w_\l)_{\coarseIdxone_1} \big), \hdots, \big( \vertex{\l}{\coarseIdxone_M}, \; (w_\l)_{\coarseIdxone_M} \big) \Big\}.
    \end{split}
    \end{equation}
    Note that the number $M$ of involved coarse vertices and their respective weights depend on the chosen subdivision scheme
    and the properties of the fine vertex $\vertex{\l+1}{\fineIdxone}$.

    Denoting by $\vertexpos{\l}{\coarseIdxone_\coarseIdxtwo}$ the positions of the $\coarseIdxone_\coarseIdxtwo^{th}$ coarse vertices (related to the fine mesh index $\fineIdxone$), the position $\vertexpos{\l+1}{\fineIdxone}$ of the fine vertex $\vertex{\l+1}{\fineIdxone}$ is given by
    \begin{equation}
        \label{eq:subdiv_rule_def}
        \vertexpos{\l+1}{\fineIdxone} = \sum_{\coarseIdxtwo = 1}^M \; (w_\l)_{\coarseIdxone_\coarseIdxtwo} \cdot \vertexpos{\l}{\coarseIdxone_\coarseIdxtwo} \,.
    \end{equation}
    This step is typically referred to as \emph{geometric averaging}.
\end{definition}

The topological refinement operations and subdivision rules define relationships between individual coarse and fine faces and vertices, respectively. Collecting these across the entire mesh gives rise to a \emph{subdivision scheme}.

\begin{definition}[Subdivision scheme]\label{def:subdivscheme}
    A subdivision scheme consists of a topological refinement operation (see Definition~\ref{def:refinement_step}) that is consistently applied across all coarse faces $\face{\l}{\coarseIdxone}$, and a collection of subdivision rules (see Definition~\ref{def:subdivrule}), one for every fine vertex $\vertex{\l+1}{\fineIdxone}$ in $\V_{\l+1}$.
\end{definition}

Typically, subdivision rules differ qualitatively between even and odd vertices, vertices on the domain boundary and interior vertices, and between vertices with different valence, i.e. the number of adjacent vertices. Hence, most subdivision schemes have several classes of subdivision rules to account for these differences in the vertices.

Various subdivision schemes have been devised throughout the history of the field. In this work, we focus on Loop's subdivision scheme \citep{Loop.1987} and its associated subdivision algorithm.
Nevertheless, the formal framework we are going to introduce here applies to a wide range of subdivision schemes, see Remark~\ref{remark_othersubdivschemes}.
The following section focuses on the construction of mesh hierarchies using mesh subdivision operators.

\subsubsection{Construction of mesh hierarchies using subdivision}
\label{sec_subdivalg_hierarchy}

The previous Definitions~\ref{def:subdivrule} and~\ref{def:subdivscheme} establish the subdivision algorithm that allows us to construct a finer mesh from a coarser mesh. Using this algorithm repeatedly leads to a hierarchy of increasingly fine meshes, beginning from an initial mesh at level $\lzero$ to the finest mesh at level $\L$. Formally, the mesh hierarchy can be constructed through the subdivision operator introduced in the following.

\begin{definition}[Mesh subdivision operator]\label{def:mesh_subdiv_operator}
Given a mesh $\T_{\l}$ on level $\l$, applying the subdivision scheme of Definition~\ref{def:subdivscheme},
i.e. carrying out topological refinement and geometric averaging once according to
Definition~\ref{def:subdivrule} (see also Figure~\ref{fig:subdiv_algorithm}),
results in a finer mesh $\T_{\l+1}$ on level $\l+1$.
We denote this mapping as \emph{mesh subdivision operator} $\Sop{}{\l}{\l+1}$ and write formally
    \begin{equation}
        \T_{\l +1} \coloneqq \Sop{}{\l}{\l+1} \T_{\l}.
    \end{equation}

Further, given an initial mesh $\T_{\lzero}$ on level $\lzero$ and a maximum level $\L$, repeatedly applying the subdivision algorithm yields the mesh $\T_{L} = \Sop{}{\L-1}{\L} \; \dots \; \Sop{}{\lzero}{1} \T_{\lzero}$, which motivates the introduction of the \emph{accumulated mesh subdivision operator}
\begin{equation}
    \Aop{}{\lzero}{\L} \coloneqq \Sop{}{\L-1}{\L} \circ \hdots \circ \Sop{}{\lzero}{1}.
\end{equation}
Finally, we adopt the convention that $\Aop{}{\l}{\l} = \mathrm{Id}$ is the identity map for any $\l$ of the mesh hierarchy.
\end{definition}

Note that the mesh subdivision operators can also be accumulated between any two intermediate levels $\lone,\ltwo$ with $\lzero \leq \lone \leq \ltwo \leq \L$ to obtain $\Aop{}{\lone}{\ltwo}$.
As $\Aop{}{\lone}{\lone+1} \equiv \Sop{}{\lone}{\lone+1}$ for all levels, we will only use the accumulated subdivision operators in the following, even when they map between \emph{neighbouring levels} $\lone$ and $\lone\!+\!1$.

In the same way, the face splitting operator $\FFltoL{\l}{\l+1}$ of Eq.~\eqref{facesplitting}
can also be extended to act across multiple levels, i.e.
\begin{equation}\label{facesplitting_levels}
    \FFltoL{\l}{\L} \coloneqq \FFltoL{\L-1}{\L} \circ \hdots \circ \FFltoL{\l+1}{\l+2} \circ \FFltoL{\l}{\l+1}.
\end{equation}

The framework we have developed so far to capture the mechanics of the subdivision algorithm is rather generic. So far, any primal subdivision schemes, i.e., ones that use a face-split like Loop's scheme~\citep{Loop.1987}, can be described in this way. Other popular primal schemes that can be described in the framework above include Catmull-Clark~\citep{Catmull.1978} or double-iterations of $\sqrt{3}$ subdivision~\citep{Kobbelt.2000}.

\section{Subdivision $k$-form spaces}
\label{sec:subdiv_k_forms}

This section introduces subdivision schemes for differential $k$-forms on the subdivision-induced mesh hierarchy constructed above. To this end, we first define subdivision $k$-form spaces as the images of $k$-form subdivision operators acting on lowest-order FEEC spaces and present some of their properties. Using these results, we then construct a basis for the subdivision $k$-form spaces and introduce the notion of subdivision-induced regularity. These two different perspectives on the subdivision $k$-form spaces allow us to establish conditions under which the spaces constitute discrete de Rham complexes. Additional details on the construction of these spaces can be found in Appendix~\ref{app:subdiv_spaces}.

\subsection{Definition of subdivision $k$-forms via subdivision operators}
\label{subsec:subdiv_k_forms_through_subdiv_algorithm}

The construction of subdivision $k$-forms is based on the lowest-order FEEC spaces $\FEECspace{k}{}$ with
\begin{equation}
    \label{eq:choice_of_feec_space}
    \FEECspace{0}{} = \CG_1, \qquad \FEECspace{1}{} = \NED_1, \qquad \FEECspace{2}{} = \DG_0,
\end{equation}
where $\CG_1$ stands for a piecewise linear continuous Galerkin space, $\NED_1$ for the lowest-order N\'ed\'elec space,
and $\DG_0$ for a piecewise constant discontinuous Galerkin space. In the following, we will refer to these spaces as the \emph{FEEC spaces}.
For a given mesh hierarchy of $\T_{\lzero}, \dots, \T_\L$ constructed by the subdivision operator in Definition~\ref{def:mesh_subdiv_operator}, we define the FEEC spaces
\begin{equation}
  \label{eq:FEEC_space_on_lvl_l}
  \FEECspace{k}{\l} \coloneq \FEECspace{k}{} \big( \T_\l \big)  
\end{equation}
on level $\l$ with $\lzero \leq \l \leq \L$. We can represent any discrete $k$-form $\omega^k_\l \in \FEECspace{k}{\l}$ as
\begin{eqnarray}\label{equ_FEEC_basis}
    \omega^k_\l = \sum_{\coarseIdxone = 1}^{\sumdimlk{\l}{k}} \freeQ{c}{k}{\l}{\coarseIdxone} \; \freeQ{\FEECbasis}{k}{\l}{\coarseIdxone},
\end{eqnarray}
where $\freeQ{\psi}{k}{\l}{\coarseIdxone}$ is the $\coarseIdxone^{\text{th}}$ basis function of $\FEECspace{k}{\l}$ with $\mathrm{dim}\big( \FEECspace{k}{\l} \big) = \dimlk{\l}{k}$ on mesh level $\l$. As with vertices, we use the index $\coarseIdxone$ for basis functions on mesh level $\l$ and the index $\fineIdxone$ for basis functions on mesh level $\l +1$.

The subdivision $k$-forms introduced in this work are the images of $k$-form subdivision operators. The following definition introduces these operators as a linear map from the coarse function space $\FEECspace{k}{\l}$ to the finer function space $\FEECspace{k}{\l+1}$. Note the use of the additional superscript $k$ to distinguish this new operator from the mesh subdivision operator of Definition~\ref{def:mesh_subdiv_operator}.

\begin{definition}[$k$-form subdivision operator with matrix representation]
\label{def:subdiv_operator}
The $k$-form subdivision operator~$\Sop{k}{\l}{\l+1}$ is defined as the linear map
\begin{equation}
\label{eq:alternative_subdiv_operator_mapping}
\begin{aligned}
    \Sop{k}{\l}{\l+1}: \; \FEECspace{k}{\l} \;\; & \to \;\; \FEECspace{k}{\l+1}, \quad
    \freeVec{\omega}{k}{\l} = \sum_{\coarseIdxone=1}^{\dimlk{\l}{k}} \freeQ{c}{k}{\l}{\coarseIdxone} \; \freeQ{\FEECbasis}{k}{\l}{\coarseIdxone} \;\; \mapsto \;\; \constrVec{\omega}{k}{\l+1}{\l} = \sum_{\fineIdxone=1}^{\dimlk{\l+1}{k}} \constrQ{c}{k}{\l+1}{\l}{\fineIdxone} \; \freeQ{\FEECbasis}{k}{\l+1}{\fineIdxone} ,
     \\
     & \quad \text{with \emph{constrained} coefficients} \quad \constrQ{c}{k}{\l+1}{\l}{\fineIdxone}  = \sum_{\coarseIdxone=1}^{\dimlk{\l}{k}} \Smatij{k}{\l}{\l+1}{\fineIdxone}{\coarseIdxone} \; \freeQ{c}{k}{\l}{\coarseIdxone} \, ,
\end{aligned}
\end{equation}
where the sparse $\dimlk{\l+1}{k} \times \dimlk{\l}{k}$ matrix
$\Smat{k}{\l}{\l +1}$ is the matrix representation of the subdivision operator $\Sop{k}{\l}{\l+1}$
with entries determined through the $k$-form subdivision rules using the relationship in Eq.~\eqref{eq:subdiv_matrix_from_subdiv_rule}
in Appendix~\ref{app:subdiv_spaces}.
The subdivision rules for mesh subdivision and $0$-form subdivision coincide.

This operator maps $k$-forms on level $\l$ represented with respect to the basis functions $\freeQ{\FEECbasis}{k}{\l}{\coarseIdxone}$ of $\FEECspace{k}{\l}$ to corresponding $k$-forms on level $\l+1$ represented with respect to the basis functions $\freeQ{\FEECbasis}{k}{\l+1}{\fineIdxone}$ of $\FEECspace{k}{\l+1}$.
In terms of its \emph{matrix representation}, the operator acts via the matrix $\Smat{k}{\l}{\l +1}$ directly on the coefficient
vector $\constrQ{c}{k}{\l+1}{\l}{\fineIdxone}$ while the basis vectors of level $\l$ are replaced with those of $\l +1$.
\end{definition}

\begin{remark}
\label{remark:constraint_notation}
Note the difference in notation for the coefficients of $k$-forms. From now on, the vertical bar in the expression $\l+1\vert\l$ indicates that a quantity on the finer level $\l +1$ has to satisfy some constraint inherited from the coarser level $\l$. Typically, these constraints are also lower-dimensional.
For example, the coefficients $\freeQ{c}{k}{\l}{\coarseIdxone}$ of $\freeVec{\omega}{k}{\l}$ in Eq.~\eqref{eq:alternative_subdiv_operator_mapping} are \emph{free} while the coefficients $\constrQ{c}{k}{\l+1}{\l}{\fineIdxone}$ are \emph{constrained} in the sense that they are restricted to the image of the subdivision matrix according to Definition~\ref{def:subdiv_operator}.
\end{remark}

Analogously to the mesh subdivision operators $\Sop{}{\l}{\l+1}$, the $k$-form subdivision operators
$\Sop{k}{\l}{\l +1}$ can be subsequently applied to connect all levels of the mesh hierarchy as done next.

\begin{definition}[Accumulated $k$-form subdivision operator]
    \label{def:accumulated_subdiv_operator}
    The accumulated $k$-form subdivision operator $\Aop{k}{\l}{\L}$ maps discrete $k$-forms from level $\l$ to level $\L$ through
    \begin{equation}
        \label{eq:accumulated_subdiv_operator}
        \Aop{k}{\l}{\L} \coloneqq \Sop{k}{\L-1}{\L} \, \circ \, \hdots \, \circ \, \Sop{k}{\l+1}{\l+2} \, \circ \, \Sop{k}{\l}{\l+1} \; \colon \;\; \FEECspace{k}{\l} \to  \FEECspace{k}{\L}.
    \end{equation}
    Its coordinate representation with respect to the bases $\freeQ{\FEECbasis}{k}{\l}{\coarseIdxone}$ of $\FEECspace{k}{\l}$ and $\freeQ{\FEECbasis}{k}{\L}{\fineIdxone}$
    of $\FEECspace{k}{\L}$ is given by the accumulated (sparse) subdivision matrix
    \begin{equation}
        \label{eq:accumulated_subdiv_matrix}
        \Amat{k}{\l}{\L} \coloneqq \Smat{k}{\L-1}{\L} \, \cdot \, \dots \, \cdot \, \Smat{k}{\l+1}{\l+2} \, \cdot \, \Smat{k}{\l}{\l+1},
    \end{equation}
    with dimensions~$\dimlk{\L}{k} \times \dimlk{\l}{k}$, where the centered dots denote matrix multiplication.
\end{definition}

Definition~\ref{def:accumulated_subdiv_operator} introduces the accumulated subdivision operator $\Aop{k}{\l}{\L}$. This operator can be split with respect to an intermediate level, as formulated in the next proposition.

\begin{proposition}
    \label{prop:split_subdiv_operators}
    Given an intermediate mesh level $\l$ with $\lzero \leq \l \leq \L$, the accumulated
    $k$-form subdivision operator can be decomposed relative to level
    $\l$ by $\Aop{k}{\lzero}{\L} = \Aop{k}{\l}{\L} \circ \Aop{k}{\lzero}{\l}$.
\end{proposition}
\begin{proof}
The property follows directly from Eq.~\eqref{eq:accumulated_subdiv_operator}.
\end{proof}

This accumulated $k$-form subdivision operator $\Aop{k}{\l}{\L}$ induces the \emph{space of subdivision $k$-forms} on level $\L$ relative to level $\l$ as the image of $\FEECspace{k}{\l}$ under the action of $\Aop{k}{\l}{\L}$. Denoting the image of a set $X$ under the map $f:X \to Y$ by
\begin{equation}\label{equ:image_map}
    f\big[X \big] \coloneqq \big\{ y \in Y: \;\; y = f(x),\; x \in X \big\},
\end{equation}
we introduce the spaces of subdivision $k$-forms in the following manner.

\begin{definition}[Subdivision $k$-form spaces]
    \label{def:uniform_subdiv_spaces}
    Consider an initial triangulation $\T_{\lzero}$ and a mesh hierarchy $\T_{\l} = \Aop{}{\lzero}{\l} \T_{\lzero}$ with $\lzero \leq \l \leq \L$ constructed using the mesh subdivision operator $\Aop{}{\lzero}{\l}$. Then, we define the subdivision $k$-form spaces $\SDFspace{k}{\L}{\l}(\T_{\lzero})$ on level $\L$ relative to level $\l$ as
    \begin{equation}\label{eq:SLambda_definition_subspace_PLambda}
    \begin{split}
        \SDFspace{k}{\L}{\l}(\T_{\lzero})
        \coloneqq \Aop{k}{\l}{\L} \big[ \FEECspace{k}{\l} \big] \, ,
     \end{split}
    \end{equation}
    where $\FEECspace{k}{\l} = \FEECspace{k}{}\big( \T_{\l} \big)$ according to Eq.~\eqref{eq:FEEC_space_on_lvl_l}. Fully explicitly, the right-hand side reads
    \begin{equation}\label{eq:SLambda_definition_subspace_PLambda_v1}
    \begin{split}
     \Aop{k}{\l}{\L} \big[ \FEECspace{k}{\l}  \big]
     = \big\{ \constrVec{\omega}{k}{\L}{\l} \in \FEECspace{k}{\L} \;\; \colon \;\; \constrVec{\omega}{k}{\L}{\l}  =  \Aop{k}{\l}{\L}  \; \freeVec{\omega}{k}{\l} , \;\; \freeVec{\omega}{k}{\l} \in \FEECspace{k}{\l} \big\}.
     \end{split}
    \end{equation}
\end{definition}

Recall that the notation ${\L|\l}$ reflects the idea outlined in Remark~\ref{remark:constraint_notation}. In the following, we will often omit the explicit dependence of the subdivision $k$-form spaces on the initial mesh $\T_{\lzero}$ if this dependency is not crucial or becomes clear from the context.

Next, let us consider the properties of the subdivision $k$-form spaces of Definition~\ref{def:uniform_subdiv_spaces}.
\begin{proposition}
    \label{prop:subspace_of_Vk}
    The subdivision $k$-form spaces $\SDFspace{k}{\L}{\l}$ have the following properties: \\
    (i) $\SDFspace{k}{\l}{\l}  \equiv \FEECspace{k}{\l}$ for all $\l \in \{\lzero, \hdots, \L\}$,\\
    (ii) $\SDFspace{k}{\L}{\l} \subset \FEECspace{k}{L}$, i.e. $\SDFspace{k}{\L}{\l}$ is a $\dimlk{\l}{k}$-dimensional subspace of $\FEECspace{k}{L}$. 
\end{proposition}
\begin{proof}
See Appendix~\ref{proof:subspace_of_Vk}.
\end{proof}

Crucially, $\SDFspace{k}{\L}{\l} \subset \FEECspace{k}{\L}$ implies that the subdivision $k$-form spaces do not exhibit increased Sobolev regularity but only a certain kind of subdivision-induced regularity, see Remark~\ref{remark:subdiv_smoothness}.

Because $\SDFspace{k}{\L}{\l}$ are subspaces of $\FEECspace{k}{\L}$ for all $\lzero \leq \l \leq \L$
according to Proposition~\ref{prop:subspace_of_Vk}, we can define the following inclusion map.
\begin{definition}[Inclusion map]
    \label{def:inclusion_map}
    The inclusion map $\iota$ is defined as the function
    \begin{equation}
        \iota\colon \; \SDFspace{k}{\L}{\l} \to \FEECspace{k}{\L}, \qquad \iota (\constrVec{\omega}{k}{\L}{\l}) = {\constrVec{\omega}{k}{\L}{\l}} = \freeVec{\omega}{k}{\L}.
    \end{equation}
\end{definition}
This definition allows us to treat elements of $\SDFspace{k}{\L}{\l}$ as elements of $\FEECspace{k}{\L}$ by forgetting that the coefficients of ${\constrVec{\omega}{k}{\L}{\l}}$ were constrained through the subdivision matrix. This will become useful
later when defining an exterior derivative that acts on these subdivision $k$-forms.

\smallskip

Proposition~\ref{prop:split_subdiv_operators} allows us to split the subdivision operators between $\l$ and $\L$ on an intermediate level $\lone$. This can be interpreted as applying the subdivision operator on an intermediate subdivision $k$-form space. For the purpose of further specifying the domains of the subdivision operators, we introduce the following subdivision operator $\primalAop{k}{\lone}{\ltwo}{\l}$. The notation now includes the initial level $\l$ to preserve information about the level $\l$ on which the degrees of freedom (DoFs) are defined.

\begin{definition}\label{def:primalAop}
  For a fixed level $\l$ and varying intermediate levels $\lone, \ltwo$ with $\lzero \leq \l \leq \lone \leq \ltwo \leq \L$,
  the (extended) accumulated $k$-form subdivision operator
  $\primalAop{k}{\lone}{\ltwo}{\l}$ is a mapping of subdivision $k$-forms from
  $\SDFspace{k}{\lone}{\l}$ to $\SDFspace{k}{\ltwo}{\l}$, i.e.
  \begin{equation}\label{equ:primalAop}
  \begin{aligned}
    \primalAop{k}{\lone}{\ltwo}{\l}: & \  \SDFspace{k}{\lone}{\l}  \to \SDFspace{k}{\ltwo}{\l} \, ,
     \quad \constrVec{\omega}{k}{\lone}{\l} = \sum_{\fineIdxone = 1}^{\sumdimlk{\lone}{k}} \constrQ{c}{k}{\lone}{\l}{\fineIdxone} \, \freeQ{\FEECbasis}{k}{\lone}{\fineIdxone} \mapsto
   \constrVec{\omega}{k}{\ltwo}{\l} = \sum_{\fineIdxone = 1}^{\sumdimlk{\ltwo}{k}} \constrQ{c}{k}{\ltwo}{\l}{\fineIdxone} \, \freeQ{\FEECbasis}{k}{\ltwo}{\fineIdxone} \\
   & \quad \text{with \emph{constrained} coefficients} \quad \constrQ{c}{k}{\ltwo}{\l}{\fineIdxone}
   = \sum_{\coarseIdxone=1}^{\dimlk{\lone}{k}}  \Amatij{k}{\lone}{\ltwo}{\fineIdxone}{\coarseIdxone} \; \constrQ{c}{k}{\lone}{\l}{\fineIdxone}\, ,
  \end{aligned}
  \end{equation}
  where $\constrQ{c}{k}{\lone}{\l}{\fineIdxone}$ is given as in Definition~\ref{def:subdiv_operator}
  and where $\Amatij{k}{\lone}{\ltwo}{\fineIdxone}{\coarseIdxone}$ is a sparse $\dimlk{\ltwo}{k} \times \dimlk{\lone}{k}$ matrix.
  The matrix representation of this operator follows the guideline of Definition~\ref{def:subdiv_operator}.

  For the special case in which $\lone =\l$, we recover Definition \ref{def:accumulated_subdiv_operator}
  and we use the simplified notation $\Aop{k}{\lone}{\ltwo}$.
\end{definition}

\begin{remark}\label{remark_exclude_l2geql1}
    Throughout the paper, we will assume $\lzero \leq \l \leq \lone \leq \ltwo \leq \L$ whenever we write
    $\Aop{k}{\lone}{\ltwo}$ or $\primalAop{k}{\lone}{\ltwo}{\l}$
    since the subdivision operator is not defined for $\lone > \ltwo$.
\end{remark}

The following computation confirms that the generalised picture of the accumulated $k$-form subdivision operator introduced in Definition~\ref{def:primalAop} is compatible with previous definitions. 
Consider the operator
$\primalAop{k}{\lone}{\ltwo}{\l}$ acting on
$\constrVec{\omega}{k}{\lone}{\l} \in \SDFspace{k}{\lone}{\l}$ like
$\primalAop{k}{\lone}{\ltwo}{\l} \  \constrVec{\omega}{k}{\lone}{\l} = \constrVec{\omega}{k}{\ltwo}{\l} \in \SDFspace{k}{\ltwo}{\l}$.
Note how the first entry in the $\cdot | \cdot$ notation is affected by this operation while
the second entry remains unmodified. Using the explicit matrix representation,
this action reads
\begin{equation}\label{equ:explicit_primal_op}
\begin{aligned}
 \primalAop{k}{\lone}{\ltwo}{\l} \  \constrVec{\omega}{k}{\lone}{\l}
& = \primalAop{k}{\lone}{\ltwo}{\l} \sum_{\fineIdxone = 1}^{\sumdimlk{\lone}{k}} \constrQ{c}{k}{\lone}{\l}{\fineIdxone} \, \freeQ{\FEECbasis}{k}{\lone}{\fineIdxone} \\
 & = \sum_{\coarseIdxone=1}^{\sumdimlk{\l}{k}} \freeQ{c}{k}{\l}{\coarseIdxone} \;
 \primalAop{k}{\lone}{\ltwo}{\l} \big(\sum_{\fineIdxone = 1}^{\sumdimlk{\lone}{k}}  \Amatij{k}{\l}{\lone}{\fineIdxone}{\coarseIdxone}  \freeQ{\FEECbasis}{k}{\lone}{\fineIdxone} \big) \\
 & = \sum_{\coarseIdxone=1}^{\sumdimlk{\l}{k}} \freeQ{c}{k}{\l}{\coarseIdxone} \;
 \big( \sum_{\fineIdxone = 1}^{\sumdimlk{\ltwo}{k}}(\Amat{k}{\lone}{\ltwo} \cdot \Amat{k}{\l}{\lone})_{\fineIdxone\coarseIdxone}
 \freeQ{\FEECbasis}{k}{\ltwo}{\fineIdxone} \big) \\
& = \sum_{\coarseIdxone=1}^{\sumdimlk{\l}{k}} \freeQ{c}{k}{\l}{\coarseIdxone} \;
 \big( \sum_{\fineIdxone = 1}^{\sumdimlk{\ltwo}{k}}(\Amat{k}{\l}{\ltwo})_{\fineIdxone\coarseIdxone}
 \freeQ{\FEECbasis}{k}{\ltwo}{\fineIdxone} \big) \\
&  = \sum_{\fineIdxone = 1}^{\sumdimlk{\ltwo}{k}} \constrQ{c}{k}{\ltwo}{\l}{\fineIdxone} \, \freeQ{\FEECbasis}{k}{\ltwo}{\fineIdxone}
= \constrVec{\omega}{k}{\ltwo}{\l}.
\end{aligned}
\end{equation}
Note that in line three the action of $\primalAop{k}{\lone}{\ltwo}{\l}$ changes the summation index to
$\sumdimlk{\ltwo}{k}$ and the basis functions to $\freeQ{\FEECbasis}{k}{\ltwo}{\fineIdxone} $,
according to Definition~\ref{def:subdiv_operator} of the
subdivision operator and its matrix representation.

This explicit computation illustrates how
$\primalAop{k}{\lone}{\ltwo}{\l}$ acts on the constrained coefficients $\constrQ{c}{k}{\lone}{\l}{}$ of elements
of $\SDFspace{k}{\lone}{\l}$ and that the operation aligns with Definitions \ref{def:subdiv_operator}
and \ref{def:accumulated_subdiv_operator}.

\begin{remark}\label{remark_shorthand}
As the preceding computation reveals, the action of $\primalAop{k}{\lone}{\ltwo}{\l}$
on $\constrVec{\omega}{k}{\lone}{\l} $ depends only on $\lone$ and $\ltwo$ but not on $\l$. In particular, note that the matrix expression $\Amat{k}{\lone}{\ltwo}$ is independent of $\l$. Nevertheless, the $\cdot | \l$ notation allows us to immediately identify on which
level the coefficients of the subdivision $k$-forms are constrained, i.e., it contains information about the level on which the DoFs were defined.

When using $\Aop{k}{\lone}{\ltwo}$ instead of $\primalAop{k}{\lone}{\ltwo}{\l}$ we assume
that either $\l = \lone$ or that the action of the operator
(especially when acting on other subdivision matrices) becomes clear from the context.
\end{remark}

Next, let us discuss the properties of the subdivision $k$-form spaces
$\SDFspace{k}{\L}{\l}(\T_\lzero)$ as we move along the hierarchy. The following definitions allow us to parameterise the hierarchy based on the level $\l$ relative to the initial and final levels $\lzero$ and $\L$.

\begin{definition}
    \label{def:refinement_and_coarsening_numbers}
    Given a hierarchy of meshes and subdivision operators from $\lzero$ to $\L$, the refinement number $\nr$ and smoothing number $\ns$ associated to the spaces $\SDFspace{k}{\L}{\l}\big(\T_\lzero\big)$ from Definition~\ref{def:uniform_subdiv_spaces} are defined as
    \begin{equation}
        \nr \coloneq \l \qquad \text{and} \qquad \ns \coloneq \L - \l.
    \end{equation}
\end{definition}

Note that these definitions imply that the spaces of the hierarchy can be parameterised using
\begin{equation}
    \label{eq:re-param_of_spaces}
    \SDFspace{k}{\L}{\l} \equiv \SDFspace{k}{ \nr + \ns}{\nr}.
\end{equation}

According to Definition~\ref{def:uniform_subdiv_spaces}, the refinement number $\nr$ represents the number of mesh refinement steps carried out before fixing the DoFs of the subdivision $k$-form spaces on level $\l$. Thus, $\nr$ is proportional to the number of DoFs of the space. In contrast, the smoothing number $\ns$ measures the number of refinement steps carried out after fixing $\l$ to construct the subdivision $k$-form spaces. Thus, as we increase $\ns$, the basis functions $\constrQ{\SDFbasis}{k}{\l}{\L}{\coarseIdxone}$ exhibit increased subdivision-induced regularity, see Remark~\ref{remark:subdiv_smoothness}. In the limit $\ns \to \infty$, the basis functions $\constrQ{\SDFbasis}{k}{\l}{\L}{\coarseIdxone}$ would become genuinely smooth,  as defined in Remark~\ref{remark:limit_basis_functions}.

\smallskip

The following propositions investigate the effect of varying $\nr$ and $\ns$.
This allows us to understand how the properties of the $k$-form spaces $\SDFspace{k}{\L}{\l}$ change as we move along the hierarchy of approximation spaces. First, we consider the case of increasing $\ns$, i.e. the relation of the spaces $\SDFspace{k}{\lone}{\l}$ with fixed $\l$ and varying finest levels $\lone$ and $\ltwo$.

\begin{proposition}
    \label{prop:subdiv_operator_refined}
    Let $\lzero \leq \l \leq \lone \leq \ltwo$. The subdivision $k$-form spaces $\SDFspace{k}{\lone}{\l}$
    are closed under the action of the accumulated subdivision operator $\primalAop{k}{\lone}{\ltwo}{\l}$ from Definition~\ref{def:primalAop} in the sense that
    \begin{equation}
        \label{eq:subdiv_operator_refined}
        \SDFspace{k}{\ltwo}{\l} = \primalAop{k}{\lone}{\ltwo}{\l} \big[ \SDFspace{k}{\lone}{\l} \big].
    \end{equation}
\end{proposition}
\begin{proof}[Proof of Proposition~\ref{prop:subdiv_operator_refined}]
    \label{proof:subdiv_operator_refined}
    Fixing $\l$ and using the shorthand notation $\Aop{k}{\lone}{\ltwo}$ instead
    of $\primalAop{k}{\lone}{\ltwo}{\l}$ whenever $\lone = \l$ (cf. Remark~\ref{remark_shorthand}),
    we can directly compute
    \begin{equation}
        \SDFspace{k}{\ltwo}{\l}
        = \Aop{k}{\l}{\ltwo} \big[ \FEECspace{k}{\l} \big]
        = \big(\primalAop{k}{\lone}{\ltwo}{\l}  \circ \Aop{k}{\l}{\lone} \big)\big[ \FEECspace{k}{\l} \big]
        = \primalAop{k}{\lone}{\ltwo}{\l} \circ \big(\Aop{k}{\l}{\lone} \big[ \FEECspace{k}{\l} \big] \big)
        = \primalAop{k}{\lone}{\ltwo}{\l} \big[\SDFspace{k}{\lone}{\l} \big],
    \end{equation}
    where we used Definition~\ref{def:uniform_subdiv_spaces} in the first and fourth equality, and Proposition~\ref{prop:split_subdiv_operators}
    in the second.
    Observe that $\mathrm{dim}\big( \SDFspace{k}{\lone}{\l} \big)
    = \mathrm{rank}\big(  \Amat{k}{\l}{\lone} \big)
    = \dimlk{\l}{k}
    = \mathrm{rank}\big( \Amat{k}{\l}{\ltwo} \big)
    = \mathrm{dim}\big(  \SDFspace{k}{\ltwo}{\l} \big)$
    by Proposition~\ref{prop:subspace_of_Vk}. Therefore, $\primalAop{k}{\lone}{\ltwo}{\l}$ is an isomorphism
    and the spaces $\SDFspace{k}{\lone}{\l}$ are closed under the action of $\primalAop{k}{\lone}{\ltwo}{\l}$
    for $\l \leq \lone \leq \ltwo \leq L$.
\end{proof}

Proposition~\ref{prop:subspace_of_Vk} showed that, for a fixed level $\l$, the spaces $\SDFspace{k}{\lone}{\l}$ have the same dimensions
for all $\lone$ with $\l \leq \lone \leq \L$. According to Proposition~\ref{prop:subdiv_operator_refined}, increasing $\L$ and thus the smoothing number $\ns$ keeps the dimension of the spaces constant.
In contrast, the next proposition investigates how increasing the refinement number $\nr$ affects the spaces $\SDFspace{k}{\L}{\l}$.

\begin{proposition}
\label{prop:uniform_spaces_nested}
Given an initial mesh $\T_{\lzero}$ and a hierarchy with $\lzero \leq \l \leq \L$, the
corresponding subdivision $k$-form spaces $\SDFspace{k}{\L}{\l}$ are nested, i.e.
\begin{equation}
    \label{eq:nested_spaces}
    \SDFspace{k}{\L}{\lzero} \subset  \SDFspace{k}{\L}{1} \subset \; \hdots \; \subset \SDFspace{k}{\L}{\l} \subset \; \hdots \; \subset \SDFspace{k}{\L}{\L} = \FEECspace{k}{\L},
\end{equation}
and have the dimensions $\dimlk{\lzero}{k} < \dimlk{1}{k} < \hdots < \dimlk{\L}{k}$.
\end{proposition}
\begin{proof}[Proof of Proposition~\ref{prop:uniform_spaces_nested}]
    \label{proof:uniform_spaces_nested}
    We prove the claim for $\SDFspace{k}{\L}{\l} \subset \SDFspace{k}{\L}{\l+1}$. This is sufficient because $\l$ is an arbitrary level between $\lzero$ and $\L$ which implies that the entire sequence of spaces is nested.
    We recall from Proposition~\ref{prop:subdiv_operator_refined} that
    \begin{equation}
        \SDFspace{k}{\L}{\l} = \primalAop{k}{\l+1}{\L}{\l} \, \big[ \SDFspace{k}{\l+1}{\l} \big]
    \end{equation}
    while Definition~\ref{def:uniform_subdiv_spaces} gives
    \begin{equation}
        \SDFspace{k}{\L}{\l+1} = \primalAop{k}{\l+1}{\L}{\l+1} \big[\FEECspace{k}{\l+1} \big].
    \end{equation}
    Proposition~\ref{prop:subspace_of_Vk} implies that $\SDFspace{k}{\l+1}{\l} \subset \FEECspace{k}{\l+1}$. Applying subdivision to both spaces results in
    $\primalAop{k}{\l+1}{\L}{\l}\big[\SDFspace{k}{\l+1}{\l}\big] \subset \primalAop{k}{\l+1}{\L}{\l+1}[\FEECspace{k}{\l+1}]$. Consequently,
    \begin{equation}
        \SDFspace{k}{\L}{\l} = \primalAop{k}{\l+1}{\L}{\l}\big[\SDFspace{k}{\l+1}{\l} \big] \; \subset \; \primalAop{k}{\l+1}{\L}{\l+1}\big[\FEECspace{k}{\l+1}\big] =  \SDFspace{k}{\L}{\l+1}.
    \end{equation}
    Finally, observe that by Proposition~\ref{prop:subspace_of_Vk}, $\mathrm{dim}\big( \SDFspace{k}{\L}{\l} \big) = \dimlk{\l}{k} < \dimlk{\l+1}{k} = \mathrm{dim}\big( \SDFspace{k}{\L}{\l+1} \big)$ and thus $\SDFspace{k}{\L}{\l}$ has to be a proper subspace of $\SDFspace{k}{\L}{\l+1}$.
\end{proof}

Note that the refinement number $\nr$ increases by $1$ every time we move from $\l$ to $\l +1$ in Eq.~\eqref{eq:nested_spaces}, where $\nr$ can take any value from $0$ to $\L - \l$. Thus, the proposition above shows that increasing the refinement number also increases the number of DoFs.

Since $\L$ needs to remain constant in order for this inclusion to work, this also implies that the smoothing number $\ns$ decreases by one with every inclusion into a superspace. This is the prevalent pattern for the approximation spaces as we move along the hierarchy. We will explore the consequences of that in more detail in Section~\ref{subsec:projection}. 

Section~\ref{sec:numerics} typically compares the approximations obtained using the subdivision $k$-form spaces $\SDFspace{k}{\ltwo}{\lone}$ with the approximations obtained from the ambient FEEC spaces $\FEECspace{k}{\ltwo}$. Under these circumstances, we consider these FEEC spaces $\FEECspace{k}{\ltwo}$ to represent the ground truth. Since the subdivision $k$-form spaces are subspaces of these FEEC spaces by Proposition~{\ref{prop:uniform_spaces_nested}}, they present an option to reduce the dimensionality of the approximation spaces without sacrificing too much accuracy. Section~\ref{subsec:projection} investigates this ability to reduce the number of DoFs, i.e. to coarsen the subdivision $k$-form spaces.

\subsection{Constructing a basis for the subdivision $k$-form spaces}
\label{subsec:basis_for_subdiv_spaces}

The previous section defined the subdivision $k$-form spaces in terms of the basis $\freeQ{\FEECbasis}{k}{\L}{\fineIdxone}$ of $\FEECspace{k}{\L}$ at the finest mesh level $\L$ with a linear constraint on the coefficients given through the subdivision matrix (cf. Eq.~\eqref{eq:alternative_subdiv_operator_mapping}).

The aim of this section is to remove this constraint on the coefficients by constructing a basis for the spaces $\SDFspace{k}{\L}{\l}$ that satisfies the constraints intrinsically. This is important for numerical computations, where we numerically determine the coefficients of a basis expansion of the approximate solution. Additional constraints on the coefficients are likely to make solving these systems cumbersome and inefficient.

To this end, we first introduce an alternative definition for subdivision $k$-form spaces and then show that
these spaces are equivalent to those defined in Definition~\ref{def:uniform_subdiv_spaces}.

\begin{definition}
    \label{def:alternative_char_for_subdiv_spaces}
    Fix an initial mesh $\T_{\lzero}$ and let $\T_{\l} = \Aop{}{\lzero}{\l} \T_{\lzero}$ for any $\l \in \{\lzero, \dots, \L\}$. Further, let $\freeQ{\FEECbasis}{k}{\l}{\coarseIdxone}$ denote the $i^{\text{th}}$ basis function of $\FEECspace{k}{\l}$ and construct \emph{constrained basis functions} as follows:
    \begin{equation}\label{def:basis_constraint}
        \constrVec{\Phi}{k}{\l}{\L}
        \coloneqq \bigcup_{\coarseIdxone = 1}^{\sumdimlk{\l}{k}}  \constrQ{\SDFbasis}{k}{\l}{\L}{\coarseIdxone}
        \quad \text{with} \quad \constrQ{\SDFbasis}{k}{\l}{\L}{\coarseIdxone} \coloneqq \Aop{k}{\l}{\L}  \; \freeQ{\FEECbasis}{k}{\l}{\coarseIdxone} \;\; \in \FEECspace{k}{\L}.
    \end{equation}
    Then, the space $\SDFspace{k}{\l}{\L}$ is defined as
    \begin{equation}
        \SDFspace{k}{\l}{\L}\big( \T_{\lzero} \big) \coloneqq \mathrm{span}\big( \constrVec{\Phi}{k}{\l}{\L} \big).
    \end{equation}
\end{definition}
As done previously, the dependency on the mesh $\T_\lzero$ will be suppressed whenever it is clear from the context. The subscript $\l|\L$ indicates that the space and its basis on level $\l$ depend, via the subdivision operator
$\Aop{k}{\l}{\L}$, on the basis coefficients from level $\L$. This contrasts Definition~\ref{def:uniform_subdiv_spaces} where the subscript $\L|\l$ is inverted. Still, both spaces are equivalent as shown next.

\begin{theorem}
    \label{theo:equivalence_of_spaces}
    The spaces $\SDFspace{k}{\l}{\L}$ and $\SDFspace{k}{\L}{\l}$ are equivalent.
\end{theorem}
\begin{proof}
    Let $\constrVec{\omega}{k}{\L}{\l} \in \SDFspace{k}{\L}{\l}$ and $\constrVec{\omega}{k}{\l}{\L} \in \SDFspace{k}{\l}{\L}$. From direct computations using the matrix representation
    of Definition~\ref{def:subdiv_operator} and Eq.~\eqref{eq:accumulated_subdiv_matrix} follows:
    \begin{equation}
    \label{eq:equivalence_of_spaces}
    \begin{aligned}
        \constrVec{\omega}{k}{\L}{\l}
        & = \sum_{\fineIdxone=1}^{\sumdimlk{\L}{k}} \constrQ{c}{k}{\L}{\l}{\fineIdxone} \; \freeQ{\FEECbasis}{k}{\L}{\fineIdxone}
        = \sum_{\fineIdxone=1}^{\sumdimlk{\L}{k}} \Big( \sum_{\coarseIdxone=1}^{\sumdimlk{\l}{k}} \Amatij{k}{\l}{\L}{\fineIdxone}{\coarseIdxone} \;\freeQ{c}{k}{\l}{\coarseIdxone} \Big) \; \freeQ{\FEECbasis}{k}{\L}{\fineIdxone} \\
        &= \sum_{\fineIdxone=1}^{\sumdimlk{\L}{k}} \freeQ{c}{k}{\l}{\coarseIdxone} \; \Big( \sum_{\coarseIdxone=1}^{\sumdimlk{\l}{k}} \Amatij{k}{\l}{\L}{\fineIdxone}{\coarseIdxone} \; \freeQ{\FEECbasis}{k}{\L}{\fineIdxone} \Big)
        = \sum_{\coarseIdxone=1}^{\sumdimlk{\L}{k}} \freeQ{c}{k}{\l}{\coarseIdxone} \; \constrQ{\SDFbasis}{k}{\l}{\L}{\coarseIdxone}
        = \constrVec{\omega}{k}{\l}{\L},
    \end{aligned}
    \end{equation}
    where we applied the definition of the constrained basis functions in Eq.~\eqref{def:basis_constraint}.
    This implies that every $\constrVec{\omega}{k}{\L}{\l}$ corresponds to an equivalent $\constrVec{\omega}{k}{\l}{\L}$ and thus $\SDFspace{k}{\l}{\L} \equiv \SDFspace{k}{\L}{\l}$.
\end{proof}

The equivalence of the spaces $\SDFspace{k}{\l}{\L}$ and $\SDFspace{k}{\L}{\l}$ implies that,
analogously to $\SDFspace{k}{\L}{\l} \subset \FEECspace{k}{\L}$ from Proposition~\ref{prop:subspace_of_Vk}, also $\SDFspace{k}{\l}{\L} \subset \FEECspace{k}{\L}$.
Therefore, we can define an inclusion map, similarly to above, as follows.
\begin{definition}
    \label{def:tilde_inclusion_map}
    The inclusion map $\auxincl$ is the map
    \begin{equation}
        \label{eq:tilde_inclusion}
        \auxincl: \; \SDFspace{k}{\l}{\L} \to \FEECspace{k}{\L},
        \qquad \auxincl \big( \constrVec{\omega}{k}{\l}{\L} \big)
        = \constrVec{\omega}{k}{\l}{\L}  \equiv  \constrVec{\omega}{k}{\L}{\l}
        = \iota \big( \constrVec{\omega}{k}{\L}{\l} \big),
    \end{equation}
    where $\constrVec{\omega}{k}{\l}{\L}$ and $\constrVec{\omega}{k}{\L}{\l}$ are
    equivalent according to Eq.~\eqref{eq:equivalence_of_spaces}.
\end{definition}

Note how in Eq.~\eqref{eq:equivalence_of_spaces} the summation index shifts from the coarse index $\coarseIdxone$ to the fine index $\fineIdxone$. Instead of contracting the subdivision matrix $\Amatij{k}{\l}{\L}{\fineIdxone}{\coarseIdxone}$ with the coefficients $\freeQ{c}{k}{\l}{\coarseIdxone}$, we now contract it with the fine basis functions $\freeQ{\FEECbasis}{k}{\L}{\fineIdxone}$ to obtain the basis functions $\constrQ{\SDFbasis}{k}{\l}{\L}{\coarseIdxone}$.

This observation motivates the following definition of a basis-exchange operator
across the subdivision-induced mesh hierarchy.

\begin{definition}[Basis-exchange (BE) operator]
    \label{def:adjoint_subdiv_operator}
    Let $\lzero \leq \l \leq \lone \leq \ltwo \leq \L$. For a given level $\l$,
    the basis-exchange (BE) operator $\dualAop{k}{\l}{\lone}{\ltwo}$ is defined as the map
    \begin{equation}
    \begin{aligned}
        \label{eq:adoint_subdiv_operator}
        & \dualAop{k}{\l}{\lone}{\ltwo} \,  \colon \; \SDFspace{k}{\l}{\lone} \to \SDFspace{k}{\l}{\ltwo} \; , \quad \constrVec{\omega}{k}{\l}{\lone} = \sum_{\coarseIdxone}^{\sumdimlk{\l}{k}} \freeQ{c}{k}{\l}{\coarseIdxone} \, \constrQ{\SDFbasis}{k}{\l}{\lone}{\coarseIdxone} \;\; \mapsto \;\; \constrVec{\omega}{k}{\l}{\ltwo} = \sum_{\coarseIdxone}^{\sumdimlk{\l}{k}} \freeQ{c}{k}{\l}{\coarseIdxone} \, \constrQ{\SDFbasis}{k}{\l}{\ltwo}{\coarseIdxone},
    \end{aligned}
    \end{equation}
    with \emph{constrained} basis functions
    $\constrQ{\SDFbasis}{k}{\l}{\lone}{\coarseIdxone}$ as
    introduced in Definition~\ref{def:alternative_char_for_subdiv_spaces}.
    \medskip

    The matrix representation of $\dualAop{k}{\l}{\lone}{\ltwo}$ is a $\dimlk{\l}{k} \times \dimlk{\l}{k}$ identity matrix since the operator's action leaves the coefficient vector $\freeQ{c}{k}{\l}{\coarseIdxone}$
    unmodified and only exchanges the predefined constrained basis functions $\constrQ{\SDFbasis}{k}{\l}{\lone}{\coarseIdxone}$ for $\constrQ{\SDFbasis}{k}{\l}{\ltwo}{\coarseIdxone}$.
\end{definition}

\begin{corollary}
  \label{cor:BE_op_is_isomorphism}
  The BE operator $\dualAop{k}{\l}{\lone}{\ltwo}$ from Eq.~\eqref{eq:adoint_subdiv_operator} is an isomorphism between the spaces $\SDFspace{k}{\l}{\lone}$ and $\SDFspace{k}{\l}{\ltwo}$ for $\lzero \leq \l \leq \lone \leq \ltwo \leq \L$.
\end{corollary}

The following theorem ties together the subdivision operator $\primalAop{k}{\lone}{\ltwo}{\l}$ and the basis-exchange operator $\dualAop{k}{\lone}{\ltwo}{\l}$ and shows that they are defined consistently. The proof of the theorem further illustrates the action of the basis-exchange operator $\dualAop{k}{\lone}{\ltwo}{\l}$ in terms of its coordinate representation similar to Eq.~\eqref{equ:explicit_primal_op} for the subdivision operator.

\begin{theorem}
  \label{theorem:consistency_of_BE_op}
  Let $\constrVec{\omega}{k}{\l}{\lone} \in \SDFspace{k}{\l}{\lone}$ and $\constrVec{\omega}{k}{\lone}{\l} \in \SDFspace{k}{\lone}{\l}$ with $\constrVec{\omega}{k}{\l}{\lone} = \constrVec{\omega}{k}{\lone}{\l}$. The basis exchange operators $\dualAop{k}{\l}{\lone}{\ltwo}$ and the subdivision operators $\primalAop{k}{\lone}{\ltwo}{\l}$ are consistent in the sense that
  \begin{equation}
    \auxincl \; \dualAop{k}{\l}{\lone}{\ltwo} \; \constrVec{\omega}{k}{\l}{\lone} = \iota \; \primalAop{k}{\lone}{\ltwo}{\l} \; \constrVec{\omega}{k}{\lone}{\l}.
  \end{equation}
\end{theorem}
\begin{proof}
Letting the BE operator act on $\constrVec{\omega}{k}{\l}{\lone} \in \SDFspace{k}{\l}{\lone}$,
we can compute
  \begin{equation}
    \begin{aligned}
      \label{eq:BE_and_A_consistent}
      \auxincl \; \dualAop{k}{\l}{\lone}{\ltwo} \; \constrVec{\omega}{k}{\l}{\lone} \qquad 
      &\overset{\mathclap{\strut\text{Def. }\ref{def:adjoint_subdiv_operator}}}= \qquad \auxincl \; \big(\sum_{\coarseIdxone=1}^{\sumdimlk{\l}{k}} \freeQ{c}{k}{\l}{\coarseIdxone} \; \constrQ{\SDFbasis}{k}{\l}{\ltwo}{\coarseIdxone} \big) \\[-6pt]
      &\overset{\mathclap{\strut\text{Def. }\ref{def:alternative_char_for_subdiv_spaces}}}= \qquad \auxincl \; \Big( \sum_{\coarseIdxone=1}^{\sumdimlk{\l}{k}} \freeQ{c}{k}{\l}{\coarseIdxone} \; \big( \sum_{\fineIdxone = 1}^{\sumdimlk{\ltwo}{k}} \; (\Amat{k}{\l}{\ltwo})_{\fineIdxone\coarseIdxone} \; \freeQ{\FEECbasis}{k}{\ltwo}{\fineIdxone} \big) \Big) \\[-6pt]
      &\overset{\mathclap{\strut\text{Def. }\ref{def:tilde_inclusion_map}}}= \qquad \sum_{\fineIdxone = 1}^{\sumdimlk{\ltwo}{k}} \Big(   \sum_{\coarseIdxone=1}^{\sumdimlk{\l}{k}} (\Amat{k}{\l}{\ltwo})_{\fineIdxone\coarseIdxone} \; \freeQ{c}{k}{\l}{\coarseIdxone} \Big) \; \freeQ{\FEECbasis}{k}{\ltwo}{\fineIdxone} \\[-6pt]
      &\overset{\mathclap{\strut\text{Prop. }\ref{prop:split_subdiv_operators}}}= \qquad \sum_{\fineIdxone = 1}^{\sumdimlk{\ltwo}{k}} \Big( \sum_{\coarseIdxone=1}^{\sumdimlk{\l}{k}} \big( \sum_{\coarseIdxtwo = 1}^{\sumdimlk{\lone}{k}} (\Amat{k}{\lone}{\ltwo})_{\fineIdxone\coarseIdxtwo} (\Amat{k}{\l}{\lone})_{\coarseIdxtwo\coarseIdxone} \big) \; \freeQ{c}{k}{\l}{\coarseIdxone} \Big) \; \freeQ{\FEECbasis}{k}{\ltwo}{\fineIdxone} \\
      &= \qquad \sum_{\fineIdxone = 1}^{\sumdimlk{\ltwo}{k}} \Big( \sum_{\coarseIdxtwo=1}^{\sumdimlk{\lone}{k}} (\Amat{k}{\lone}{\ltwo})_{\fineIdxone\coarseIdxtwo} \; \constrQ{c}{k}{\lone}{\l}{\coarseIdxtwo} \Big) \; \freeQ{\FEECbasis}{k}{\ltwo}{\fineIdxone} \\[-8pt]
      &\overset{\mathclap{\strut\text{Def. }\ref{def:primalAop}}}= \qquad \iota \; \primalAop{k}{\lone}{\ltwo}{\l} \; \constrVec{\omega}{k}{\lone}{\l}. \hspace{70mm} \qedhere
    \end{aligned}
  \end{equation}
\end{proof}
The previous theorem confirms that the entire setup of the subdivision $k$-form spaces up until now is coherent. In particular, it shows that the two perspectives of constrained coefficients versus constrained basis functions are consistent since applying the subdivision operator $\primalAop{k}{\lone}{\ltwo}{\l}$ yields the same result as swapping basis functions with $\dualAop{k}{\l}{\lone}{\ltwo}$. The computation also reveals that the basis-exchange operator $\dualAop{k}{\l}{\lone}{\ltwo}$ introduces the subdivision matrix $\Amat{k}{\lone}{\ltwo}$ by
\begin{equation}
  \label{eq:BE_op_introduces_transposed_subdiv_matrix}
  \dualAop{k}{\l}{\lone}{\ltwo} \; \colon \;\; \sum_{\fineIdxone = 1}^{\sumdimlk{\lone}{k}} \; (\Amat{k}{\l}{\lone})_{\fineIdxone\coarseIdxone} \; \freeQ{\FEECbasis}{k}{\lone}{\fineIdxone} \;\; \mapsto \;\; \sum_{\fineIdxone = 1}^{\sumdimlk{\ltwo}{k}} \Big( \sum_{\coarseIdxtwo = 1}^{\sumdimlk{\lone}{k}} (\Amat{k}{\lone}{\ltwo})_{\fineIdxone\coarseIdxtwo} (\Amat{k}{\l}{\lone})_{\coarseIdxtwo\coarseIdxone} \Big) \; \freeQ{\FEECbasis}{k}{\ltwo}{\fineIdxone}
\end{equation}
according to the fourth line of Eq.~\eqref{eq:BE_and_A_consistent}. In contrast to the subdivision operator $\primalAop{k}{\lone}{\ltwo}{\l}$, where the subdivision matrix $\Amat{k}{\lone}{\ltwo}$ acts on the coefficient vector, applying the basis-exchange operator contracts the matrix with the fine basis functions $\freeQ{\FEECbasis}{k}{\ltwo}{\fineIdxone}$. The following remark gathers the results of this discussion.

\begin{remark}\label{remark_duality}
The BE operator  $\dualAop{k}{\l}{\lone}{\ltwo}$ can be considered to be
`dual' to the subdivision operator $\primalAop{k}{\lone}{\ltwo}{\l}$ of Definition~\ref{def:primalAop}
in the sense that, for any $\lone, \ltwo$ such that $\l \leq \lone \leq \ltwo \leq \L$, there is
\begin{equation}\label{equ:A_duality}
\begin{aligned}
    \primalAop{k}{\lone}{\ltwo}{\l} \; &\colon \;\; \SDFspace{k}{\lone}{\l} \to \SDFspace{k}{\ltwo}{\l}, \\
    \dualAop{k}{\l}{\lone}{\ltwo} \; &\colon \;\; \SDFspace{k}{\l}{\lone} \to \SDFspace{k}{\l}{\ltwo}.
\end{aligned}
\end{equation}
The `duality' expresses itself in the way that the basis-exchange operator introduces a (transposed) subdivision matrix in Eq.~\eqref{eq:BE_op_introduces_transposed_subdiv_matrix} when we consider it as a map between the ambient spaces $\FEECspace{k}{\lone}$ and $\FEECspace{k}{\ltwo}$.
\end{remark}

\begin{remark}\label{remark_on_lvslprime}
Note that according to Definition~\ref{def:adjoint_subdiv_operator}, a change in basis functions
could in principle work also for the case where $\ltwo \leq \lone$. However, this contrasts the requirement of
the subdivision operator mentioned in Remark~\ref{remark_exclude_l2geql1}.
Therefore, for the duality relation in Eq.~\eqref{equ:A_duality} to hold,
we assume $\lone \leq \ltwo$ for both $ \dualAop{k}{\l}{\lone}{\ltwo}$ and $\primalAop{k}{\lone}{\ltwo}{\l}$
throughout this work.
\end{remark}

\paragraph{Brief summary:}
Although the spaces $\SDFspace{k}{\l}{\L}$ and $\SDFspace{k}{\L}{\l}$ are equivalent,
we will formally distinguish the two spaces as they represent two ways of looking at the same subdivision $k$-form $\constrVec{\omega}{k}{\l}{\L} \equiv \constrVec{\omega}{k}{\L}{\l}$. Namely, the above definitions leave us with two different ways of
describing the same subdivision $k$-form:
\begin{enumerate}
 \item Any $k$-form $\constrVec{\omega}{k}{\L}{\l} \in \SDFspace{k}{\L}{\l}$ is written as
\begin{equation}
    \label{equ:kform_v3}
    \constrVec{\omega}{k}{\L}{\l} = \sum_{\fineIdxone = 1}^{\sumdimlk{L}{k}} \constrQ{c}{k}{\L}{\l}{\fineIdxone} \, \freeQ{\FEECbasis}{k}{\L}{\fineIdxone}
    \quad \text{with}
    \quad \constrQ{c}{k}{\L}{\l}{\fineIdxone} = \sum_{\coarseIdxone=1}^{\sumdimlk{\l}{k}} \Amatij{k}{\l}{\L}{\fineIdxone}{\coarseIdxone} \freeQ{c}{k}{\l}{\coarseIdxone}
\end{equation}
and treated as an object that is represented with respect to the basis of $\FEECspace{k}{L}$ with constrained coefficients. Because of this representation, $\constrVec{\omega}{k}{\L}{\l} \in \SDFspace{k}{\L}{\l}$ has a canonical inclusion $\iota\big(\constrVec{\omega}{k}{\L}{\l}\big)$ into $\FEECspace{k}{L}$. Considering $\iota\big(\constrVec{\omega}{k}{\L}{\l}\big)$ is convenient for everything that involves exterior derivatives and the de Rham complex because it allows us to apply existing knowledge from Finite Element Exterior Calculus \citep{Arnold.FEEC}.

 \item On the other hand, we may consider the same $k$-form $\constrVec{\omega}{k}{\l}{\L}$ as an element of
 space $\SDFspace{k}{\l}{\L}$. We then assume that the $k$-form is represented with respect to the basis $\constrQ{\SDFbasis}{k}{\l}{\L}{\coarseIdxone}$ of $\SDFspace{k}{\l}{\L}$ as
\begin{equation}\label{eq_subdiv_kform}
    \constrVec{\omega}{k}{\l}{\L} = \sum_{\coarseIdxone=1}^{\sumdimlk{\l}{k}} \freeQ{c}{k}{\l}{\coarseIdxone} \, \constrQ{\SDFbasis}{k}{\l}{\L}{\coarseIdxone}
    \quad \text{with}
    \quad \constrQ{\SDFbasis}{k}{\l}{\L}{\coarseIdxone} = \sum_{\fineIdxone = 1}^{\sumdimlk{L}{k}} \Amatij{k}{\l}{\L}{\fineIdxone}{\coarseIdxone} \freeQ{\FEECbasis}{k}{\L}{\fineIdxone} \, .
\end{equation}
The latter point of view is helpful for numerical purposes and whenever we discuss the hierarchical properties of the spaces $\SDFspace{k}{\l}{\L}$.

\end{enumerate}

\noindent In this subsection we have established the foundations of the subdivision $k$-form spaces $\SDFspace{k}{\l}{\L}$ and constructed bases that are suitable for FE simulations. Before we move on to connecting the spaces to establish differential complexes, we want to provide some additional intuition regarding the basis functions of the subdivision $k$-form spaces. The following proposition describes their support.

\begin{proposition}
    \label{prop:support}
    Consider any mesh levels $\l$ and $\lone$ with $\lzero \leq \l < \lone \leq L$. The support of the constrained basis function $\constrQ{\SDFbasis}{k}{\l}{\lone}{\coarseIdxone}$ associated to the simplex $\simplex{k}{\l}{\coarseIdxone} \in \K^k_{\l}$ is given by
    \begin{equation}
        \label{eq:support_of_basis_functions}
        \support\big( \constrQ{\SDFbasis}{k}{\l}{\lone}{\coarseIdxone} \big) = \mathrm{Int}\big( \FFltoL{\l}{\lone} \big( (\simplexsubset{2}{\l})_\coarseIdxone \big) \big) \; \quad \text{with} \quad (\simplexsubset{2}{\l})_\coarseIdxone := \FVl{\l} \circ \VFl{\l} \circ \F\K^k_{\l} \big( \simplex{k}{\l}{\coarseIdxone} \big),
    \end{equation}
    where $\mathrm{Int}$ denotes the interior of a set and where the mesh refinement $ \FFltoL{\l}{\lone}$
    is defined according to Eq.~\eqref{facesplitting_levels}. Note that for any $k \in \{0,1,2\}$,
    $(\simplexsubset{2}{\l})_\coarseIdxone $ is the \emph{face two-ring} around simplex
    $\simplex{k}{\l}{\coarseIdxone}$.
\end{proposition}
\begin{proof}
    See Appendix~\ref{proof:support}.
\end{proof}

\begin{figure}
\sbox\twosubbox{
  \resizebox{\dimexpr.95\textwidth-1em}{!}{%
    \includegraphics[height=0.1cm]{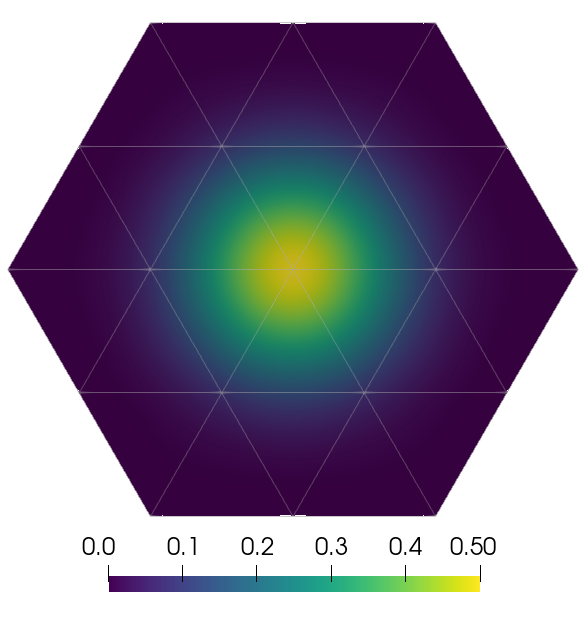}%
    \includegraphics[trim= 0 0 14cm 0, height=0.1cm]{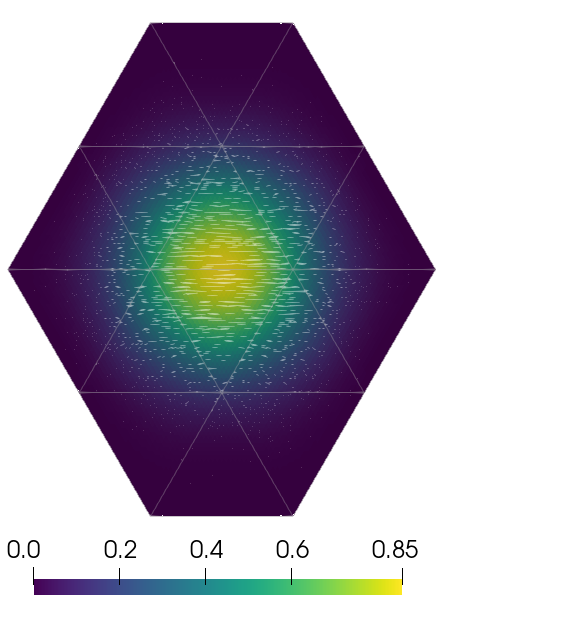}%
    \includegraphics[trim= 0 0 10cm 0, height=0.1cm]{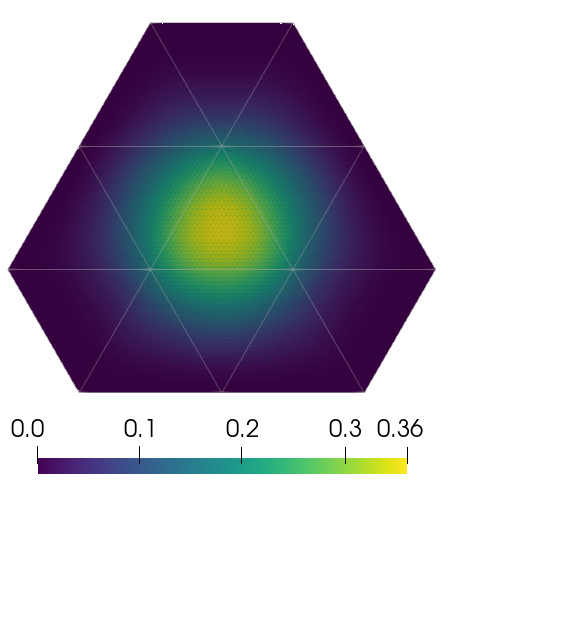}
  }
}
\setlength{\twosubht}{\ht\twosubbox}

\centering

\subcaptionbox{Basis function $\constrQ{\SDFbasis}{0}{0}{5}{\coarseIdxone}$ with support
$\support\big( \constrQ{\SDFbasis}{0}{0}{5}{\coarseIdxone} \big) $.\label{fig:0_form_limit}\bigskip}{
\includegraphics[trim= 0 0 -1cm 0, height=\twosubht]{0_form_basis_function_lvl5_with_mesh.png}
}\quad
\subcaptionbox{Basis function $\constrQ{\SDFbasis}{1}{0}{5}{\coarseIdxone}$ with support
$\support\big( \constrQ{\SDFbasis}{1}{0}{5}{\coarseIdxone} \big) $.\label{fig:1_form_limit}\bigskip}{
  \includegraphics[trim= -1.7cm 0 3cm 0, height=\twosubht]{1_form_basis_function_lvl5_with_mesh_and_glyphs.png}%
}\quad
\subcaptionbox{Basis function $\constrQ{\SDFbasis}{2}{0}{5}{\coarseIdxone}$ with support
$\support\big( \constrQ{\SDFbasis}{2}{0}{5}{\coarseIdxone} \big) $.\label{fig:2_form_limit}}{
\includegraphics[height=\twosubht]{2_form_basis_function_lvl5_with_mesh.png}
}
\caption{The subplots show $0$-, $1$- and $2$-form constrained basis functions according to Definition~\ref{def:alternative_char_for_subdiv_spaces} on level $\l =0$ with $\ns = 5$. Note that the smoothness increases in the sense of Remark~\ref{remark:subdiv_smoothness} despite the basis functions remaining in the same Sobolev space.
The colorbar indicates the scalar values of the $0$- and $2$-form basis functions. For the $1$-form basis function, the colorbar and the white lines indicate the magnitude and direction of the vector field. All basis functions are plotted on their respective supports according to Proposition~\ref{prop:support}.}
\label{fig:basis_function_plots}
\end{figure}

Figure~\ref{fig:basis_function_plots} displays an example of a sequence of
constrained basis functions $\constrQ{\SDFbasis}{k}{0}{5}{\coarseIdxone}$
for $k \in \{0,1,2\}$ according to Definition~\ref{def:alternative_char_for_subdiv_spaces}.
The constrained basis functions $\constrQ{\SDFbasis}{k}{0}{5}{\coarseIdxone}$ are obtained by subdividing the basis function $\freeQ{\FEECbasis}{k}{0}{\coarseIdxone}$ associated to the simplices $\simplex{k}{0}{\coarseIdxone}$ at the centre of the respective plot.

Figures~\ref{fig:0_form_limit} to~\ref{fig:2_form_limit} cover exactly the face two-rings $(\simplexsubset{2}{\l})_\coarseIdxone$ around these centre simplices $\simplex{k}{\l}{\coarseIdxone}$ (see Eq.~\eqref{eq:support_of_basis_functions}) and thus, they show the $k$-form basis functions $\constrQ{\SDFbasis}{k}{0}{5}{\coarseIdxone}$ on their respective supports. Note that we assumed that all vertices in the support of the basis functions are regular, i.e., they have exactly 6 neighbours. The supports would contain more or fewer faces if there were any irregular vertices but the stencil in Eq.~\eqref{eq:support_of_basis_functions} remains valid.

The following remark covers subdivision of constant $k$-forms and the partition of unity property of the basis functions $\constrQ{\SDFbasis}{k}{\l}{\L}{\coarseIdxone}$.

\begin{remark}
  Let $\mathbb{1}: \T_\l \to 1$ denote the constant scalar function on $\T_\l$. For $k \in \{0,2\}$, the spaces $\FEECspace{k}{\l}$ contain $\mathbb{1}$. Scalar $(k=0)$ subdivision schemes are typically designed to preserve constant functions, i.e. $\Aop{0}{\l}{\L} \mathbb{1} = \mathbb{1}$. For the Wang schemes, that is also true for constant $2$-forms. This implies $\mathbb{1} \subset \SDFspace{k}{\l}{\L}$ for $k \in \{0,2\}$. Moreover, since the bases $\freeQ{\FEECbasis}{0}{\l}{\coarseIdxone}$ and $\freeQ{\FEECbasis}{2}{\l}{\coarseIdxone}$ are partitions of unity, we find
  \begin{equation}
    \Aop{k}{\l}{\L} \mathbb{1} = \Aop{k}{\l}{\L} \Big( \sum_{\coarseIdxone = 1}^{\sumdimlk{\l}{k}} \freeQ{\FEECbasis}{k}{\l}{\coarseIdxone} \Big) = \sum_{\coarseIdxone = 1}^{\sumdimlk{\l}{k}} \constrQ{\SDFbasis}{k}{\l}{\L}{\coarseIdxone} = \mathbb{1} \qquad \text{for} \;\; k \in \{0,2\}.
  \end{equation}
  This shows that the basis functions $\constrQ{\SDFbasis}{k}{\l}{\L}{\coarseIdxone}$ form a (positive) partition of unity for $k \in \{0, 2\}$.

  The case of $k=1$ is similar. The space $\FEECspace{1}{\l}$ contains (component-wise) constant $1$-forms or vector fields and the Wang $1$-form subdivision scheme preserves these as well, implying that $\SDFspace{1}{\l}{\L}$ contains constant $1$-forms. However, the bases $\freeQ{\FEECbasis}{1}{\l}{\coarseIdxone}$ and thus $\constrQ{\SDFbasis}{1}{\l}{\L}{\coarseIdxone}$ do not form a partition of unity.
\end{remark}

Finally, we close this section by clarifying the sense in which we consider subdivision as a smoothing operation.

\begin{remark}[Subdivision-induced regularity]
    \label{remark:subdiv_smoothness}
    We emphasize that subdivision-induced regularity does not imply increased Sobolev regularity at any fixed level. The $k$-form subdivision spaces $\SDFspaceZero{k}{\l}{\L} \subset \FEECspace{k}{\L}$ remain subspaces of the standard FEEC spaces for finite $\L$ and do not possess increased Sobolev regularity. Nevertheless, the subdivision operators act as averaging operators across element interfaces because the induced basis functions $\constrQ{\SDFbasis}{k}{\l}{\L}{\coarseIdxone}$ are convex combinations of fine basis functions $\freeQ{\FEECbasis}{k}{\L}{\fineIdxone}$ of $\FEECspace{k}{\L}$. For this reason, repeated subdivision contracts suitable inter-element jump seminorms (e.g. jumps of tangential traces for $k=1$), so that any high-frequency interface mismatches are asymptotically suppressed under refinement.

    We refer to this suppression of inter-element jumps as \emph{subdivision-induced regularity}. This notion reflects improved approximation behaviour and stability properties without implying increased Sobolev smoothness at any fixed subdivision level.
\end{remark}

\begin{remark}[Limit basis functions]
  \label{remark:limit_basis_functions}
  In the following discussions, in particular to interpret our numerical results in Section~\ref{subsec:projection}, we will refer to the \emph{limit basis functions} associated to a given subdivision scheme. Formally, we define the limit basis functions as
  \begin{equation}
    \constrQ{\SDFbasis}{k}{\l}{\infty}{\coarseIdxone} \coloneqq \lim_{\L \to \infty} \constrQ{\SDFbasis}{k}{\l}{\L}{\coarseIdxone}.
  \end{equation}
  We assume that this limit exists. The resulting limit basis functions are genuinely smooth objects. For example, the limit basis functions of Loop subdivision are $\mathcal{C}^2$ everywhere except for at extraordinary vertices, where they are $\mathcal{C}^1$, see \cite{Loop.1987}. The spaces spanned by the limit basis functions are denoted by $\SDFspace{k}{\l}{\infty}$.
\end{remark}

\subsection{Differential complexes of spaces of subdivision $k$-forms}
\label{subsec:compatibility}

The spaces of subdivision $k$-forms derived in the previous sections were induced by independent (i.e. not connected in a sequence) $k$-form subdivision schemes. This section connects these $0$-, $1$- and $2$-form subdivision schemes into a sequence of spaces.
More concretely, below we will derive conditions and operators that ensure that the spaces of subdivision $k$-forms $\SDFspace{k}{\l}{\L}$ constitute a discrete de Rham complex such that the diagram in Figure~\ref{diag:1} commutes.

The derivation of these compatibility conditions is built upon the known discrete de Rham complex of the FEEC spaces $\FEECspace{k}{\l}$ on every level $\l$ in $\lzero\leq\l\leq\L$. To this end, we will introduce an auxiliary derivative operator $\auxextd$ for the subdivision $k$-form spaces which will be related to the weak derivative
$\extd$ of $\FEECspace{k}{\L}$ on the finest level $\L$ using the inclusion map $\auxincl$.

\begin{figure}[t!]
\begin{equation*}
\adjustbox{scale=0.85}{
\begin{tikzcd}[scale = 0.5, column sep = tiny, row sep = small]
    & \FEECspace{0}{\l} \arrow[rrrrr, color=gray, "\extd", dashed]                                                    &                         &                                                                         &  &                                                             & \FEECspace{1}{\l} \arrow[rrrrr, color=gray, "\extd", dashed]                                             &                         &                                                                  &  &                                                             & \FEECspace{2}{\l}                                          &      &                                            \\[4pt]
    & \SDFspace{0}{\l}{\l} \arrow[rrd, "\dualAop{0}{\l}{\l}{\lone}"] \arrow[ldd, "\Aop{0}{\l}{\lone}"] \arrow[u, leftrightarrow, "\equiv"'] \arrow[rrrrr, color=gray, "\Tilde{\extd}", dashed] &                         &                                                                         &  &                                                             & \SDFspace{1}{\l}{\l} \arrow[ldd, "\Aop{1}{\l}{\L}"] \arrow[rrd, "\dualAop{1}{\l}{\l}{\lone}"] \arrow[u, leftrightarrow, "\equiv"'] \arrow[rrrrr, color=gray, "\Tilde{\extd}", dashed] &                         &                                                                  &  &                                                             & \SDFspace{2}{\l}{\l} \arrow[ldd, "\Aop{2}{\l}{\lone}"] \arrow[rrd, "\dualAop{2}{\l}{\l}{\lone}"] \arrow[u, leftrightarrow, "\equiv"'] &      &                                            \\
    &                                                                          &                         & \SDFspace{0}{\l}{\lone} \arrow[ldd, "\auxincl", hook] \arrow[ddd, "{\dualAop{0}{\l}{\l}{\L}}"] \arrow[rrrrr, color=gray, "\Tilde{\extd}" near start, dashed] &  &                                                             &                                                                   &                         & \SDFspace{1}{\l}{\lone} \arrow[ldd, "\auxincl", hook] \arrow[ddd, "{\dualAop{1}{\l}{\l}{\L}}"] \arrow[rrrrr, color=gray, "\Tilde{\extd}" near start, dashed] &  &                                                             &                                             &      & \SDFspace{2}{\l}{\lone} \arrow[ldd, "\auxincl", hook] \arrow[ddd, "\dualAop{2}{\l}{\l}{\L}"] \\
\SDFspace{0}{\lone}{\l} \arrow[rrd, "\iota", hook] \arrow[ddd, "\Aop{0}{\lone}{\L}"] \arrow[rrru, leftrightarrow, "\equiv"'] &                                                                          &                         &                                                                         &  & \SDFspace{1}{\lone}{\l} \arrow[rrd, "\iota", hook] \arrow[ddd, "\Aop{1}{\lone}{\L}"] \arrow[rrru, leftrightarrow, "\equiv"'] &                                                                   &                         &                                                                  &  & \SDFspace{2}{\lone}{\l} \arrow[rrd, "\iota", hook] \arrow[ddd, "\Aop{2}{\lone}{\L}"] \arrow[rrru, leftrightarrow, "\equiv"'] &                                             &      &                                            \\
    &                                                                          & \FEECspace{0}{\lone} \arrow[rrrrr, color=gray, "\extd"' near start, dashed] &                                                                         &  &                                                             &                                                                   & \FEECspace{1}{\lone} \arrow[rrrrr, color=gray, "\extd"' near start, dashed] &                                                                  &  &                                                             &                                             & \FEECspace{2}{\lone} &                                            \\
    &                                                                          &                         & \SDFspace{0}{\l}{\L} \arrow[ldd, "\auxincl", hook] \arrow[rrrrr, color=gray, "\Tilde{\extd}", dashed]               &  &                                                             &                                                                   &                         & \SDFspace{1}{\l}{\L} \arrow[ldd, "\auxincl", hook] \arrow[rrrrr, color=gray, "\Tilde{\extd}", dashed]               &  &                                                             &                                             &      & \SDFspace{2}{\l}{\L} \arrow[ldd, "\auxincl", hook]               \\
\SDFspace{0}{\L}{\l} \arrow[rrd, "\iota", hook] \arrow[rrru, leftrightarrow, "\equiv"']               &                                                                          &                         &                                                                         &  & \SDFspace{1}{\L}{\l} \arrow[rrd, "\iota", hook] \arrow[rrru, leftrightarrow, "\equiv"']               &                                                                   &                         &                                                                  &  & \SDFspace{2}{\L}{\l} \arrow[rrd, "\iota", hook] \arrow[rrru, leftrightarrow, "\equiv"']               &                                             &      &                                            \\
    &                                                                          & \FEECspace{0}{\L} \arrow[rrrrr, color=gray, "\extd", dashed]   &                                                                         &  &                                                             &                                                                   & \FEECspace{1}{\L} \arrow[rrrrr, color=gray, "\extd", dashed]   &                                                                  &  &                                                             &                                             & \FEECspace{2}{\L}   &
\end{tikzcd}
}
\end{equation*}
\caption{Full commutative diagram of the spaces of subdivision $k$-forms including the subdivision, inclusion, and derivative operators.
The black lines connect the individual hierarchies of subdivision $k$-form spaces for $k = 0$ (left), $k=1$ (centre) and $k=2$ (right). The derivative operators $\extd$ and $\auxextd$ (dashed grey lines) join the spaces of the hierarchies to form de Rham complexes.}
\label{diag:1}
\end{figure}
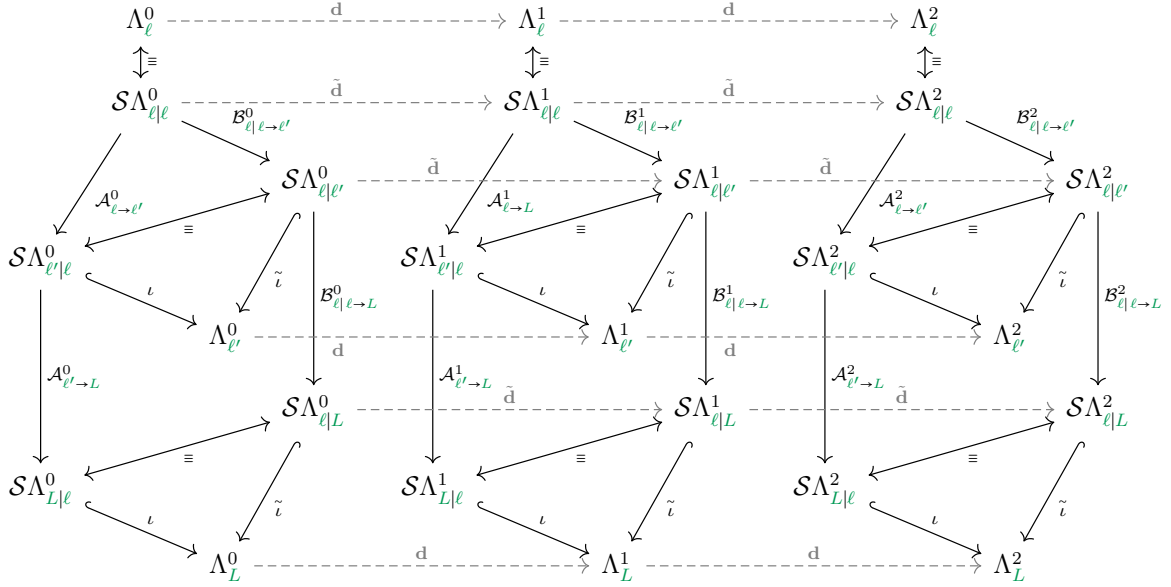

\subsubsection*{Matrix representation of the exterior derivative $\extd$}

Let us first consider the spaces $\FEECspace{k}{\l}$ for any $\l \in \{ \lzero, \dots, \L\}$ while recalling that we use
lowest order FEEC spaces (cf. Eq.~\eqref{eq:choice_of_feec_space}).
The fact that these spaces form a differential complex implies that there exists a
matrix $\Dmat{k}{\l}$ (the coordinate expression of $\extd$) that acts on the basis functions
of $\FEECspace{k}{\l}$ as
\begin{equation}
    \label{eq:derivative_matrix_introduced}
    \extd \freeQ{\FEECbasis}{k}{\l}{\coarseIdxtwo} = \sum_{\coarseIdxone = 1}^{ {\sumdimlk{\l}{k+1}  }}
    \Dmatij{k}{\l}{\coarseIdxone}{\coarseIdxtwo} \; \freeQ{\FEECbasis}{k+1}{\l}{\coarseIdxone} \qquad \forall \coarseIdxtwo = 1, \dots, \sumdimlk{\l}{k} \, ,
\end{equation}
where $\sumdimlk{\l}{k}$ and $\sumdimlk{\l}{k+1}$ coincide with the dimensions of the spaces $\FEECspace{k}{\l}$ and $\FEECspace{k+1}{\l}$, respectively.
Then, the derivative of $\freeVec{\omega}{k}{\l} \in \FEECspace{k}{\l}$ is given by
\begin{equation}   \label{eq:exterior_deriv_in_coordinates}
\begin{aligned}
    \extd \freeVec{\omega}{k}{\l}
    & =  \extd \Big( \sum_{\coarseIdxtwo = 1}^{ { \sumdimlk{\l}{k}     }} \freeQ{c}{k}{\l}{\coarseIdxtwo} \;  \freeQ{\FEECbasis}{k}{\l}{\coarseIdxtwo} \Big) = \sum_{\coarseIdxtwo = 1}^{ {\sumdimlk{\l}{k}  }} \; \freeQ{c}{k}{\l}{\coarseIdxtwo} \; \extd \freeQ{\FEECbasis}{k}{\l}{\coarseIdxtwo} \\
    & = \sum_{\coarseIdxone = 1}^{ { \sumdimlk{\l}{k+1}  }} \Big( \sum_{\coarseIdxtwo = 1}^{ { \sumdimlk{\l}{k}  }} \Dmatij{k}{\l}{\coarseIdxone}{\coarseIdxtwo} \; \freeQ{c}{k}{\l}{\coarseIdxtwo} \Big) \freeQ{\FEECbasis}{k+1}{\l}{\coarseIdxone}.
\end{aligned}
\end{equation}
With the particular choice of spaces $\FEECspace{k}{\l}$ made in Eq.~\eqref{eq:choice_of_feec_space}, we can explicitly compute the matrix $\Dmat{k}{\l}$ to find
\begin{equation}\label{Dmatrix}
    \Dmatij{k}{\l}{\coarseIdxtwo}{\coarseIdxone} =
    \begin{cases}
            &1, \;\; \text{if} \; \simplex{k}{\l}{\coarseIdxone} \in \K^{k} \K^{k+1}_{\l} \big( \simplex{k+1}{\l}{\coarseIdxtwo} \big) \; \text{and their orientations agree}, \\
            - \!\!\!\!\!\! &1, \;\; \text{if} \; \simplex{k}{\l}{\coarseIdxone} \in \K^{k} \K^{k+1}_{\l} \big( \simplex{k+1}{\l}{\coarseIdxtwo} \big) \; \text{and their orientations are opposite}, \\
              &0, \;\; \text{otherwise},
    \end{cases}
\end{equation}
where $ \Dmatij{k}{\l}{\coarseIdxtwo}{\coarseIdxone}$ is the transpose of  $ \Dmatij{k}{\l}{\coarseIdxone}{\coarseIdxtwo}$.
We will not go into detail about simplex orientations. For our purposes, it is enough to consider them a mechanism that ensures the alignment of two simplices such that Eq.~\eqref{eq:derivative_matrix_introduced} holds.

Next, we will precisely define the notion of compatibility in our context of subdivision $k$-forms
and break these resulting abstract compatibility requirements down to the level of coordinate representations.
This leads to concrete conditions that the subdivision matrices have to satisfy to achieve compatibility.

\subsubsection*{Definition of auxiliary exterior derivative $\auxextd$}

The following definition introduces an auxiliary derivative operator that acts on the spaces of subdivision $k$-forms using the inclusion map $\auxincl : \SDFspace{k}{\lone}{\ltwo} \to \FEECspace{k}{\ltwo}$ of Definition~\ref{def:inclusion_map}.

\begin{definition}[Auxiliary exterior derivative]
    \label{def:d_tilde}
    The auxiliary derivative operator $\auxextd$ is a map $\auxextd: \SDFspace{k}{\lone}{\ltwo} \to \SDFspace{k+1}{\lone}{\ltwo}$ that satisfies
    \begin{alignat}{2}
        &(i) \quad \auxincl \circ \auxextd = \extd \circ \auxincl && \quad \colon \;\; \SDFspace{k}{\lone}{\ltwo} \to \FEECspace{k+1}{\ltwo} \\
        &(ii) \quad \auxextd \circ \dualAop{k}{\l}{\lone}{\ltwo} = \dualAop{k+1}{\l}{\lone}{\ltwo} \circ \auxextd && \quad \colon \;\; \SDFspace{k}{\l}{\lone} \to \SDFspace{k+1}{\l}{\ltwo}
    \end{alignat}
    for $\l \leq \lone \leq \ltwo \leq \L$ and $k = 0,1$.
\end{definition}

The following lemma relates the auxiliary derivative $\auxextd$ to the exterior derivative $\extd$ across the mesh hierarchy.

\begin{lemma}
    \label{lemma:d_tilde_commuted_with_d}
    Let $\l \leq \lone \leq \ltwo \leq \L$ for any intermediate levels $\lone$ and $\ltwo$. Definition~\ref{def:d_tilde} is equivalent to both
    \begin{alignat}{4}
        &(i) \quad  &&\auxextd \circ \dualAop{k}{\lone}{{\lone}}{\ltwo} &&= \dualAop{k+1}{\lone}{{\lone}}{\ltwo} \circ \extd &&\quad
        \;\; \quad (\SDFspace{k}{\lone}{\lone} \equiv \FEECspace{k}{\lone} \to \SDFspace{k+1}{{\lone}}{\ltwo}), \\
        &(ii) &&\extd \circ \auxincl \circ \dualAop{k}{\lone}{{\lone}}{\ltwo} &&= \auxincl \circ \dualAop{k+1}{\lone}{{\lone}}{\ltwo} \circ \extd &&\quad \;\; \quad (\SDFspace{k}{\lone}{\lone} \equiv \FEECspace{k}{\lone} \to \FEECspace{k+1}{\ltwo}).
    \end{alignat}
\end{lemma}

\begin{proof}
    The claim $(i)$ follows from Definition~\ref{def:d_tilde} by
    \begin{equation}
        \auxextd \big(\constrVec{\omega}{k}{\lone}{\ltwo}) = \auxextd \circ\dualAop{k}{\lone}{\lone}{\ltwo} \big( \constrVec{\omega}{k}{\lone}{\lone} \big) = \dualAop{k+1}{\lone}{\lone}{\ltwo} \circ \auxextd \big( \constrVec{\omega}{k}{\lone}{\lone} \big) = \dualAop{k+1}{\lone}{\lone}{\ltwo} \circ \extd \big( \constrVec{\omega}{k}{\lone}{\lone} \big),
    \end{equation}
   recalling the action of the BE operator from Definition~\ref{def:adjoint_subdiv_operator}
   and using the fact that $\auxincl = \mathrm{Id}$ when acting on $\SDFspace{k}{\lone}{\lone} \equiv \FEECspace{k}{\lone}$ and thus $\auxextd \equiv \extd$ in that case.
   Claim $(ii)$ follows by inclusion into $\FEECspace{k}{\ltwo}$, that is
    \begin{equation}
        \auxincl \circ \dualAop{k+1}{\lone}{{\lone}}{\ltwo} \circ \extd = \auxincl \circ \auxextd \circ \dualAop{k}{\lone}{{\lone}}{\ltwo} = \extd \circ \auxincl \circ \dualAop{k}{\lone}{{\lone}}{\ltwo} \, ,
    \end{equation}
    using property $(i)$ in the first equality and property (i) of Definition~\ref{def:d_tilde} in the second.
\end{proof}

The definition of $\auxextd$ in Definition~\ref{def:d_tilde} in combination with Lemma~\ref{lemma:d_tilde_commuted_with_d} implies that the following diagram commutes:
\begin{equation}
\begin{tikzcd}[column sep = large]
    \FEECspace{k}{\lone} \arrow[r, leftrightarrow, "\equiv"] \arrow[d, "\extd"] & \SDFspace{k}{\lone}{\lone} \arrow[r, "\dualAop{k}{\lone}{\lone}{\ltwo}"] \arrow[d, "\auxextd"] & \SDFspace{k}{\lone}{\ltwo} \arrow[r, "\auxincl"] \arrow[d, "\auxextd"] & \FEECspace{k}{\ltwo} \arrow[d, "\extd"] \\
    \FEECspace{k+1}{\lone} \arrow[r, leftrightarrow, "\equiv"] & \SDFspace{k+1}{\lone}{\lone} \arrow[r, "\dualAop{k+1}{\lone}{\lone}{\ltwo}"] & \SDFspace{k+1}{\lone}{\ltwo} \arrow[r, "\auxincl"] & \FEECspace{k+1}{\ltwo}
\end{tikzcd}
\end{equation}
Note that the two independent commutation relations (i) and (ii) of Definition~\ref{def:d_tilde}, together with
the fact that $\auxincl$ of Definition~\ref{def:tilde_inclusion_map} is the identity map on the initial level $\lone$, define the auxiliary exterior derivative $\auxextd$ in Definition~\ref{def:d_tilde} in terms of $\extd$.

Using this auxiliary derivative, the following definition specifies how to set up compatible $k$-form subdivision schemes such that the induced subdivision $k$-form spaces constitute a discrete de
Rham complex.

\begin{definition}[Compatible $k$-form subdivision schemes]
    \label{def:compatible_subdiv_schemes}
    Let $\primalAop{k}{\lone}{\ltwo}{\l}$, $k\in \{0,1,2\}$, be the subdivision operators for the $k$-form subdivision schemes with basis-exchange operators $\dualAop{k}{\l}{\lone}{\ltwo}$ as in Definition~\ref{def:adjoint_subdiv_operator}.
    Then, we call $k$-form subdivision schemes \emph{compatible} iff the spaces $\SDFspace{k}{\lone}{\ltwo}$ induced through
    $\primalAop{k}{\lone}{\ltwo}{\l}$ comprise a discrete de Rham complex, i.e.
    \begin{equation}
    \begin{tikzcd}[column sep = normal]
        0 \arrow[r] & \SDFspace{0}{\lone}{\ltwo} \arrow[r, "\auxextd"] & \SDFspace{1}{\lone}{\ltwo} \arrow[r, "\auxextd"] & \SDFspace{2}{\lone}{\ltwo} \arrow[r] & 0 \, ,
    \end{tikzcd}
    \end{equation}
    with $\auxextd \circ \auxextd = 0$ on all subdivision levels $\lone$ with $\l \leq \lone \leq \ltwo$.
\end{definition}

 The following theorem shows that the property $\auxextd \circ \auxextd = 0$ required in Definition~\ref{def:compatible_subdiv_schemes} follows from $\extd \circ \extd = 0$ for the FEEC spaces $\FEECspace{k}{\lone}$. Together with the fact that $\dualAop{k}{\lone}{\lone}{\ltwo}$ is a cochain isomorphism, this implies that the subdivision $k$-form spaces $\SDFspace{k}{\lone}{\ltwo}$ constitute discrete de Rham complexes.

\begin{theorem}[Compatibility condition]
\label{theo:uniform_de_rham_complex}
For the auxiliary exterior derivative $\auxextd$ of Definition~\ref{def:d_tilde},
the subdivision $k$-form spaces $\SDFspace{k}{\lone}{\ltwo}$ comprise a discrete de Rham complex and thus, the subdivision schemes are compatible in the sense of Definition \ref{def:compatible_subdiv_schemes}.
\end{theorem}
\begin{proof}
    By Definition~\ref{def:d_tilde}, $\auxextd:\SDFspace{k}{\lone}{\ltwo} \to \SDFspace{k+1}{\lone}{\ltwo}$. For arbitrary intermediate levels $\lone$ and $\ltwo$, the auxiliary derivative $\auxextd$ satisfies
    \begin{equation}
        \auxextd \auxextd \constrVec{\omega}{0}{\lone}{\ltwo} = \auxextd \auxextd \dualAop{0}{\lone}{\lone}{\ltwo} \constrVec{\omega}{0}{\lone}{\lone} = \auxextd \dualAop{1}{\lone}{\lone}{\ltwo} \extd \constrVec{\omega}{0}{\lone}{\lone} = \dualAop{2}{\lone}{\lone}{\ltwo}\extd \extd \constrVec{\omega}{0}{\lone}{\lone} = 0
    \end{equation}
    using Lemma~\ref{lemma:d_tilde_commuted_with_d} and the fact that $(\FEECspace{k}{\lone}, \extd)$ is a differential complex.
    Further, the basis-exchange operators $\dualAop{k}{\lone}{\lone}{\ltwo}$ are isomorphisms between
    $\FEECspace{k}{\lone}$ and $\SDFspace{k}{\lone}{\ltwo}$ according to Cor.~\ref{cor:BE_op_is_isomorphism} and commute with the exterior derivative $\extd$, see Lemma~\ref{lemma:d_tilde_commuted_with_d}.
    Consequently, $\dualAop{k}{\lone}{\lone}{\ltwo}$ is a cochain isomorphism and induces an isomorphism of the cohomologies of the FEEC spaces $\FEECspace{k}{\lone}$ and the subdivision $k$-form spaces $\SDFspace{k}{\lone}{\ltwo}$. Thus, the spaces $\SDFspace{k}{\lone}{\ltwo}$ form a discrete de Rham complex under the derivative operator $\auxextd$.
\end{proof}

\subsubsection*{Matrix representations of auxiliary exterior derivative $\auxextd$}

Having characterised the auxiliary exterior derivative $\auxextd$ in terms of the subdivision operators $\primalAop{k}{\lone}{\ltwo}{\lone}$
and the exterior derivative $\extd$, we can specify the action of $\auxextd$ on an element of $\SDFspace{k}{\lone}{\ltwo}$ directly.
\begin{lemma}
    \label{lemma:d_tilde_action_in_coordinates}
    The auxiliary exterior derivative $\auxextd$ acts on $\constrVec{\omega}{k}{\lone}{\ltwo} \in \SDFspace{k}{\lone}{\ltwo}$ by
    \begin{equation}\notag
      \auxextd:  \constrVec{\omega}{k}{\lone}{\ltwo} = \sum_{\coarseIdxtwo = 1}^{\sumdimlk{\lone}{k}} \freeQ{c}{k}{\lone}{\coarseIdxtwo} \; \constrQ{\SDFbasis}{k}{\lone}{\ltwo}{\coarseIdxtwo}  \mapsto \auxextd \constrVec{\omega}{k}{\lone}{\ltwo} = \sum_{\coarseIdxone = 1}^{\sumdimlk{\lone}{k+1}} \Big( \sum_{\coarseIdxtwo = 1}^{\sumdimlk{\lone}{k}} \Dmatij{k}{\lone}{\coarseIdxone}{\coarseIdxtwo} \; \freeQ{c}{k}{\lone}{\coarseIdxtwo} \Big) \constrQ{\SDFbasis}{k+1}{\lone}{\ltwo}{\coarseIdxone}
    \end{equation}
    using the matrix representation of $\extd$ as specified in Eq.~\eqref{Dmatrix}.
\end{lemma}
\begin{proof}
    Starting from Lemma~\ref{lemma:d_tilde_commuted_with_d}, direct computation yields
    \begin{equation}
    \begin{aligned}
        \label{eq:d_tilde_in_coordinates}
        \auxextd \constrVec{\omega}{k}{\lone}{\ltwo} &= \auxextd \dualAop{k}{\lone}{\lone}{\ltwo} \constrVec{\omega}{k}{\lone}{\lone} = \dualAop{k+1}{\lone}{\lone}{\ltwo} \extd \constrVec{\omega}{k}{\lone}{\lone} = \dualAop{k+1}{\lone}{\lone}{\ltwo} \extd \big( \sum_{\coarseIdxtwo = 1}^{\sumdimlk{\lone}{k}} \freeQ{c}{k}{\lone}{\coarseIdxtwo} \; \freeQ{\FEECbasis}{k}{\lone}{\coarseIdxtwo} \big) \\
        &= \dualAop{k+1}{\lone}{\lone}{\ltwo} \Big( \sum_{\coarseIdxone = 1}^{\sumdimlk{\lone}{k+1}} \big( \sum_{\coarseIdxtwo = 1}^{\sumdimlk{\lone}{k}} \Dmatij{k}{\lone}{\coarseIdxone}{\coarseIdxtwo} \; \freeQ{c}{k}{\lone}{\coarseIdxtwo} \big) \; \freeQ{\FEECbasis}{k+1}{\lone}{\coarseIdxone} \Big) \\
        &= \sum_{\coarseIdxone = 1}^{\sumdimlk{\lone}{k+1}} \big( \sum_{\coarseIdxtwo = 1}^{\sumdimlk{\lone}{k}} \Dmatij{k}{\lone}{\coarseIdxone}{\coarseIdxtwo} \; \freeQ{c}{k}{\lone}{\coarseIdxtwo} \big) \; \constrQ{\SDFbasis}{k+1}{\lone}{\ltwo}{\coarseIdxone} \;,
    \end{aligned}
    \end{equation}
    where we used $\constrVec{\omega}{k}{\lone}{\lone} \in \SDFspace{k}{\lone}{\lone}\equiv \FEECspace{k}{\lone}$,
    the linearity of $\dualAop{k+1}{\lone}{\lone}{\ltwo}$, and
    the matrix representation of $\extd$ in Eq.~\eqref{eq:exterior_deriv_in_coordinates}. Since $\freeQ{\FEECbasis}{k+1}{\lone}{\coarseIdxone} = \constrQ{\SDFbasis}{k+1}{\lone}{\lone}{\coarseIdxone}$, the last equality follows from Definition~\ref{def:adjoint_subdiv_operator}.
\end{proof}

\begin{remark}
Observe the similarity in how the exterior derivative $\extd$ in Eq.~\eqref{eq:exterior_deriv_in_coordinates}
acts on $k$-forms $\freeVec{\omega}{k}{\lone} \in \FEECspace{k}{\lone}$ and how
the auxiliary exterior derivative $\auxextd$ in Eq.~\eqref{eq:d_tilde_in_coordinates} acts on subdivision $k$-forms
$\constrVec{\omega}{k}{\lone}{\ltwo} \in \SDFspace{k}{\lone}{\ltwo}$.
That is, in both cases the same matrix $\Dmat{k}{\lone}$ acts on the coefficients $ \freeQ{c}{k}{\lone}{\coarseIdxtwo} $
while only the basis functions are different.
\end{remark}

The following final lemma of this section breaks down the abstract commutativity requirements in Definition~\ref{def:d_tilde}
into easily verifiable conditions on the coordinate level in terms of the matrix representations of the subdivision
and derivative operators.

\begin{lemma}
    \label{lemma:compatibility_condition_in_coordinates}
    Condition (ii) of Lemma~\ref{lemma:d_tilde_commuted_with_d} for arbitrary $\lone$ with $\l \leq \lone \leq \ltwo$, i.e.
    \begin{equation}
        \extd \circ \auxincl \circ \dualAop{k}{\lone}{\lone}{\ltwo} = \auxincl \circ \dualAop{k+1}{\l}{\lone}{\ltwo} \circ \extd \;\; \qquad  \;\; (\SDFspace{k}{\lone}{\lone} \equiv \FEECspace{k}{\lone} \to \FEECspace{k+1}{\ltwo}) \ \ \ \text{for} \ \ \  k = 0,1,
    \end{equation}
    implies that the subdivision matrices have to satisfy
    \begin{equation}
        \label{eq:subdiv_matrix_commutativity}
        \Dmat{k}{\ltwo} \cdot \Amat{k}{\lone}{\ltwo} = \Amat{k+1}{\lone}{\ltwo} \cdot \Dmat{k}{\lone}, \qquad \text{for} \quad k = 0,1,
    \end{equation}
    where $\Dmat{k}{\lone}$ and $\Dmat{k}{\ltwo}$ are the respective matrix representations of $\extd$ on levels $\lone$ and $\ltwo$ according to Eq.~\eqref{eq:derivative_matrix_introduced}.
\end{lemma}

\begin{proof}
    Let $\freeVec{\omega}{k}{\lone} = \sum_{\coarseIdxone = 1}^{\sumdimlk{\lone}{k}} \freeQ{c}{k}{\lone}{\coarseIdxone} \freeQ{\FEECbasis}{k}{\lone}{\coarseIdxone} \in \FEECspace{k}{\lone}$. Using Eq.~\eqref{eq:exterior_deriv_in_coordinates}, we can compute
    \begin{alignat}{3}
        \big(\extd \circ \auxincl \circ \dualAop{k}{\lone}{\lone}{\ltwo}\big) \, \freeVec{\omega}{k}{\lone} &= \sum_{\fineIdxtwo=1}^{\sumdimlk{\ltwo}{k+1}} \Big( &&\;\,\sum_{\fineIdxone=1}^{\sumdimlk{\ltwo}{k}} &&\sum_{\coarseIdxone=1}^{\sumdimlk{\lone}{k}} \Dmatij{k}{\ltwo}{\fineIdxtwo}{\fineIdxone} \, \Amatij{k}{\lone}{\ltwo}{\fineIdxone}{\coarseIdxone} \freeQ{c}{k}{\lone}{\coarseIdxone} \Big) \, \freeQ{\FEECbasis}{k+1}{\ltwo}{\fineIdxtwo}, \\
        \big(\auxincl \circ \dualAop{k+1}{\lone}{\lone}{\ltwo} \circ \extd \, \freeVec{\omega}{k}{\lone}\big) &= \sum_{\fineIdxtwo=1}^{\sumdimlk{\ltwo}{k+1}} \Big( &&\;\,\sum_{\coarseIdxtwo=1}^{\sumdimlk{\lone}{k+1}} &&\sum_{\coarseIdxone=1}^{\sumdimlk{\lone}{k}} \Amatij{k+1}{\lone}{\ltwo}{\fineIdxtwo}{\coarseIdxtwo} \Dmatij{k}{\lone}{\coarseIdxtwo}{\coarseIdxone} \,  \freeQ{c}{k}{\lone}{\coarseIdxone} \Big) \, \freeQ{\FEECbasis}{k+1}{\ltwo}{\fineIdxtwo}.
    \end{alignat}
    Hence, to guarantee $\big(\auxincl \circ \dualAop{k+1}{\lone}{\lone}{\ltwo} \circ \extd \big) \, \freeVec{\omega}{k}{\lone} =  \big(\extd \circ \auxincl \circ \dualAop{k}{\lone}{\lone}{\ltwo}\big) \, \freeVec{\omega}{k}{\lone}$, the subdivision matrices have to satisfy $\Dmat{k}{\ltwo} \cdot \Amat{k}{\lone}{\ltwo} = \Amat{k+1}{\lone}{\ltwo} \cdot \Dmat{k}{\lone}$ for $k = 0, 1$.
\end{proof}

There exist a few subdivision schemes that satisfy this requirement (see Remark~\ref{remark_othersubdivschemes}). We choose to adopt the subdivision schemes introduced in \cite{Wang.2006, Wang.2008} for triangular meshes. The authors devised subdivision schemes that satisfy Eq.~\eqref{eq:subdiv_matrix_commutativity} by fixing a subdivision scheme for $0$-forms and $2$-forms and used the commutation relation in Eq.~\eqref{eq:subdiv_matrix_commutativity} to design their $1$-form subdivision schemes. Besides these algebraic constraints on the subdivision weights, they also add certain symmetry constraints to make the $1$-form weights unique.

Whenever we consider a concrete subdivision scheme from now on and unless stated otherwise, we will consider Wang's compatible subdivision schemes \citep{Wang.2006, Wang.2008} for triangle meshes, because they fit into the framework established in this work without any need for modifications. We will refer to them as \emph{Wang schemes} in the following.

\begin{remark}\label{remark_othersubdivschemes}
Besides the Wang schemes, other subdivision schemes also satisfy the requirement of Lemma~\ref{lemma:compatibility_condition_in_coordinates}.
For example, \cite{Wang.2006, Wang.2008} also derived one for quadrilateral meshes where they start from the Catmull-Clark \citep{Catmull.1978} and the Doo-Sabin \citep{Doo.78} subdivision schemes.
As in the triangle case, they find commuting $1$-form schemes. The resulting schemes are comprehensively summarised in the supplemental material of \cite{Goes.2016}.

Alternatively, \cite{Huang2012} introduces a $1$-form subdivision scheme based on $\sqrt{3}$-subdivision \citep{Kobbelt.2000} for $0$-forms. Their $1$-form scheme is again determined from the commutation relation but does not lead to a unique $2$-form scheme a priori. However, their $2$-form scheme can be made unique by enforcing additional properties for the spectrum of the $2$-form subdivision matrix. Finally, \cite{Custers.2020} developed a $1$-form subdivision scheme based on a half-edge representation of $1$-forms. Their work is also based on the idea of commuting subdivision and derivative operators but employs the concept of a dual mesh.

There also exist ``trivial'' subdivision schemes with respect to lowest-order finite elements, see \cite{Wang.2008}. These schemes are called the Whitney subdivision schemes and they exactly reproduce the Whitney forms, i.e. the basis functions $\freeVec{\FEECbases}{k}{\lone}$ as defined in Eq.~\eqref{eq:FEECbases_definition}. This implies that they do not alter the mesh boundary during subdivision. For this reason, we will later use the Whitney subdivision schemes for purposes of domain approximation, see Section~\ref{subsec:domain_approximation} and Appendix~\ref{app:domain_approximation}.
\end{remark}

This closes our discussion regarding the compatibility of subdivision schemes and their induced function spaces. The essential finding is a derivative operator $\auxextd: \SDFspace{k}{\lone}{\ltwo} \to \SDFspace{k+1}{\lone}{\ltwo}$ that is consistent with the weak (exterior) derivative $\extd$ of FE methods if the subdivision matrices satisfy the commutation relation in Lemma~\ref{lemma:compatibility_condition_in_coordinates}.
Additionally, Theorem~\ref{theo:uniform_de_rham_complex} shows that the spaces of subdivision $k$-forms
induced by the subdivision operators comprise de Rham complexes.

\subsection{Subdivision k-form spaces with zero boundary conditions}
\label{subsec:complex_with_vanishing_trace}

This section presents subdivision $k$-form spaces with vanishing trace boundary conditions (BC). This section follows the most important steps of Appendix~\ref{app:vanishing_boundary_complex}, culminating in the main theorem that proves the preservation of the relative de Rham complex, cf. Eq.~\eqref{eq:infinite_dim_relative_de_Rham}.

The construction is analogous to the one we used for the complex $\big( \SDFspace{k}{\l}{\L}, \auxextd \big)$. Given that the chosen subdivision schemes behave well close to the boundary in the sense of Eq.~\eqref{eq:subdiv_preserves_vanishing_trace} in Appendix~\ref{app:vanishing_boundary_complex}, the only basis functions with non-zero trace are the ones associated to boundary vertices and edges of the mesh $\T_\l$. Thus, we can discard these basis functions to guarantee that the spaces satisfy the BC. The resulting spaces are then isomorphic to the analogous FEEC spaces and thus comprise a de Rham complex.

Propositions~\ref{prop:no_support_at_bdry} and~\ref{prop:Wang_subdiv_preserves_zero_trace} show that basis functions associated to interior simplices have vanishing traces. Denoting the sets of interior vertices $\Vint_{\l}$, edges $\Eint_{\l}$ and faces $\Fint_{\l}$ of the mesh $\T_{\l}$ by
\begin{equation}
    \label{def:interior_simplices}
    \Vint_{\l} \coloneq \V_{\l} \setminus \V (\partial \T_{\l}), \qquad \Eint_{\l} \coloneq \E_{\l} \setminus \E (\partial \T_{\l}), \qquad \Fint_\l = \F_\l
\end{equation}
lets us define the subdivision $k$-form spaces in the following way.
\begin{definition}
    \label{def:subdiv_spaces_with_vanishing_trace}
    The subdivision $k$-form spaces $\SDFspaceZero{k}{\l}{\L}$ with vanishing trace BC are defined as the span of all basis functions $\constrQ{\SDFbasis}{k}{\l}{\L}{\coarseIdxone}$ associated to non-boundary simplices $\simplex{k}{\l}{\coarseIdxone}$, i.e.,
    \begin{equation}
        \SDFspaceZero{k}{\l}{\L} \big(\T_\l\big) \coloneq \mathrm{span} \big( \constrVec{\SDFbasesZero}{k}{\l}{\L}\big) \qquad \text{with} \qquad \constrVec{\SDFbasesZero}{k}{\l}{\L} \coloneq \big\{ \constrQ{\SDFbasis }{k}{\l}{\L}{\coarseIdxone} \; : \; \simplex{k}{\l}{\coarseIdxone} \in \Kint_{\l} \big\}.
    \end{equation}
\end{definition}

Upon defining appropriate subdivision and basis exchange operators, Appendix~\ref{app:vanishing_boundary_complex} shows that these spaces are isomorphic to the lowest-order FEEC spaces with vanishing traces. This culminates in the following main theorem of the appendix.

\begin{theorem}[Complex of subdivision $k$-forms with vanishing trace BC]
    \label{theo:complex_with_vanishing_trace}
    The spaces $\SDFspaceZero{k}{\l}{\L}$ of subdivision $k$-forms with vanishing trace BC constitute the discrete differential complex
    \begin{equation}
        \label{eq:complex_of_trunc_spaces}
        \begin{tikzcd}
            0 \arrow[r] & \SDFspaceZero{0}{\l}{\L} \arrow[r, "\auxextd"] & \SDFspaceZero{1}{\l}{\L} \arrow[r, "\auxextd"] & \SDFspaceZero{2}{\l}{\L} \arrow[r] & 0.
        \end{tikzcd}
    \end{equation}
    The complex has the same cohomology as the relative de Rham complex.
\end{theorem}
\begin{proof}
    See \ref{proof:complex_with_vanishing_trace}.
\end{proof}

The previous theorem is the central statement of this section. We will confirm this result numerically in Section~\ref{subsec:maxwell_EV_simulation} by the absence of spurious eigenvalues in the Maxwell eigenvalue problem.

\section{Implementation of spaces of subdivision $k$-forms}
\label{sec:implementation}

This section provides details on how to practically implement the subdivision $k$-form spaces as well as the
derivative operators. First, we suggest a method to approximate the physical domain since in case of mesh refinement via subdivision, this
is a nontrivial task. Then, after introducing the Maxwell eigenvalue problem,
we address the assembly of FE matrices as they occur in the context of this test case.
The presented algorithms will be employed and verified in Section~\ref{sec:numerics}.

\subsection{Approximation of the physical domain}
\label{subsec:domain_approximation}

\begin{figure}[t!]
    \centering
    \begin{subfigure}[t]{0.3\textwidth}
        \centering
        \includegraphics[width=\textwidth]{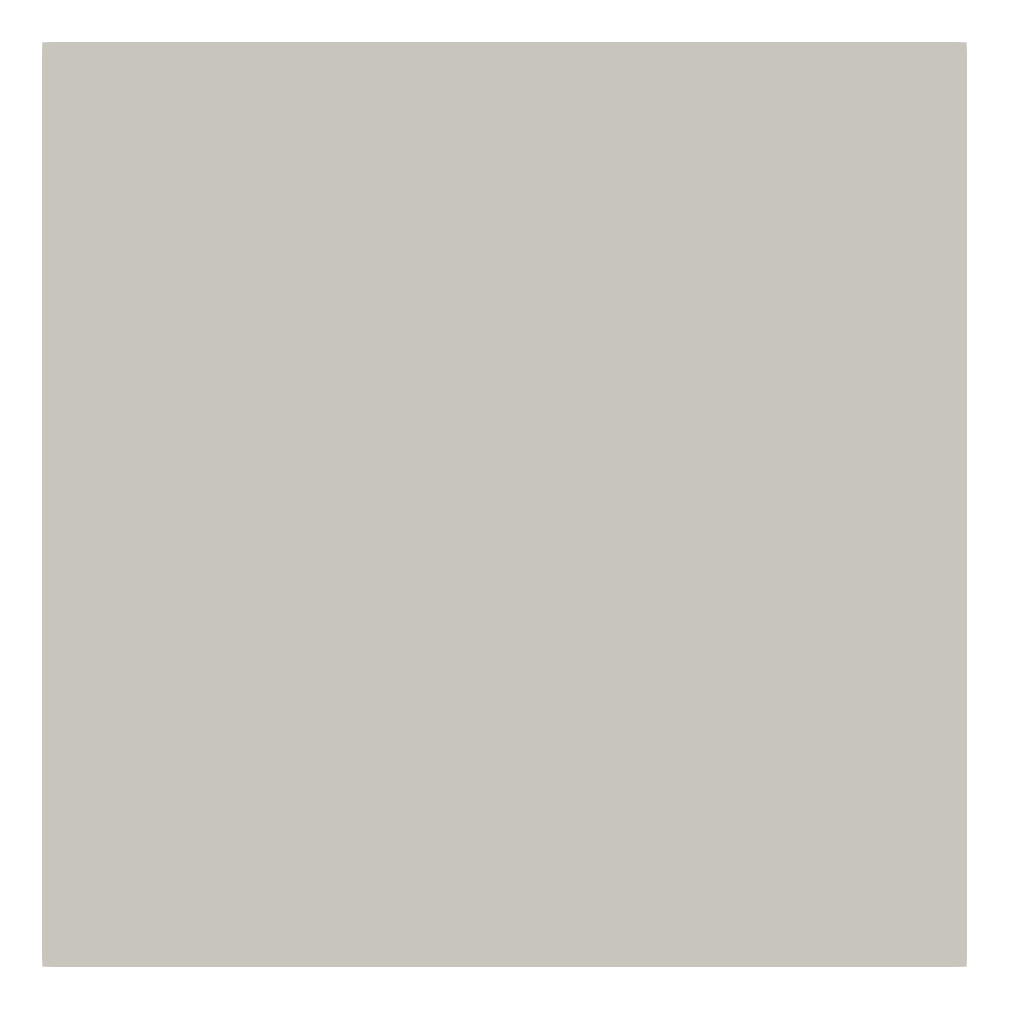}
        \caption{physical domain $\M$}
        \label{subfig:physical_domain}
    \end{subfigure}%
    ~
    \begin{subfigure}[t]{0.3\textwidth}
        \centering
        \includegraphics[width=\textwidth]{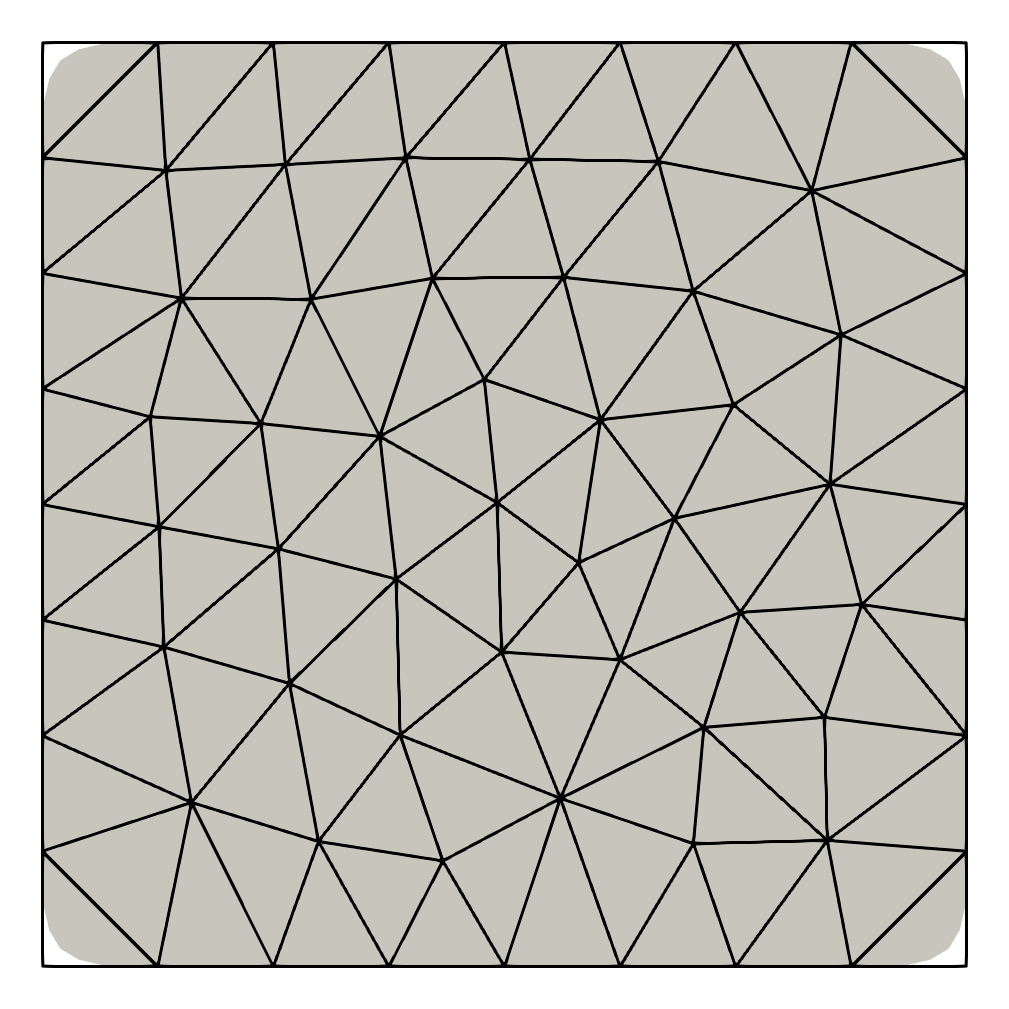}
        \caption{ Wireframe of $\TFE$ (black) covering $\M$ and the domain of its Loop refinement $\AopLoop{0}{\L} \; \TFE$ (grey). Note the shrinking around the corners of the domain.}
        \label{subfig:Loop_refined_physical_domain}
    \end{subfigure}
    ~
    \begin{subfigure}[t]{0.3\textwidth}
        \centering
        \includegraphics[width=\textwidth]{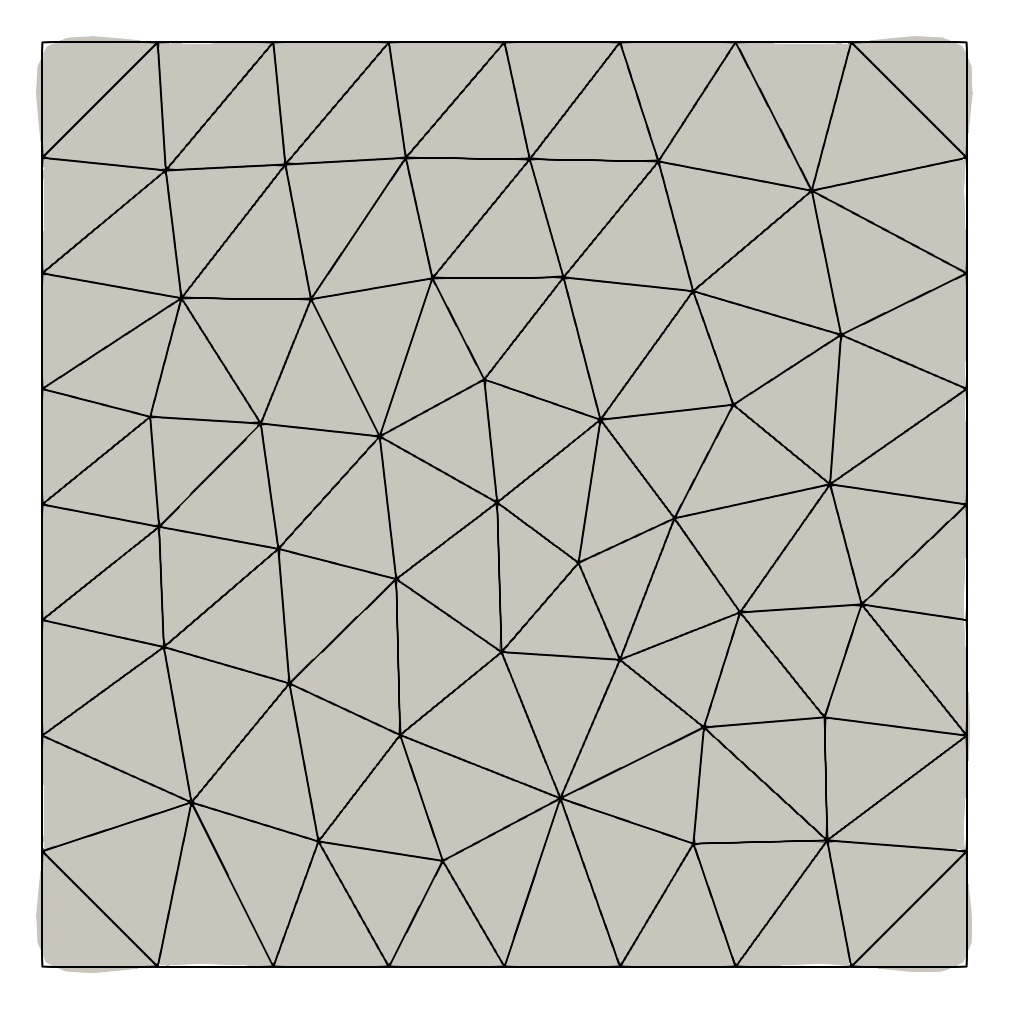}
        \caption{Wireframe of $\TFE$ (black) and domain of the Loop refinement of $\T_\lzero$ from Eq.~\eqref{mesh_construct} (grey). Corners are represented better at the cost of some overshoot.}
        \label{subfig:Loop_approximated_physical_domain}
    \end{subfigure}
    \caption{Approximation of the physical domain $\M$ through Loop subdivision.}
\end{figure}

To perform simulations with our subdivision $k$-form spaces, we need to construct suitable computational domains.
The domain of the system under consideration is typically given as a manifold $\M$ on which the partial differential equation must hold.
In standard FE methods, the domain $\M$ is approximated by a mesh $\TFE \approx \M$.
If we directly used the mesh $\TFE$ as the initial mesh $\T_{\lzero}$ to construct a subdivision mesh hierarchy, the finest mesh $\T_\L$ would typically no longer resemble $\M$ closely enough to be considered a good approximation because Loop subdivision changes the shape of $\T_{\lzero}$. Loop subdivision usually shrinks $\T_{\lzero}$ (cf. Figure~\ref{subfig:Loop_refined_physical_domain}) since the subdivision rules are convex linear combinations of coarse vertices to obtain finer vertices.

To overcome this challenge, Appendix~\ref{app:domain_approximation} describes an algorithm to obtain an initial mesh $\T_{\lzero}$ such that
the mesh
\begin{equation}\label{mesh_construct}
   \AopLoop{\lzero}{\L} \; \T_{\lzero} \eqqcolon \T^{\Loop}_\L  \approx \M
\end{equation}
can be considered a satisfying approximation of the physical domain $\M$ in an appropriate sense, see Figure~\ref{subfig:Loop_approximated_physical_domain}. The algorithm is based on a least squares approximation of the vertex positions of a fine mesh that resembles the domain sufficiently well. The resulting domain approximation method can be interpreted as a projection of the coordinate function into the subdivision $0$-form space $\SDFspace{0}{\lzero}{\L}$ with a lumped mass matrix. A detailed description of the algorithm can be found in Appendix~\ref{app:domain_approximation}.

We will discuss the impact of the domain approximation on the overall approximation performance of the subdivision $k$-form spaces in Section~\ref{sec:numerics}.

\subsection{Matrix assembly for Maxwell's eigenvalue problem}
\label{subsec:assembly}

This section introduces the Maxwell eigenvalue test case, a popular benchmark for testing the preservation of the de Rham complex.
For now, we focus on the assembly of the FE matrices of the test case. The results of our simulations will be discussed and interpreted in Section~\ref{subsec:maxwell_EV_simulation}.

\subsubsection{The weak form of the Maxwell eigenvalue problem}
\label{sec:max_EV_problem}

Let $\M = (0, \pi)^2 \subset \mathbb{R}^2$ be a two-dimensional region.
The weak formulation of the Maxwell eigenvalue problem is given by: Find an eigenvalue $\lambda_i \in \mathbb{R}$ and
its 1-form eigenmode $u \in \mathring{\mathrm{V}}^1$ such that
\begin{equation}
    \label{eq:weak_maxwell}
    ( \extd u, \, \extd v)_\M =  \lambda_i ( u, \, v)_\M \qquad \forall v \in \mathring{\mathrm{V}}^1,
\end{equation}
where $\mathring{\mathrm{V}}$ is a suitable infinite-dimensional function space that satisfies the zero trace boundary condition, i.e. $\trace \, u = 0$, cf. Appendix~\ref{app:vanishing_boundary_complex}. The analytical solution of this equation has the eigenvalues
\begin{equation}
    \label{eq:continuous_eigenvalues}
    \lambda_i = m_i^2 + n_i^2,
\end{equation}
where $m_i$ and $n_i$ are positive integers (including $0$). To discretise the weak formulation in Eq.~\eqref{eq:weak_maxwell}, we seek an approximation $((\lambda_i)_h, u_h)$ of the eigenpair $(\lambda_i, u)$ in $\mathbb{R} \times \mathring{\mathrm{V}}_h^1$, where $\mathring{\mathrm{V}}_h^1 \subset \mathring{\mathrm{V}}^1$ is a suitable finite-dimensional subspace of $\mathring{\mathrm{V}}^1$.
The spectrum of the discretised system should converge towards the continuous spectrum in Eq.~\eqref{eq:continuous_eigenvalues} to yield physically meaningful results. In particular, we should not observe any spurious modes. 

As computational mesh, we use the triangulation $\T^{\Loop}_\L$ from
the previous Section~\ref{subsec:domain_approximation} that approximates the domain $\M$ ($\T^{\Loop}_\L \approx \M$) which we obtain via
Algorithm~\ref{algo:domain_approximation}.

\begin{remark}
    \label{remark:maxwell_tests_entire_complex}
    The Maxwell eigenvalue problem tests the entire discrete complex even though it seems to involve only $1$-forms and their derivative, cf.~\cite{Boffi.2010}. For this reason, we restrict ourselves to the case of $k=1$ when presenting the results of our simulations of the Maxwell eigenvalue problem in Section~\ref{subsec:maxwell_EV_simulation}.
\end{remark}

\subsubsection{Assembly of matrices}

The matrices resulting from assembling weak FE formulations of PDEs depend on the choice of the approximation space $\mathring{V}^1_h \subset \mathring{V}^1$.
For the case of the Maxwell eigenvalue problem in Eq.~\eqref{eq:weak_maxwell}, \cite{Arnold.FEEC} has shown that choosing $\mathring{V}^1_h\big(\T_\L) = \mathring{\NED}_1\big(\T_\L)$, i.e. the lowest-order N\'ed\'elec spaces with vanishing trace (cf. Appendix~\ref{app:vanishing_boundary_complex}), leads to stable and convergent numerical schemes that preserve important structures of the system in Eq.~\eqref{eq:weak_maxwell}, such as the analytical spectrum in Eq.~\eqref{eq:continuous_eigenvalues}.

Initially, we will assemble the matrices for $V^1_h$ and later enforce the boundary conditions by discarding basis functions as described in Appendix~\ref{app:vanishing_boundary_complex}. Most standard FE libraries have an implementation of $\FEECspace{1}{} = \NED_1$ elements that allows us to assemble the matrices that arise from the discretisation of the Maxwell eigenvalue problem in Eq.~\eqref{eq:weak_maxwell}, i.e.
\begin{equation}
    \label{eq:NED_matrices}
    \freeVec{\mathbf{M}}{1}{\L} \  \text{with} \  \big( \freeVec{\mathbf{M}}{1}{\L} \big)_{\coarseIdxone \coarseIdxtwo} = \big( \freeQ{\FEECbasis}{1}{\L}{\coarseIdxone}, \;  \freeQ{\FEECbasis}{1}{\L}{\coarseIdxtwo} \big)_{\T_\L}  \  \ \text{and} \  \ \freeVec{\mathbf{C}}{1}{\L} \ \text{with} \ \big( \freeVec{\mathbf{C}}{1}{\L} \big)_{\coarseIdxone \coarseIdxtwo} = \big( \extd \freeQ{\FEECbasis}{1}{\L}{\coarseIdxone}, \; \extd \freeQ{\FEECbasis}{1}{\L}{\coarseIdxtwo} \big)_{\T_\L},
\end{equation}
where $\freeQ{\FEECbasis}{1}{\L}{\coarseIdxone}$ denotes the $i^{\text{th}}$ basis function of $\NED_1$.

To construct corresponding matrices associated to the spaces of subdivision $1$-forms, we can use the assembled matrices $\freeVec{\mathbf{M}}{1}{\L}$ and $\freeVec{\mathbf{C}}{1}{\L}$, referred to in the following as \emph{FE matrices}.
That is, in the simplest case when no subdivision is performed, we recover the $\NED_1$ finite elements since $\NED_1\big(\T_{\L}\big) = \SDFspace{1}{\L}{\L}$ and, therefore, can use them directly.
To assemble the \emph{subdivision FE matrices} associated to $\SDFspace{1}{\l}{\L}$, we need to compute the entries of the subdivision mass matrix $\constrVec{\mathbf{M}}{1}{\l}{\L}$ and the subdivision curl-curl matrix $\constrVec{\mathbf{C}}{1}{\l}{\L}$ defined as:
\begin{equation}
    \big( \constrVec{\mathbf{M}}{1}{\l}{\L} \big)_{\coarseIdxone \coarseIdxtwo} \coloneqq \Big( \constrQ{\SDFbasis}{1}{\l}{\L}{\coarseIdxone}, \; \constrQ{\SDFbasis}{1}{\l}{\L}{\coarseIdxtwo} \Big)_{\T_\L} \qquad \text{and} \qquad \big( \constrVec{\mathbf{C}}{1}{\l}{\L} \big)_{\coarseIdxone \coarseIdxtwo} \coloneqq \Big( \extd \constrQ{\SDFbasis}{1}{\l}{\L}{\coarseIdxone}, \; \extd \constrQ{\SDFbasis}{1}{\l}{\L}{\coarseIdxtwo} \Big)_{\T_\L}.
\end{equation}

By expanding the basis functions in terms of the subdivision operator
$\Aop{1}{\l}{\L}$ using the corresponding subdivision matrix $\Amat{1}{\l}{\L}$,
we can compute the subdivision mass matrix $\constrVec{\mathbf{M}}{1}{\l}{\L}$ as follows:
\begin{align}
    \big( \constrVec{\mathbf{M}}{1}{\l}{\L} \big)_{\coarseIdxone \coarseIdxtwo} &= \Big( \constrQ{\SDFbasis}{1}{\l}{\L}{\coarseIdxone}, \; \constrQ{\SDFbasis}{1}{\l}{\L}{\coarseIdxtwo} \Big)_{\T_\L} = \Big( \Aop{1}{\l}{\L} \freeQ{\FEECbasis}{1}{\l}{\coarseIdxone}, \;  \Aop{1}{\l}{\L} \freeQ{\FEECbasis}{1}{\l}{\coarseIdxtwo} \Big)_{\T_\L} \\
    &= \Big( \sum_{\fineIdxone = 1}^{\sumdimlk{\L}{1}} \Amatij{1}{\l}{\L}{\fineIdxone}{\coarseIdxone} \freeQ{\FEECbasis}{1}{\L}{\coarseIdxone}, \; \sum_{\fineIdxtwo = 1}^{\sumdimlk{\L}{1}} \Amatij{1}{\l}{\L}{\fineIdxtwo}{\coarseIdxtwo} \freeQ{\FEECbasis}{1}{\L}{\coarseIdxtwo} \Big)_{\T_\L} \\
    &= \sum_{\fineIdxone = 1}^{\sumdimlk{\L}{1}} \sum_{\fineIdxtwo = 1}^{\sumdimlk{\L}{1}} \Amatij{1}{\l}{\L}{\fineIdxone}{\coarseIdxone} \; \big( \freeQ{\FEECbasis}{1}{\L}{\coarseIdxone}, \;  \freeQ{\FEECbasis}{1}{\L}{\coarseIdxtwo} \big)_{\T_\L} \; \Amatij{1}{\l}{\L}{\fineIdxtwo}{\coarseIdxtwo}.
\end{align}
This is equivalent to computing
\begin{equation}
    \label{eq:subdiv_mass_mat}
    \constrVec{\mathbf{M}}{1}{\l}{\L} = \big( \Amat{1}{\l}{\L} \big)^T \; \cdot \; \freeVec{\mathbf{M}}{1}{\L} \; \cdot \; \Amat{1}{\l}{\L}
\end{equation}
using the $\NED_1$ mass matrix $\freeVec{\mathbf{M}}{1}{\L}$ from Eq.~\eqref{eq:NED_matrices}. Analogously, the subdivision curl-curl matrix $\constrVec{\mathbf{C}}{1}{\l}{\L}$ can be assembled by computing
\begin{equation}
    \label{eq:subdiv_curl_mat}
    \constrVec{\mathbf{C}}{1}{\l}{\L} = \big( \Amat{1}{\l}{\L} \big)^T \; \cdot \; \freeVec{\mathbf{C}}{1}{\L} \; \cdot \; \Amat{1}{\l}{\L}.
\end{equation}
Note that we need to use the $1$-form subdivision matrix $ \Amat{1}{\l}{\L}$ even though $\extd \constrQ{\SDFbasis}{1}{\l}{\L}{\coarseIdxone} \in \SDFspace{2}{\l}{\L}$ is a $2$-form. The reason for that is described in Remark~\ref{remark:extd_k_form_refinement}.

Finally, we need to enforce the boundary conditions for the spaces $\SDFspace{1}{\l}{\L}$. According to Definition~\ref{def:subdiv_spaces_with_vanishing_trace}, discarding all basis functions associated to boundary edges is sufficient to constrain the space $\SDFspace{1}{\l}{\L}$ to $\SDFspaceZero{1}{\l}{\L}$. On a matrix level, this can be achieved in two ways. First, we can delete the rows and columns corresponding to the boundary edges from the matrices $\constrVec{\mathbf{M}}{1}{\l}{\L}$ and $\constrVec{\mathbf{C}}{1}{\l}{\L}$. This means that we remove the DoFs associated to the boundary edges. The second option is to set these rows and columns to $0$ except for a $1$ on the main diagonal. We will then solve for the boundary DoFs but they will turn out to have coefficients of $0$ (given a $0$ initial guess for the coefficients). In either case, we are able to restrict to the spaces $\SDFspaceZero{1}{\l}{\L}$ through simple modifications of the subdivision FE matrices.

\begin{remark}
  \label{remark:denser_matrices}
  We refer to Eqs.~\eqref{eq:subdiv_mass_mat} and~\eqref{eq:subdiv_curl_mat} as \emph{unrefinement}. Effectively, the standard FE matrices on level $\L$ are pulled back to level $\l$ using the subdivision matrices. As a result, the dimension of the eigenvalue problem for intermediate levels
  $\l$ is reduced from $\vert\K^1_\L\vert$ to $\vert\K^1_\l\vert$. Simultaneously, the density of the subdivision FE matrices increases because of the larger supports of the basis functions $\constrQ{\SDFbasis}{k}{\l}{\L}{\coarseIdxone}$. Benchmarking the Maxwell eigenvalue problem in Section~\ref{subsubsubsec:computational_time} shows that this trade-off of dimensionality versus density leads to a significant efficiency gain of our approach compared to lowest-order FEEC methods on $\T_{\L}$.
\end{remark}

Eqs.~\eqref{eq:subdiv_mass_mat} and~\eqref{eq:subdiv_curl_mat} show that, on a matrix level, the discrete Maxwell eigenvalue problem from Eq.~\eqref{eq:weak_maxwell} is given by the following generalized eigenvalue problem: Find $\lambda_i \in \mathbb{R}$ and a coefficient vector $\bar{u}_{\l} \in \mathbb{R}^{\sumdimlk{\l}{1}}$ such that
\begin{equation}
    \constrVec{\mathbf{C}}{1}{\l}{\L} \, \bar{u}_{\l} = \lambda_i \, \constrVec{\mathbf{M}}{1}{\l}{\L} \, \bar{u}_{\l}.
\end{equation}
Section~\ref{subsec:maxwell_EV_simulation} presents solutions of this system obtained using the sparse generalized eigenproblem solver from \cite{Hernandez:SLEPC}.

\begin{remark}
    \label{remark:extd_k_form_refinement}
    We want to be able to use as many algorithms that are provided by standard FE libraries as possible to make our implementation more performant and robust. For example, the matrices of the subdivision $k$-form spaces $\SDFspace{k}{\l}{\L}$ can usually be assembled in terms of their ambient FEEC spaces $\FEECspace{k}{\L}$ and the subdivision matrices like in Eqs.~\eqref{eq:subdiv_mass_mat} and~\eqref{eq:subdiv_curl_mat}.
    
    The assembly of inner products with derivatives has one additional subtlety that explains Eq.~\eqref{eq:subdiv_curl_mat}. Standard FE libraries like FEniCS \citep{Logg.2012} compute the derivative of $\constrVec{\omega}{k}{\l}{\L} \in \SDFspace{k}{\l}{\L}$ by
    \begin{equation}
    \begin{split}
    \extd \, \auxincl \, \constrVec{\omega}{k}{\l}{\L} &= \extd \, \sum_{\fineIdxone = 1}^{\sumdimlk{\L}{k}} \Big( \sum_{\coarseIdxone = 1}^{\sumdimlk{\l}{k}} \Amatij{k}{\l}{\L}{\fineIdxone}{\coarseIdxone} \freeQ{c}{k}{\l}{\coarseIdxone} \Big) \freeQ{\FEECbasis}{k}{\L}{\fineIdxone} \\
    &= \sum_{\fineIdxone = 1}^{\sumdimlk{\L}{k}} \Big( \sum_{\coarseIdxone = 1}^{\sumdimlk{\l}{k}} \Amatij{k}{\l}{\L}{\fineIdxone}{\coarseIdxone} \freeQ{c}{k}{\l}{\coarseIdxone} \Big) \, \extd \freeQ{\FEECbasis}{k}{\L}{\fineIdxone} = \constrVec{\omega}{k+1}{\l}{\L},
    \end{split}
    \end{equation}
where the derivative operator acts on the basis functions $\freeQ{\FEECbasis}{k}{\L}{\fineIdxone}$ of the ambient space $\FEECspace{k}{\L}$. Note that the resulting {$(k+1)$}-form $\constrVec{\omega}{k+1}{\l}{\L}$ and $\extd \freeQ{\FEECbasis}{k}{\L}{\fineIdxone}$ are still related by the $k$-form subdivision matrix $\Amatij{k}{\l}{\L}{\fineIdxone}{\coarseIdxone}$. In fact, by Lemma~\ref{lemma:d_tilde_commuted_with_d} and the identity in Theorem~\ref{theorem:consistency_of_BE_op}, 
\begin{equation}
  \extd \, \auxincl \, \constrVec{\omega}{k}{\l}{\L} = \auxincl \, \dualAop{k+1}{\l}{\l}{\L} \, \extd \, \freeVec{\omega}{k}{\l} = \iota \, \primalAop{k+1}{\l}{\l}{\L} \, \extd \, \freeVec{\omega}{k}{\l},
\end{equation} 
which implies
\begin{equation}
    \label{eq:A_k_as_coordinates_for_mathcal_A_k+1}
    \primalAop{k+1}{\l}{\l}{\L} \Big( \sum_{\coarseIdxone = 1}^{\sumdimlk{\l}{k}} \freeQ{c}{k}{\l}{\coarseIdxone} \, \extd \freeQ{\FEECbasis}{k}{\l}{\coarseIdxone} \Big) = \sum_{\fineIdxone = 1}^{\sumdimlk{\L}{k}} \Big( \sum_{\coarseIdxone = 1}^{\sumdimlk{\l}{k}} \Amatij{k}{\l}{\L}{\fineIdxone}{\coarseIdxone} \freeQ{c}{k}{\l}{\coarseIdxone} \Big) \, \extd \freeQ{\FEECbasis}{k}{\L}{\fineIdxone}.
\end{equation}
By analogy with the action of the subdivision operator in Eq.~\eqref{equ:primalAop}, this means that the coordinate expression of $\primalAop{k+1}{\l}{\l}{\L}$ with respect to $\extd \freeVec{\FEECbases}{k}{\l}$ and $\extd \freeVec{\FEECbases}{k}{\L}$ is given by $\Amat{k}{\l}{\L}$. We can think of that as an equivalent condition to the one presented in Lemma~\ref{lemma:compatibility_condition_in_coordinates}.
\end{remark}

\section{Numerical experiments}
\label{sec:numerics}

This section presents numerical results of simulations using our subdivision $k$-form spaces.
We will first study the projection errors and convergence behaviour under uniform refinement of the subdivision $k$-form spaces.

Second, the simulations verify that the subdivision $k$-form spaces preserve the de Rham complex under discretisation. To this end, we apply a standard benchmark test case, namely the Maxwell eigenvalue problem introduced in Section~\ref{sec:max_EV_problem}, to illustrate
structure preservation that manifests itself by an accurate representation of the discrete eigenvalue spectrum of this test case.

\subsection{Approximation of $k$-forms}
\label{subsec:projection}

In this section we explore the capability of subdivision $k$-form spaces $\SDFspace{k}{\l}{\L}$ to approximate continuous differential forms. To this end, let us consider the following continuous differential $k$-forms $\freeVec{\omega}{k}{}$ for $k \in \{0,1,2\}$:
\begin{align}
    \label{eq:differential_forms_for_projection}
    \freeVec{\omega}{0}{}(x_1, x_2) = \freeVec{\omega}{2}{}(x_1, x_2) &= \sin(4 \pi x_1) + e^{2x_2}, \\
    \freeVec{\Vec{\omega}}{1}{}(x_1, x_2) &= \big[ \sin(2\pi x_1) \cos(2\pi x_1), \;\; -\cos(2\pi x_1) \sin(2\pi x_2) \big]^T,
\end{align}
where $\freeVec{\Vec{\omega}}{1}{}$ denotes the vector proxy of the $1$-form $\freeVec{\omega}{1}{}$.  Note that we use the same scalar function for the continuous 0- and 2-forms because this will allow us to better compare the convergence of the corresponding subdivision $k$-forms towards their continuous counterparts.

The approximations $\constrVec{\omega}{k}{\l}{\L}$ of the continuous $k$-forms $\freeVec{\omega}{k}{}$ are obtained by projecting the latter into our subdivision $k$-form spaces, i.e., by solving
\begin{equation}
    \big( \constrVec{\omega}{k}{\l}{\L}, \mu \big)_{\T_\L} = \big( \freeVec{\omega}{k}{}, \mu\big)_{\T_\L} \qquad \forall \mu \in \SDFspace{k}{\l}{\L}.
\end{equation}
The resulting projection error $\constrVec{e}{k}{\l}{\L}$ can be measured by computing the $L^2$ error
\begin{equation}\label{aproxerr}
    \constrVec{e}{k}{\l}{\L} \coloneqq \big\vert \big\vert \, \constrVec{\omega}{k}{\l}{\L} - \freeVec{\omega}{k}{} \,\big\vert \big\vert_{L^2} = \sqrt{\big( \constrVec{\omega}{k}{\l}{\L} - \freeVec{\omega}{k}{}, \; \constrVec{\omega}{k}{\l}{\L} - \freeVec{\omega}{k}{}\big)_{\T_\L}} ,
\end{equation}
with $\constrVec{\omega}{k}{\l}{\L}  \in \SDFspace{k}{\l}{\L} $, evaluated on the finest mesh $\T_\L$. In the following, we refer to the projection error in Eq.~\eqref{aproxerr} as the \emph{approximation error}. The following convergence study compares the approximation errors for varying parameters $\l$ and $\L$ across the form degrees $k$. The simulations are conducted with a fixed initial mesh $\T_{\lzero}$.

\begin{figure}[th!]
    \centering
    \begin{minipage}[t]{0.48\textwidth}
      \centering
      \includegraphics[trim=0.0cm 0.0cm 1.6cm 1.4cm, clip, width=\textwidth]{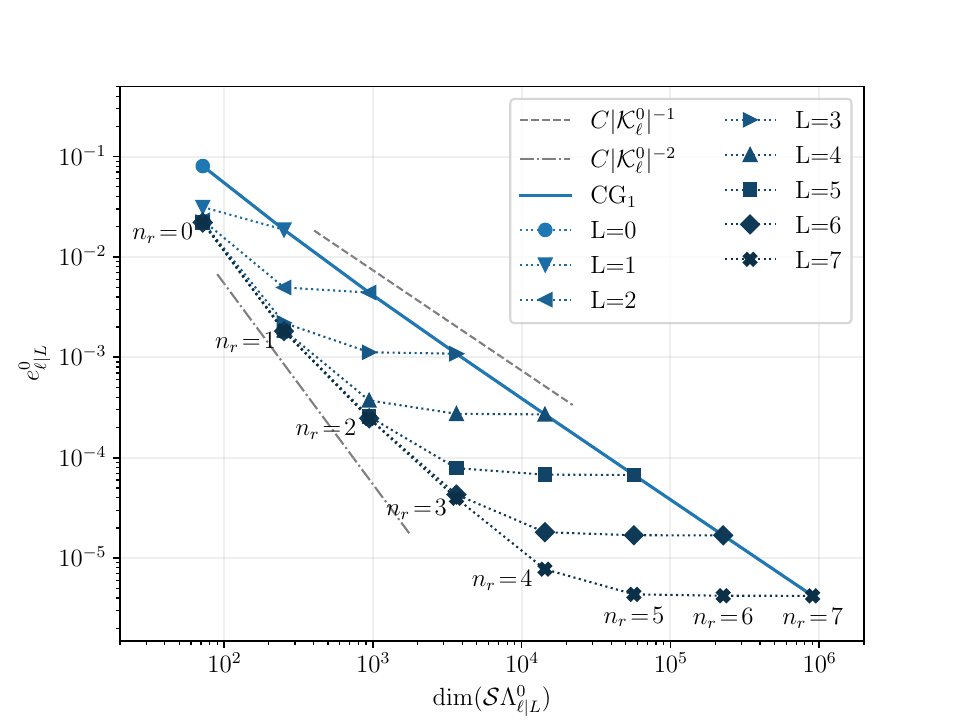}
      \captionsetup{width=\textwidth}
      \caption{Approximation error of $\freeVec{\omega}{0}{}$ from Eq.~\eqref{eq:differential_forms_for_projection} when projected into $\SDFspace{0}{\l}{\L}$ for fixed finest levels $\L$. The refinement number $\nr$ parametrises the nested spaces (according to Proposition~\ref{prop:uniform_spaces_nested})
      along a hierarchy. As an example, we annotated the hierarchy with $\L = 7$.}
      \label{fig:approx_errors_k_0_with_nr}
    \end{minipage}%
    \hspace{0.02\textwidth}
    \begin{minipage}[t]{0.48\textwidth}
      \centering
      \includegraphics[trim=0.0cm 0.0cm 1.6cm 1.4cm, clip, width=\textwidth]{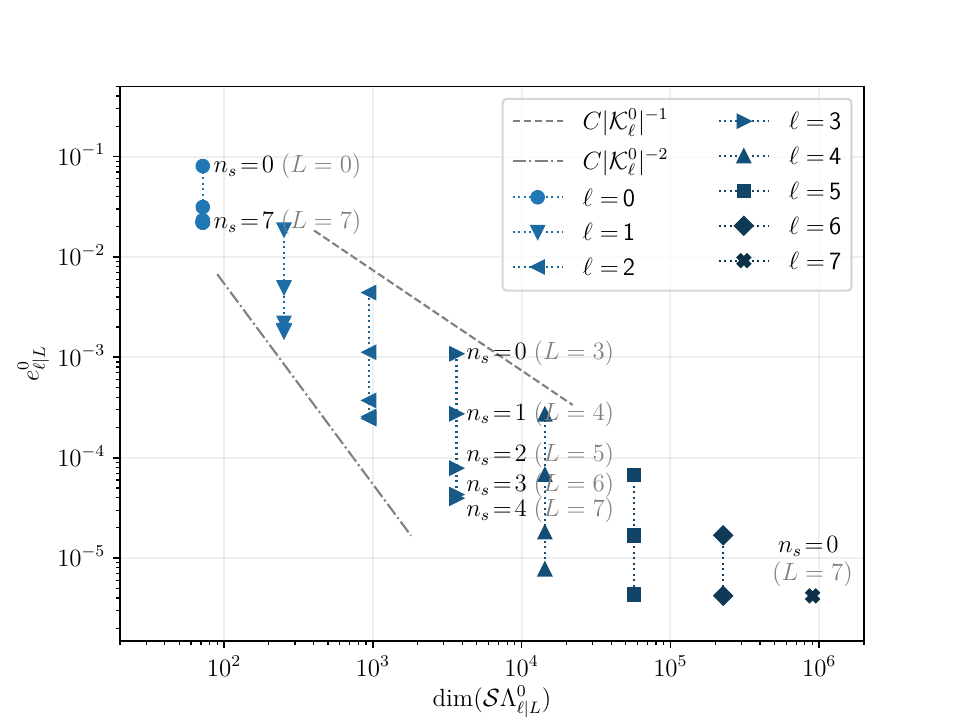}
      \captionsetup{width=\textwidth}
      \caption{Approximation error of $\freeVec{\omega}{0}{}$ from Eq.~\eqref{eq:differential_forms_for_projection} when projected into $\SDFspace{0}{\l}{\L}$ for fixed initial levels $\l$ and thus constant numbers of DoFs (cf. Prop~\ref{prop:subspace_of_Vk}). We chose the maximum value of $\L$ to be $7$. Varying the smoothing number $\ns \in \{ 0, \dots, 7-\l \}$ moves along the vertical lines and thus between hierarchies with different $\L \leq 7$.
      }
      \label{fig:approx_errors_k_0_with_ns}
    \end{minipage}%
\end{figure}

\begin{figure}[th!]
    \centering
    \begin{minipage}[t]{0.48\textwidth}
        \centering
        \includegraphics[trim=0.0cm 0.0cm 1.6cm 1.4cm, clip, width=\textwidth]{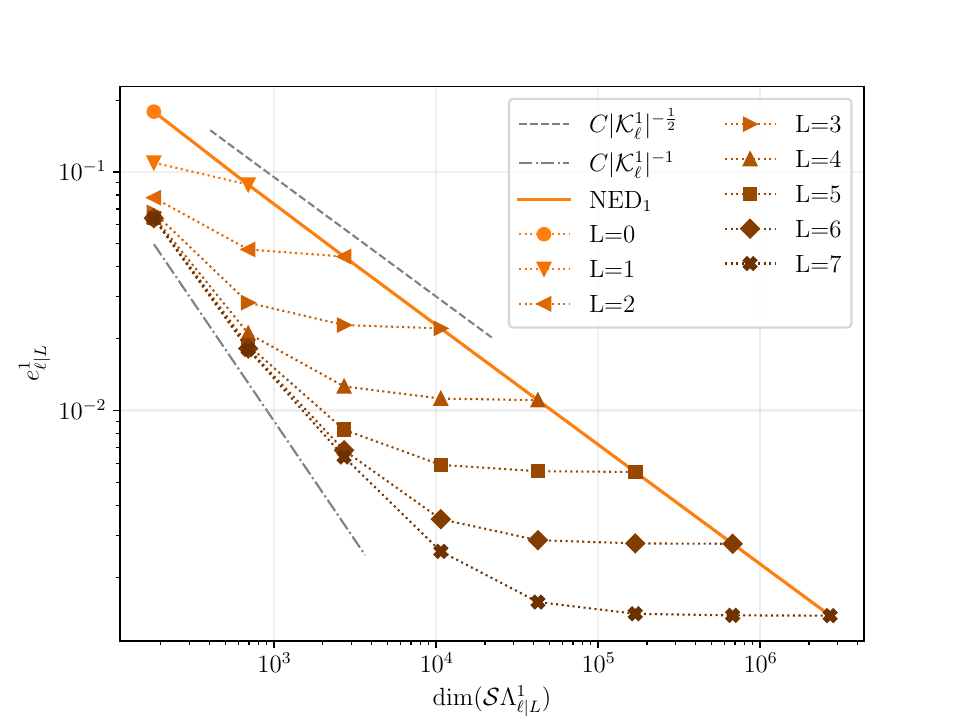}
        \caption{Approximation error of $\freeVec{\omega}{1}{}$ from Eq.~\eqref{eq:differential_forms_for_projection} into $\SDFspace{1}{\l}{\L}$ for fixed $\L$.}
        \label{fig:approx_errors_k_1}
    \end{minipage}%
    \hspace{0.02\textwidth}
    \begin{minipage}[t]{0.48\textwidth}
        \centering
        \includegraphics[trim=0.0cm 0.0cm 1.6cm 1.4cm, clip, width=\textwidth]{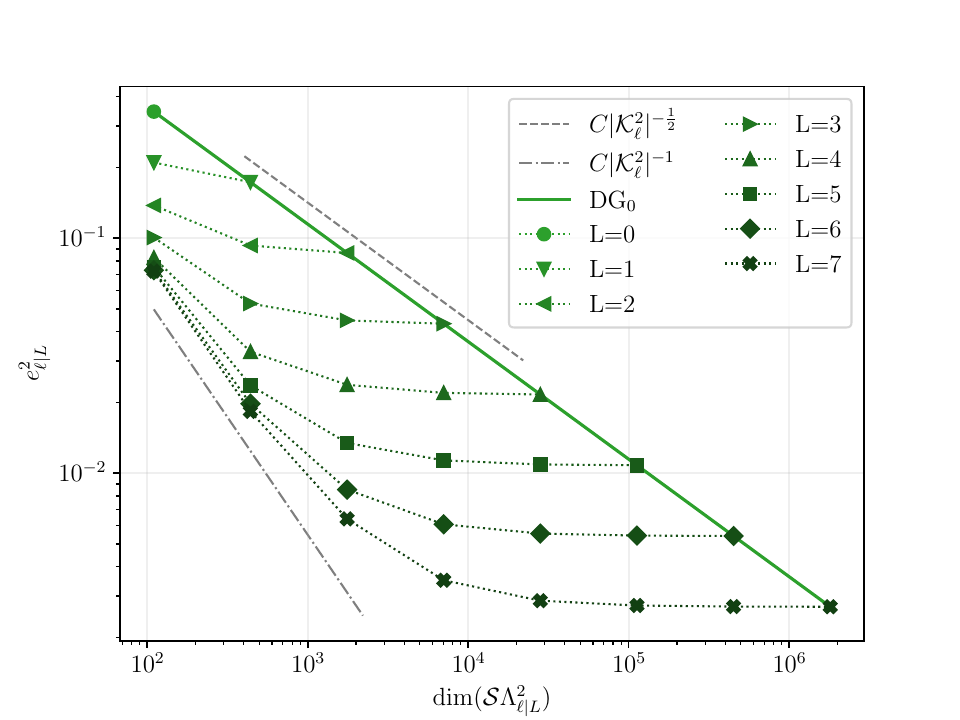}
        \caption{Approximation error of $\freeVec{\omega}{2}{}$ from Eq.~\eqref{eq:differential_forms_for_projection} into $\SDFspace{2}{\l}{\L}$ for fixed $\L$.}
        \label{fig:approx_errors_k_2}
    \end{minipage}%
\end{figure}

Figures~\ref{fig:approx_errors_k_0_with_nr},~\ref{fig:approx_errors_k_0_with_ns},~\ref{fig:approx_errors_k_1}, and~\ref{fig:approx_errors_k_2} show the observed approximation errors of the subdivision $k$-form spaces $\SDFspace{k}{\l}{\L}$ for $k=0$, $k=1$, and $k=2$, respectively, over the parameter values $\L\in\{ 0, \hdots, 7\}$ and $\l \in \{0, \hdots, \L\}$. The approximation errors of the FEEC spaces $\FEECspace{k}{\L} = \SDFspace{k}{\L}{\L}$ are represented by solid lines. The figures confirm the expected convergence rates under function space refinement for the FEEC spaces $\FEECspace{k}{\L}$. The errors obtained using subdivision $k$-forms with constant $\L$ and variable $\l$ are represented by dotted lines. Moving from a data point to its left or right neighbour along these curves represents decreasing or increasing $\l$ by $1$, respectively.

The number of DoFs of both FEEC and subdivision $k$-form spaces scales exponentially with $\l$ as indicated by the equidistance of the markers along the x-axis.
All values obtained from using our subdivision $k$-forms lie to the left and below the FEEC reference curve. This indicates that for any parameter set we explored, the subdivision methods either (i) have a lower approximation error with the same number of DoFs or (ii) require a lower number of DoFs while maintaining a comparable accuracy.

To understand the different kinds of convergence exhibited by the subdivision $k$-form spaces $\SDFspace{k}{\l}{\L}$, we are going to revisit the idea of parameterising the spaces as in
Definition~\ref{def:refinement_and_coarsening_numbers}, using the refinement number $\nr = \l$ and the smoothing number $\ns = \L - \l$. Figures~\ref{fig:approx_errors_k_0_with_nr} and~\ref{fig:approx_errors_k_0_with_ns} display examples of how both numbers count across the hierarchy. In the following, we first focus on the approximation error curves of the $0$-form spaces $\SDFspace{0}{\l}{\L}$ in Figures~\ref{fig:approx_errors_k_0_with_nr} and~\ref{fig:approx_errors_k_0_with_ns} and then extend the argument to $k \!= \! 1$ and $ k \!=\! 2$. Note that both figures contain the same data points but we will look at them from two different perspectives, namely by varying $\nr$ and $\ns$, respectively.

\paragraph{Approximation errors along a hierarchy with fixed $\L$.}

The first perspective on the hierarchies of approximation spaces $\SDFspace{0}{\l}{\L}$ is to consider $\L$ fixed and only $\l$ as variable. In this case, $\nr + \ns = \L$, and we can parameterise the entire hierarchy in terms of $\nr$ using $\SDFspace{0}{\nr}{\L}$. As illustrated in Figure~\ref{fig:approx_errors_k_0_with_nr}, varying the refinement number $\nr$ means moving up and down along the hierarchy in the sense of Proposition~\ref{prop:uniform_spaces_nested}, that is, the number of DoFs scales with $\nr$.

This perspective allows us to describe the shape of the approximation error curve in Figure~\ref{fig:approx_errors_k_0_with_nr}. For the annotated hierarchy with $\L = 7$, we observe almost quadratic convergence until about $\nr = 4$. As $\nr$ increases beyond $4$, the approximation error stagnates and thus the convergence rate plummets. We refer to the refinement number at which this shift happens as the \emph{critical refinement number} $\nrcrit$. The next paragraph will show that the value of $\nrcrit$ for a given $\L$ is tied to the smoothing number.

\paragraph{Determining the critical smoothing number $\nscrit$.}
\label{par:critical_ns}

Figure~\ref{fig:approx_errors_k_0_with_ns} highlights the second perspective on the spaces $\SDFspace{0}{\l}{\L} = \SDFspace{0}{\l}{\l+\ns}$, where we parametrise the spaces in terms of a fixed level $\l$ and a variable smoothing number $\ns$. The fixed $\l$ implies that the dimension of the spaces $\SDFspace{0}{\l}{\l+\ns}$ is constant according to Proposition~\ref{prop:subspace_of_Vk}. For this reason, the approximation error curves in Figure~\ref{fig:approx_errors_k_0_with_ns} become vertical lines. Every marker along a given error curve corresponds to a value of $\ns$ (see annotations). Note that $\ns$ is determined by $\L$ because it is defined as $\ns=\L-\l$, see Definition~\ref{def:refinement_and_coarsening_numbers}. In other words, since $\l$ is fixed, to increase $\ns$ we have to increase $\L$. The maximum level $\L$ we used for simulations is $7$. Therefore, the data points in Figures~\ref{fig:approx_errors_k_0_with_nr} and~\ref{fig:approx_errors_k_0_with_ns} are exactly the same.

For the annotated sequence with $\l=3$ (right-pointing triangles) in Figure~\ref{fig:approx_errors_k_0_with_ns}, we observe that the error decreases by around 1.5 orders of magnitude through smoothing steps alone. We also note that the approximation error converges to a non-zero value. The reason for that behaviour is that additional smoothing steps only let the basis functions converge closer towards the limit basis functions without introducing additional DoFs. Thus, the approximation error of $\SDFspace{0}{\l}{\L}$ converges towards the approximation error of the space $\SDFspace{0}{\l}{\infty}$ spanned by the smooth limit basis functions (cf. Remark~\ref{remark:limit_basis_functions}).

Figure~\ref{fig:approx_errors_k_0_with_ns} shows that the approximation error decreases significantly until around $\ns = 3$ and stagnates after that. That is consistent with the findings of \cite{Goes.2016}, who found that subdividing for more than 3 levels does not significantly increase the approximation quality. Hence, we find that the \emph{critical smoothing number} for $k\!=\!0$ is $\nscrit = 3$. This implies that the critical refinement number from the paragraph above satisfies $\nrcrit = \L-\nscrit$.

\paragraph{Changing convergence rates along a hierarchy.}

The critical smoothing number $\nscrit = 3$ allows us to understand the observed convergence rates of the approximation errors in Figure~\ref{fig:approx_errors_k_0_with_nr}. The observation of almost quadratic convergence until about $\nrcrit = \L- \nscrit = 4$ can be explained by the fact that, for all spaces with $\nr \leq \nrcrit$, the smoothing number satisfies $\ns \geq \nscrit = 3$. Thus, the basis functions can be considered sufficiently accurate
approximations of the limit basis functions and the convergence rate is enhanced by the subdivision-induced regularity, see Remark~\ref{remark:subdiv_smoothness}.

For $\nr \geq \nrcrit$ and fixed $\L$, the basis functions of the spaces $\SDFspace{0}{\nr}{\L}$ can no longer be considered good approximations of the limit basis functions since $\ns < \nscrit$. This is why the convergence rate begins to stagnate. In the resulting flat parts of the error curves, the loss of subdivision-induced regularity is counteracted by the increasing number of DoFs.

The approximation error curves with $\L > 7$ will look similar to the ones in Fig.~\ref{fig:approx_errors_k_0_with_nr}. Initially for small $\nr$, the approximation error converges at an enhanced rate because $\ns > \nscrit$. Since the last space of the hierarchy, $\SDFspace{k}{\L}{\L} = \FEECspace{k}{\L}$, is a FEEC space, the approximation error curves need to return to the regular FEEC errors for $\ns < \nscrit$. This leads to the flat part of the curve. The shift between these two dominant effects occurs at $\ns = \nscrit$.

\paragraph{Convergence rates for fixed $\ns$.}

Picking sequences of spaces $\SDFspace{0}{\l}{\l + \ns}$ that share the same $\ns$ reveals that the convergence rates of the approximation error over the number of DoFs increase as $\ns$ increases. To see this, consider first the case $\ns = 0$ which corresponds to the FEEC reference curve that has a linear convergence rate in the number of DoFs. On the other end of the spectrum, spaces with $\ns \geq 3$ feature an almost quadratic convergence rate due to the subdivision-enhanced regularity. The cases of $0 < \ns < \nscrit$ end up between these two cases.

\bigskip

\noindent All of the observations for $k\!=\!0$ also apply to the spaces of subdivision $1$- and $2$-forms. Figures~\ref{fig:approx_errors_k_1} and~\ref{fig:approx_errors_k_2} indicate that the critical smoothing numbers for $k\!=\!1$ and $k\!=\!2$ are $\nscrit=4$ and $\nscrit=5$, respectively. As for the $0$-form spaces, we observe that the convergence rates for $1$- and $2$-form spaces with $\ns \geq \nscrit$ are almost doubled, from $\tfrac{1}{2}$ (dashed line) to almost $1$ (dash-dotted line).

The following remark discusses implications that the above observations have on the practical implementation of the subdivision $k$-form spaces $\SDFspace{k}{\l}{\L}$.

\begin{remark}
  \label{remark:practical_implications}
  In practice, the finest level $\L$ will be fixed before the simulation so that the hierarchy of meshes and subdivision operators can be assembled as a precomputation. It is possible to increase $\L$ during the simulation but it would come at an additional computational cost since we would need to build the FE matrices in Eq.~\eqref{eq:NED_matrices} on the new finest level $\L$ again.

  The previous discussion shows that enhanced convergence rates can be observed for the spaces $\SDFspace{k}{\l}{\L}$ with $\lzero \leq \l \leq \L - \nscrit$. Hence, if we have some a priori insight about the finest scales of the problem under consideration, we can choose $\l$ and $\L = \l + \nscrit$ accordingly.
\end{remark}

\paragraph{Reduction in DoFs by coarsening the approximation spaces.}

The stagnation of the approximation errors for spaces with $\nr > \nrcrit$ (i.e. we are in the flat parts of the convergence curves) provides an opportunity to reduce simulation DoFs relative to a given FEEC space $\FEECspace{k}{\L}$, given that the desired solution is smooth enough (on the scale of the mesh). In the sense of Proposition~\ref{prop:uniform_spaces_nested}, this means we can move from a super- to a subspace with less DoFs without increasing the additional approximation error significantly.

Consider two simulations, one on $\T_\L$ using $\FEECspace{k}{\L}$ and one on $\T_{\L-\nscrit}$ using $\SDFspace{k}{\L-\nscrit}{\L}$. The discussion above implies that these two simulations have similar accuracies. However, the simulation using $\SDFspace{k}{\L-\nscrit}{\L}$ has only a fraction of the number of DoFs compared to $\FEECspace{k}{\L}$ because the mesh that controls the DoFs is $\nscrit$ times coarser. For example, for $k=2$, we only need $(1/4)^\nscrit$ times the number of DoFs. We observe similar scaling for $k=0$ and $k=1$. For instance, Section~\ref{subsec:eigenmodes} presents two $1$-form eigenmodes obtained with $\ns = 4$ levels of coarsening. This represents a significant potential for saving of computational resources. We will investigate these aspects in more detail in Section~\ref{subsubsubsec:computational_time}.

\FloatBarrier

\subsection{Simulations of the Maxwell problem}
\label{subsec:maxwell_EV_simulation}

This section employs the subdivision $k$-form spaces we constructed in the previous sections to simulate the Maxwell eigenvalue test case given in, e.g., \cite{Arnold.FEEC, Buffa.2011_IGA_with_splines}. The purpose of this example is not to assess the performance of the method for general electromagnetic simulations, but rather to provide a structure-sensitive test of the compatibility and exactness of the proposed subdivision $k$-form spaces on a simply connected domain. From a structure-preservation point of view, the crucial result is the absence of spurious modes. After confirming that, we will investigate the convergence of the eigenvalues under uniform mesh refinement, in particular with regards to the domain approximation described in Section~\ref{subsec:domain_approximation}.

The following discussions feature only approximations of $1$-forms. Recall that the Maxwell eigenvalue problem still tests the entire complex, cf. Remark~\ref{remark:maxwell_tests_entire_complex}.

\subsubsection{Convergence of eigenvalues with domain approximation}
\label{sec:geodomainerror}

As discussed in Section~\ref{subsec:domain_approximation}, the mesh $\T^{\Loop}_\L$, obtained by refining the adjusted mesh $\T_\lzero$ according to Eq.~\eqref{mesh_construct}, is a good approximation of the domain $\M$ but does not perfectly approximate it. We call the resulting error the \emph{geometric domain error}. As it turns out, this error has an impact on the accuracy of the computed eigenvalues (EV) and on their convergence behaviour.

To capture the impact of this geometric domain error, there are two cases to be distinguished for our convergence analysis: (i) we construct one mesh hierarchy with fixed initial and finest levels $\lzero$ and $\L$ and then study convergence towards the analytical eigenvalues when refining the intermediate levels $\l$; or
(ii) we fix only the finest level $\L$ and construct a new mesh hierarchy from $\l$ to $\L$ for every $\l$. Let us consider both cases in more detail.

\paragraph{Case (i).}
In case (i), always starting the hierarchy from $\lzero$
implies that $\T_{\lzero}$ depends only on $\L$, according to Algorithm~\ref{algo:domain_approximation},
and therefore $\T_{\lzero}$ is the same for all spaces $\SDFspace{k}{\l}{\L}(\T_{\lzero})$. In the following discussion, we will use $\T_\lzero$ whenever we consider case (i) and keep the dependency on $\T_\lzero$ explicit.
When increasing $\L$, the basis functions $\constrQ{\SDFbasis}{k}{\l}{\L}{\coarseIdxone}$
are smoothed out and gain subdivision-induced regularity but the approximation quality of the domain remains largely unchanged. In fact, we have
\begin{equation}
    \AopLoop{\lzero}{\L} \big( \T_{\lzero} \big) \not\to \M = [ 0, \pi]^2 \quad \text{as} \quad \L \to \infty,
\end{equation}
that is, even with increasingly fine domain meshes on level $\L$, the geometric domain error does not fully vanish.
This behaviour can be observed in Figure~\ref{fig:individual_eigenvalue}. The curves indicate that
with increasing $\L$ the simulated EV using the subdivision $1$-form spaces
$\SDFspaceZero{1}{\l}{\L}\big( \T_{\lzero} \big)$ converge to slightly different values as we refine the space, i.e. as $\l$ approaches $\L$.
To be more precise, for any $\L$, the approximations converge towards the eigenvalues obtained using
$\SDFspaceZero{1}{\L}{\L} \big( \T_{\lzero} \big) = \FEECspaceZero{1}{\L}$.

As such they are consistent with the finest FEEC approximation. Since we know that the latter converge to the correct spectrum on any simply connected domain, we can conclude that the observed small discrepancies in the computed eigenvalues in Figure~\ref{fig:individual_eigenvalue} from the value $\lambda_3 = 2$ are due to the geometric domain error, which leads to a slightly different analytical spectrum.

\smallskip

\paragraph{Case (ii).}
In case (ii), only the finest level $\L$ is fixed while, for every level $\l$, we use a different mesh hierarchy starting at $\l$ according to Algorithm~\ref{algo:domain_approximation}. We will denote the initial meshes of these different hierarchies by $\hat{\T}_\l$ to indicate that they are distinct from the meshes in the hierarchy in case (i) (i.e. $\hat{\T}_\l \neq \T_{\l}$ for any $0 < \l \leq \L$). With increasing values for $(\l, \L)$, the geometric domain error now becomes smaller and smaller since the finest mesh $\hat{\T}_\L$ approximates the domain $\Omega$  increasingly accurately, i.e.
\begin{equation}
    \AopLoop{\l}{\L} \big( \hat{\T}_\l \big) \to \Omega \quad \text{as} \quad \l, \L \to \infty.
\end{equation}
Therefore, to reduce the influence of the geometric domain error on the errors of the computed eigenvalues, we
investigate them here in case (ii) for the spaces $\SDFspaceZero{1}{\l}{\L} \big( \hat{\T}_\l \big)$. Using these spaces, we indeed observe a convergence of the eigenvalues against the corresponding analytical values, as shown in Figure~\ref{fig:individual_eigenvalue_with_domain_approx_on_every_level} for the third eigenvalue $\lambda_3=2$. A careful examination of our results reveals that the overall error is dominated by the geometry domain error and not by the actual approximation error. Therefore, we do not investigate the convergence rates of the eigenvalues in more detail. The quest to find better test cases where these two error contributions are decoupled is beyond the scope of this work.

\begin{figure}[!t]
    \centering
    \begin{minipage}[t]{.48\textwidth}
       \begin{center}
        \includegraphics[trim=0.8cm 0.0cm 1.6cm 1.4cm, clip, width=\textwidth]{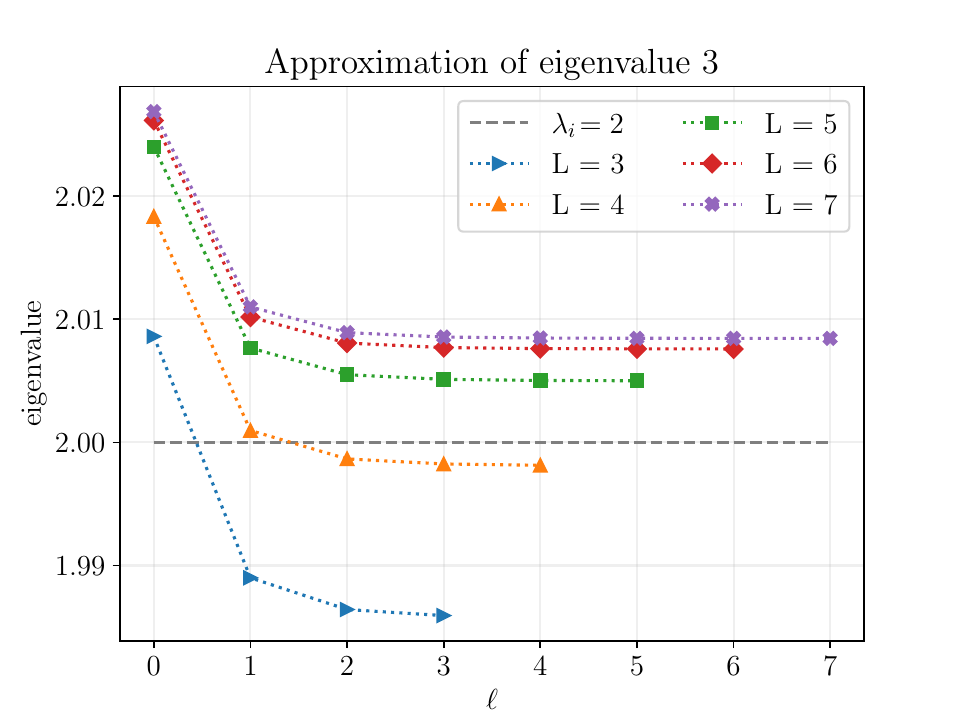}
       \end{center}
        \vspace{-8mm}
        \caption{Approximations of the eigenvalue $\lambda_3 = 2$ using $\SDFspaceZero{1}{\l}{\L}(\T_{\lzero})$ with fixed domain approximation for $(\lzero, \L)$
        (case (i) of Section~\ref{sec:geodomainerror}).
        }
        \label{fig:individual_eigenvalue}
    \end{minipage}
    \hfill
    \begin{minipage}[t]{0.484\textwidth}
        \begin{center}
        \includegraphics[trim=0.7cm 0.0cm 1.6cm 1.4cm, clip, width=\textwidth]{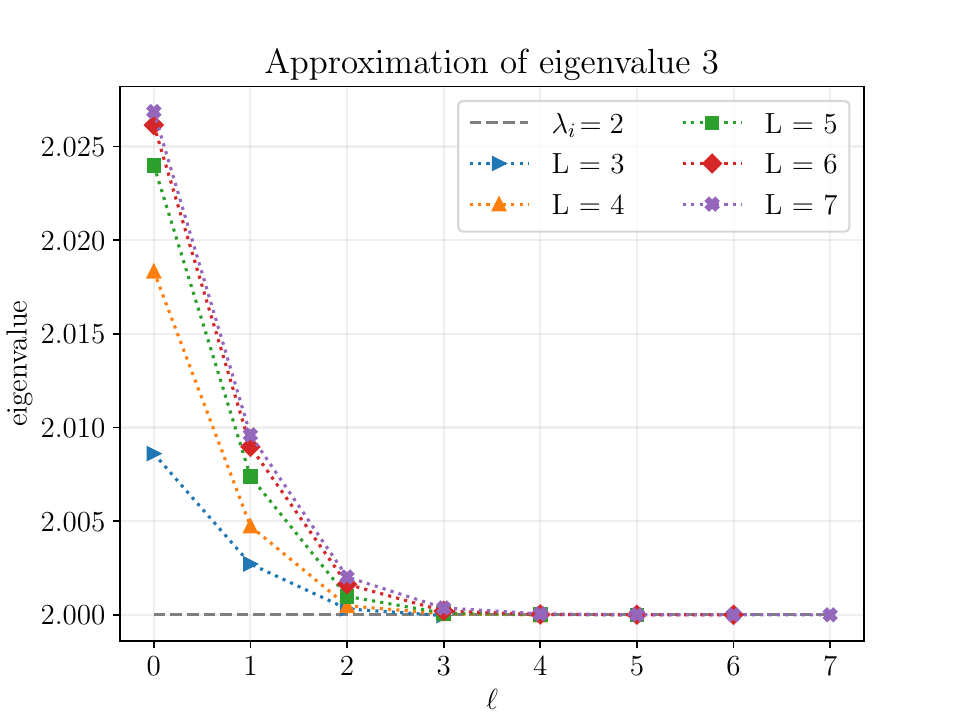}
    \end{center}
    \vspace{-8mm}
    \caption{Approximations of the eigenvalue $\lambda_3 = 2$ using $\SDFspaceZero{1}{\l}{\L}\big( \hat{\T}_\l \big)$ with
    domain approximation for every pair $(\l, \L)$ (case (ii) of Section~\ref{sec:geodomainerror}).
    }
    \label{fig:individual_eigenvalue_with_domain_approx_on_every_level}

    \end{minipage}
\end{figure}

\subsubsection{Representation quality of the discrete eigenvalue spectrum}

Let us next look at the discrete eigenvalue spectrum and compare it to the analytical values. Figure~\ref{fig:individual_spectra}
shows the first 50 discrete eigenvalues computed with our structure-preserving subdivision $k$-form spaces
$\SDFspaceZero{1}{\l}{\L}\big( \T_{\lzero} \big)$ (case (i)).
The results confirm that simulations based on subdivision $k$-forms reproduce the spectrum observed
for $\FEECspaceZero{1}{\L}$.

Hence, the match of discrete and analytic values in Figure~\ref{fig:individual_spectra} confirms the structure-preserving properties of the complex formed by the subdivision $k$-form spaces $\SDFspaceZero{k}{\l}{\L}\big( \T_{\lzero} \big)$ proven in Theorem~{\ref{theo:complex_with_vanishing_trace}}. This applies to both cases of refinement strategy; that is,
although the spaces $\SDFspaceZero{k}{\l}{\L}\big( \T_{\lzero} \big)$ of case (i) exhibit the above discussed geometric domain error, the computed eigenvalues still agree with the ones computed using $\FEECspaceZero{k}{\L}$.
Similarly, when approximating the domain using the strategy of case (ii), we observe convergence of the discrete EV spectrum towards the analytical eigenvalues computed on a square domain.

\begin{figure}[t]
        \centering
        \includegraphics[trim=0.6cm 0.0cm 1.0cm 1.4cm, clip, width=0.6\textwidth]{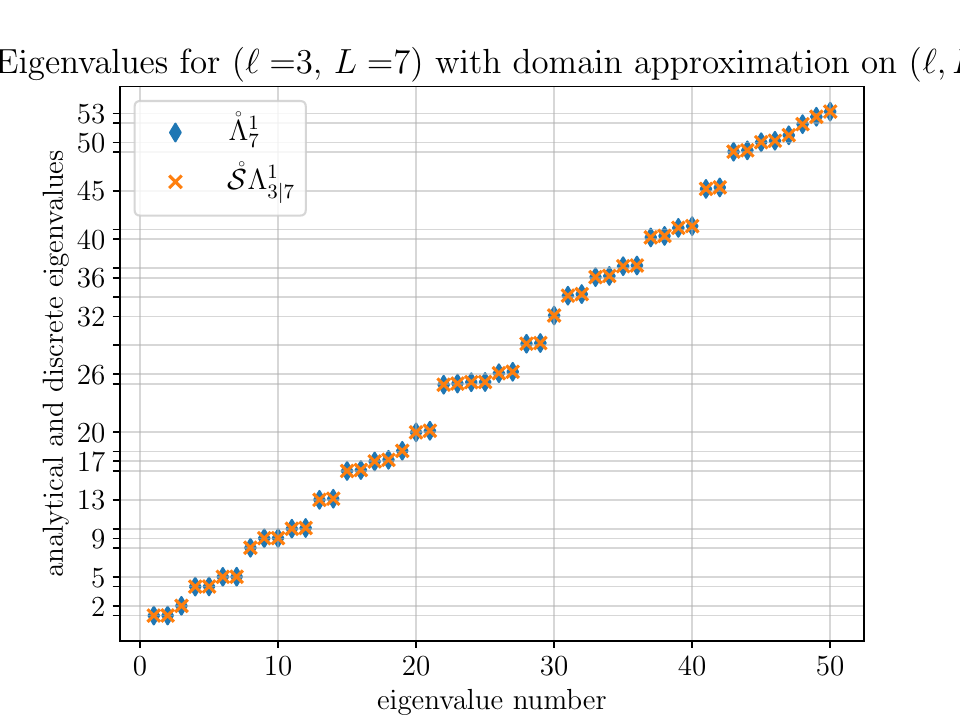}
        \caption{The first $50$ computed discrete EV $(\lambda_h)_i$ over eigenvalue number $i$.
        The grey horizontal lines indicate analytical EV on a square domain (cf. Eq.~\eqref{eq:continuous_eigenvalues}). The mesh hierarchy starts at $\T_{\lzero}$, i.e., we are in case (i) of Section~\ref{sec:geodomainerror}.
        }
        \label{fig:individual_spectra}
\end{figure}

\subsubsection{Subdivision-induced regularity of eigenmodes}
\label{subsec:eigenmodes}

This section investigates the approximation quality of the eigenmodes of the Maxwell eigenvalue problem.
We will only consider case (i), i.e. $\SDFspaceZero{k}{\l}{\L} \big(\T_\lzero\big)$, but similar results hold for case (ii) (not shown).
In the following, we omit the dependence on the mesh $\T_\lzero$ again.

Figures~\ref{fig:eigenmode_2} and~\ref{fig:eigenmode_58} contain plots of the eigenmodes corresponding to $\lambda_3 = 2$
and $\lambda_{52} = 58$, respectively, computed with coarse (left) and finer FEEC
spaces (right) and with the subdivision $k$-form space $\SDFspaceZero{1}{3}{7}$ (centre).
Comparing the three different figures illustrates the effects of adding subdivision to FEEC spaces
or, in other words, of using subdivision $k$-forms rather than standard FEEC ones.

\begin{figure}[!htb]

    \centering
    \hspace{0.00mm}
    \subcaptionbox{$\SDFspaceZero{1}{3}{3}$ ($\approx 10^4$ DoFs)}{%
      \includegraphics[width=0.29\textwidth]{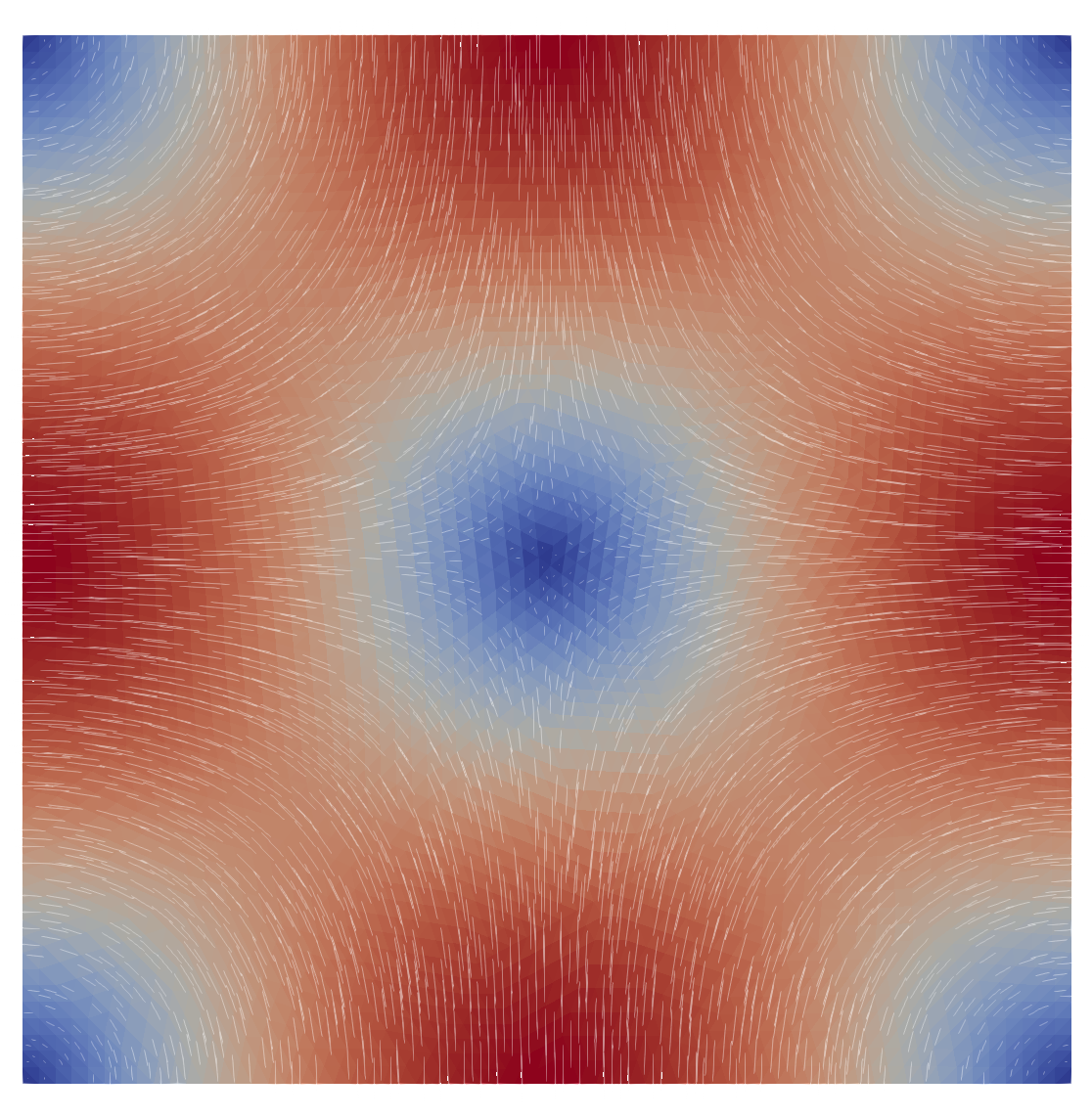}%
    }
    \!\!\!
    \subcaptionbox{$\SDFspaceZero{1}{3}{7}$ ($\approx 10^4$ DoFs)}{%
      \raisebox{1.3mm}{
       \includegraphics[width=0.28\textwidth]{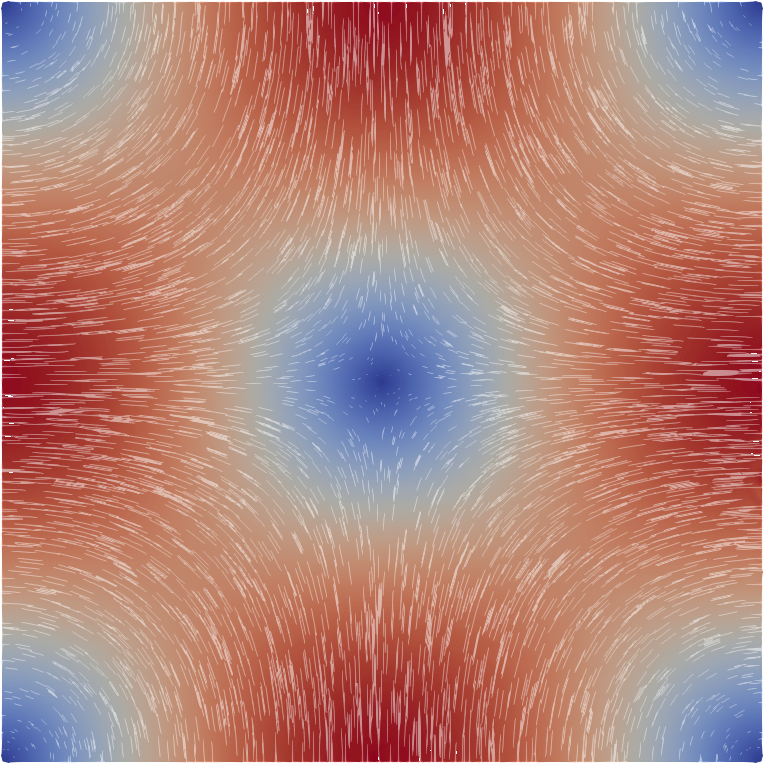}%
      }
    }
    \!\!
    \subcaptionbox{$\SDFspaceZero{1}{7}{7}$ ($\approx 2.7\cdot 10^6$ DoFs)}{%
      \includegraphics[width=0.29\textwidth]{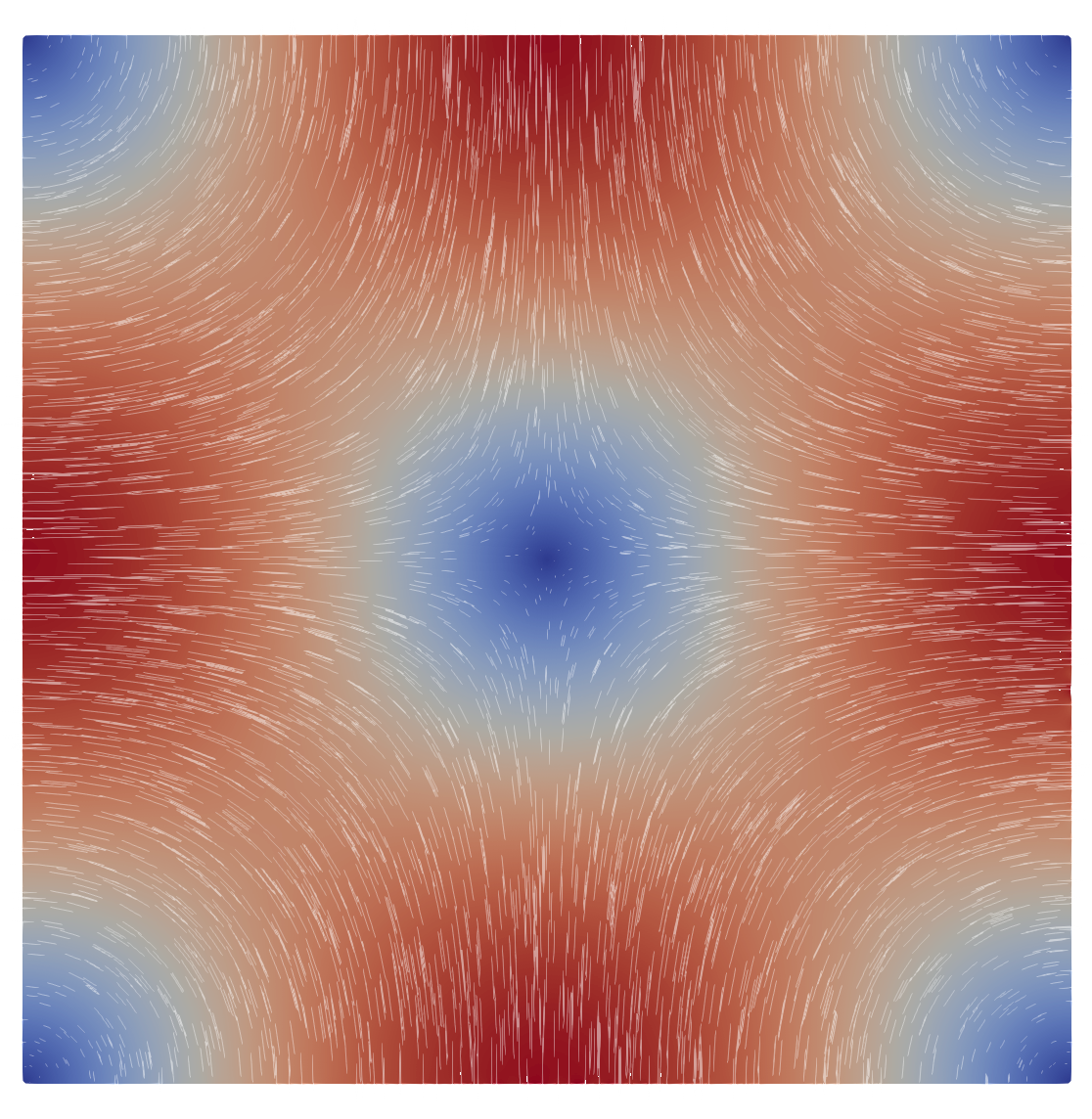}
    }
    \includegraphics[width=0.068\textwidth]{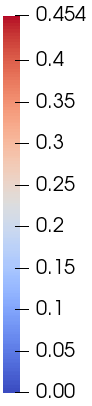}
    \caption{Approximation of the eigenmode corresponding to the eigenvalue $\lambda_3 = 2$ using subdivision and FEEC $k$-form spaces. The colorbar on the right displays the magnitude of the vector field while the white lines in the figures indicate its direction. The hierarchy of the approximation spaces starts at $\T_{\lzero}$, i.e., we are in case (i).
    }
    \label{fig:eigenmode_2}
\end{figure}

\begin{figure}
    \centering
    \subcaptionbox{$\SDFspaceZero{1}{3}{3}$ ($\approx 10^4$ DoFs)}{%
      \includegraphics[width=0.29\textwidth]{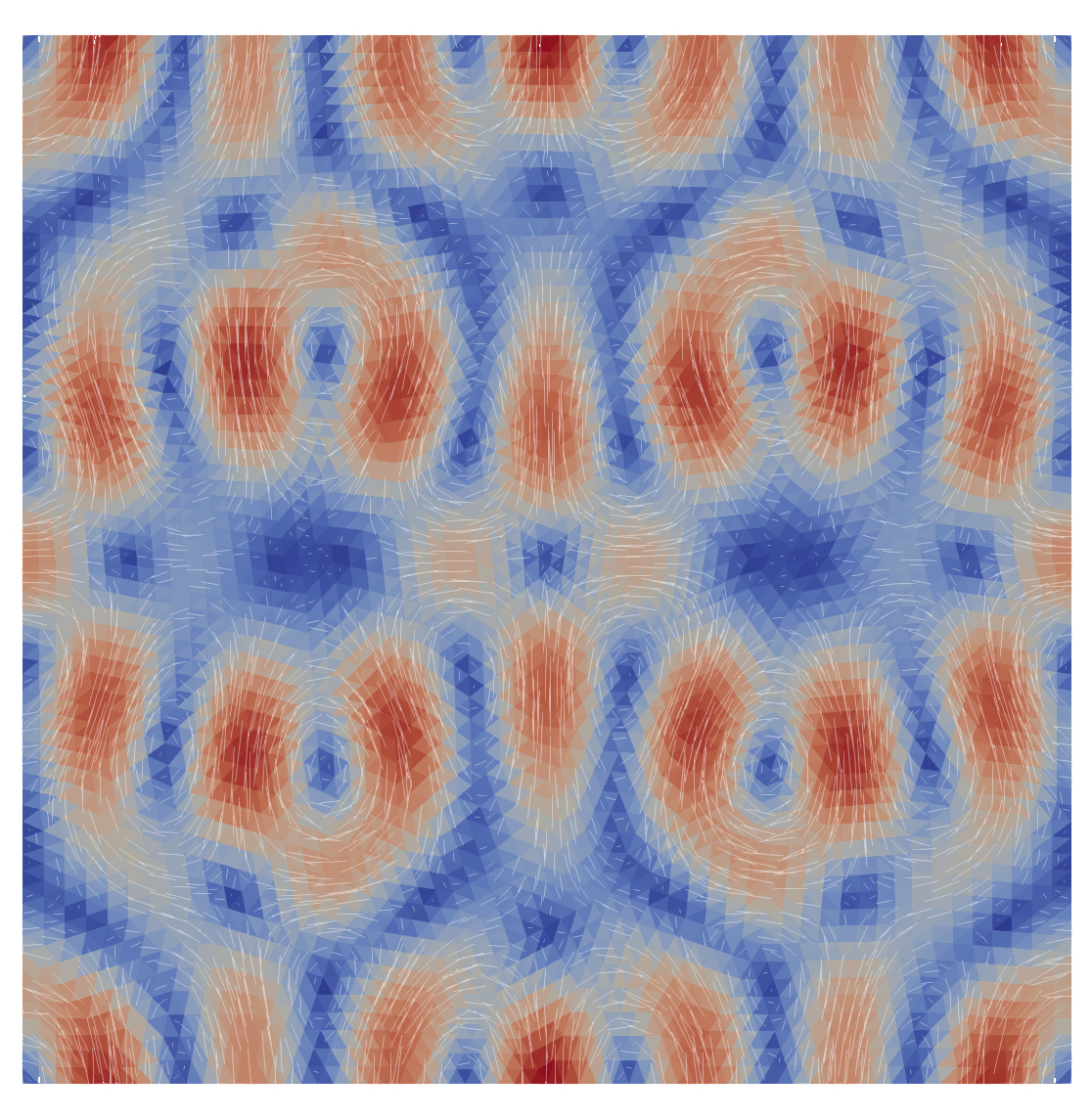}%
    }
    \!\!\!
    \subcaptionbox{$\SDFspaceZero{1}{3}{7}$ ($\approx 10^4$ DoFs)}{%
      \raisebox{1.3mm}{
        \includegraphics[width=0.28\textwidth]{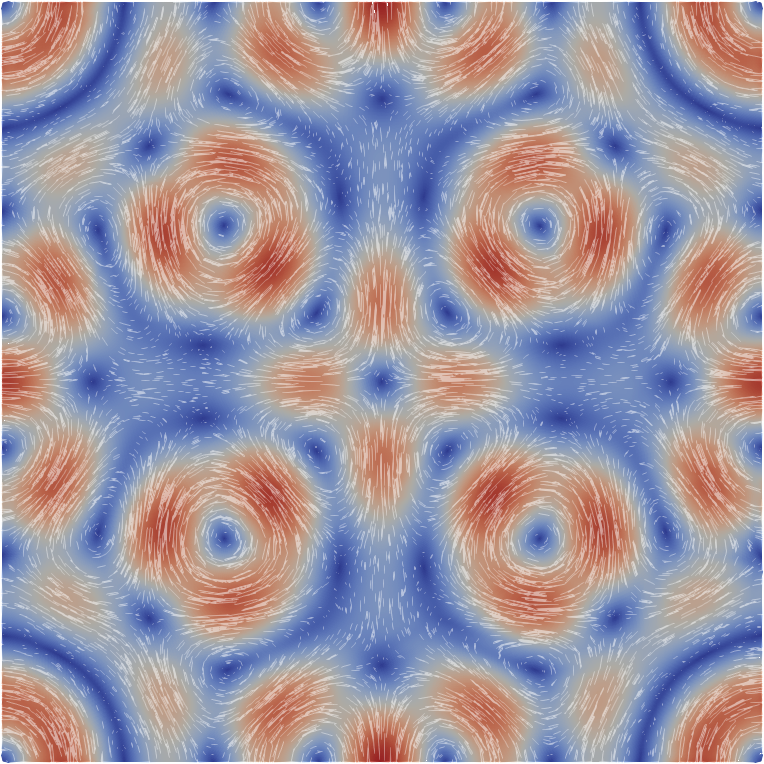}%
      }
    }
    \!\!
    \subcaptionbox{$\SDFspaceZero{1}{7}{7}$ ($\approx 2.7\cdot 10^6$ DoFs)}{%
      \includegraphics[width=0.29\textwidth]{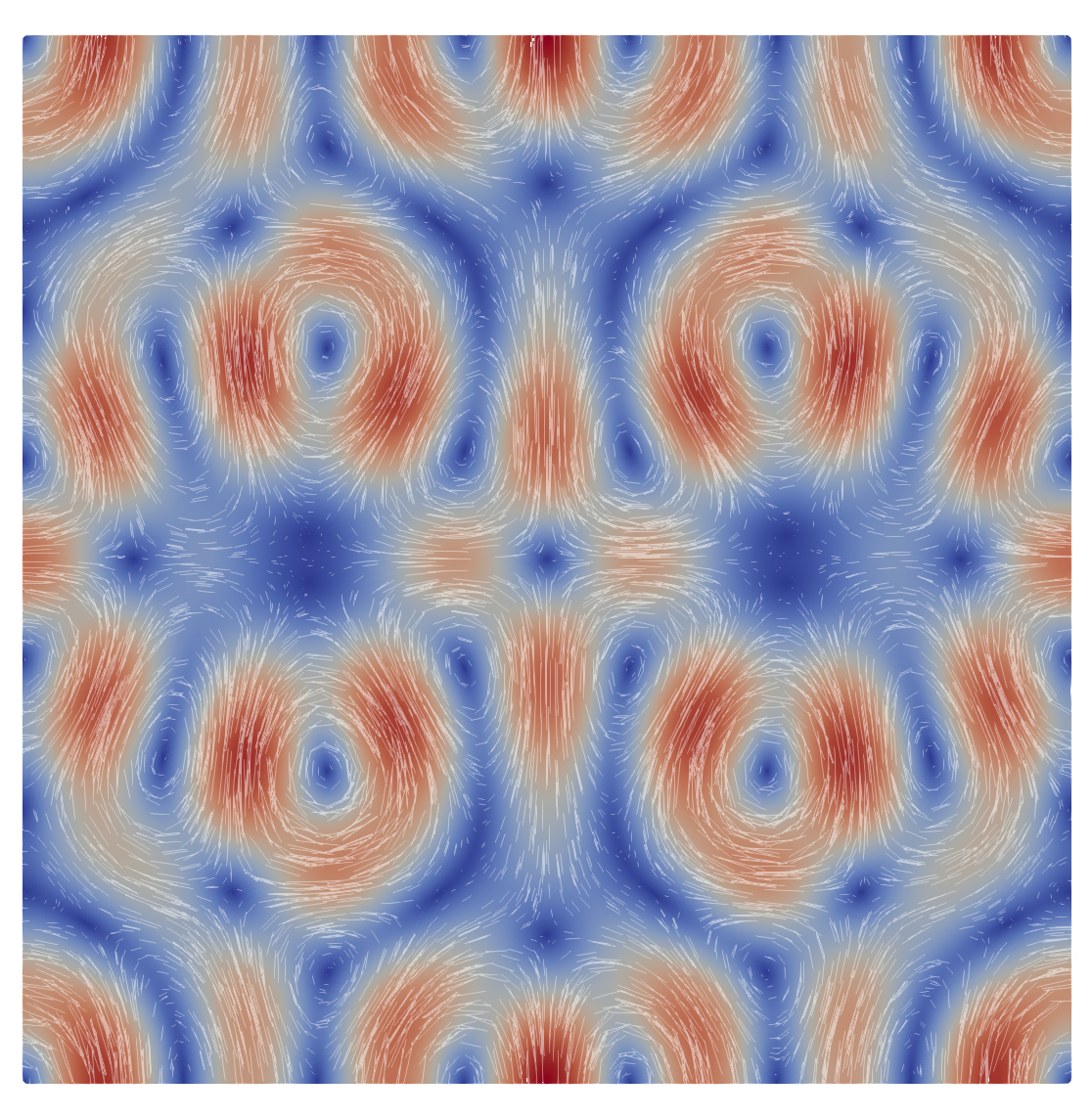}
    }
    \includegraphics[width=0.06\textwidth]{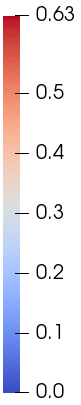}
    \caption{Approximation of the eigenmode corresponding to the eigenvalue $\lambda_{52} = 58$ using different $k$-form spaces. The colorbar on the right displays the magnitude of the vector field while the white lines in the figures indicate its direction. The hierarchy of the approximation spaces starts at $\T_{\lzero}$, i.e., we are in case (i).
    }
    \label{fig:eigenmode_58}
\end{figure}

The first sequence of computed $1$-form eigenmodes shown in Figure~\ref{fig:eigenmode_2} belongs to the third eigenvalue $\lambda_3 = 2$ and is obtained using the approximation spaces $\SDFspaceZero{1}{3}{3} = \FEECspaceZero{1}{3}$ (left), $\SDFspaceZero{1}{3}{7}$ (centre) and $\SDFspaceZero{1}{7}{7} = \FEECspaceZero{1}{7}$ (right).
We observe that the eigenmode of $\SDFspaceZero{1}{3}{7}$ appears visually smoother than the eigenmode of $\SDFspaceZero{1}{3}{3}$ even though both spaces have the same number of DoFs. This effect is a result of using the refined basis functions $\constrQ{\SDFbasis}{1}{3}{7}{\coarseIdxone}$ with increased subdivision-induced regularity (cf. Remark~\ref{remark:subdiv_smoothness}) instead of the standard FE basis functions $\freeQ{\FEECbasis}{1}{3}{\coarseIdxone}$. Further, comparing the approximate eigenmodes obtained using $\SDFspaceZero{1}{3}{7}$ and $\SDFspaceZero{1}{7}{7}$ shows no visible difference. This implies that the space $\SDFspaceZero{1}{3}{7}$ is capable of capturing the same level of detail of the solution of $\SDFspaceZero{1}{7}{7}$ even though it only has a fraction of the DoFs.

The differences between these plots are reflected in the convergence curves of Figure~\ref{fig:approx_errors_k_1}. For a comparison between the centre and right panels we first fix the finest level $\L =7$ and then coarsen by $4$ levels (hence $\l=3$, cf. Section~\ref{subsec:projection}) which gives us the corresponding spaces used for the plots. In terms of accuracy, the convergence plots in Figure~\ref{fig:approx_errors_k_1} reveal that this decreases the accuracy only marginally (i.e. we are at the flat part of the convergence curves), but significantly reduces the number of DoFs. Further, when moving from the left to the centre panel, we observe that the smoothing number $\ns = \L - \l$ increases from $0$ to $4$. Effectively, that means that the simulation has the same number of DoFs but the basis functions are smoother in the centre panel. Therefore, the result in the centre plot of Figure~\ref{fig:eigenmode_2} is more accurate than the one on the left plot.

The second sequence of computed $1$-form eigenmodes in Figure~\ref{fig:eigenmode_58} corresponds to the eigenvalue $\lambda_{52} = 58$ and is obtained by using the same approximation spaces as the first sequence in Figure~\ref{fig:eigenmode_2}. This time, the fields exhibit larger quantitative differences. We can see that the solution obtained using $\SDFspaceZero{1}{3}{7}$ is the most rotationally symmetric while for the two eigenmodes of $\SDFspaceZero{1}{3}{3}$ and $\SDFspaceZero{1}{7}{7}$ various features appear slightly vertically stretched.
As such, using our subdivision $k$-forms on an intermediate level $\l$ with about $\ns = 3,4$ seems to have, at least in this setting, a positive effect on the symmetry of the solution. In contrast, using the coarse and fine FEEC spaces ($\SDFspaceZero{1}{3}{3}$ and $\SDFspaceZero{1}{7}{7}$, respectively) we suspect that the orientation of the triangular grid cells (built up from similarly oriented nested triangles) might impact the symmetry of the fields, an effect that should decrease with higher mesh resolution. In fact, a comparison of the coarser fields in figure (a) with the finer one in (c) confirms that higher resolution indeed reduces this stretching effect. It is subject of ongoing work to better understand the impact of our subdivision $k$-forms on the overall quality (and symmetry) of these eigenmodes.

Hence, relative to comparable FEEC spaces, the subdivision $k$-forms on a suitable intermediate level allow us to save computational resources (using much less DoFs) while achieving comparably accurate numerical solutions for the Maxwell problem and potentially other problems with similarly smooth analytical solutions.

\subsubsection{Computational efficiency of the subdivision $k$-forms}
\label{subsubsubsec:computational_time}

Finally, we evaluate the computational efficiency of the novel subdivision $k$-forms relative to a standard FEEC method.
To this end, we compare the computational simulation times when using the spaces of subdivision $1$-forms $\SDFspaceZero{1}{\l}{\L} \big( \T_{\lzero} \big)$ for different levels $\l$ with each other and with the FEEC spaces $\FEECspaceZero{1}{\L}$.
We restrict ourselves to presenting results for case (i) of Section~\ref{sec:geodomainerror} where we compute only one mesh hierarchy starting from a fixed $\T_{\lzero}$. This is the case that is relevant in practical applications, where we regard the initial level as fixed and obtain nested spaces along the hierarchy.

In the following, we will split the different parts of the algorithm into pre-processing steps and runtime steps. This distinction is only relevant in the case of dynamical systems and hence, does not apply to the Maxwell eigenvalue problem. However, it still provides us with an idea of the potential of subdivision $k$-form spaces for simulations that need to be evolved in time.

\smallskip

When using FEEC, there are two computationally expensive parts: the assembly of the FE matrices and solving the generalized eigenvalue problem. We will denote the durations to carry out these operations by $t_{\text{Assem}}$ and $t_{\text{Solver}}$(FEEC), respectively.

For a fixed finest level $\L$, methods using the subdivision $k$-form spaces need to carry out the same assembly with duration $t_{\text{Assem}}$.
The duration of the unrefinement of the assembled FE matrices on level $\L$ to level $\l$ (cf. Remark~\ref{remark:denser_matrices}) is denoted
by $t_{\text{Unref}}$. Finally, we need to account for the duration $t_{\text{Subd}}$ of running the subdivision algorithm.

After these preprocessing steps, the duration of running the eigensolver when using the subdivision methods is denoted by $t_{\text{Solver}}$(Subd).
In total, i.e. including the preprocessing steps, we find that the subdivision methods save time if
\begin{equation*}
    t_{\text{Solver}}\text{(FEEC)} > t_{\text{Subd}} + t_{\text{Unref}} + t_{\text{Solver}}\text{(Subd)}.
\end{equation*}

\begin{figure}
    \begin{center}
        \includegraphics[trim=0.2cm 0.0cm 0.0cm 0.8cm, clip, width=0.75\textwidth]{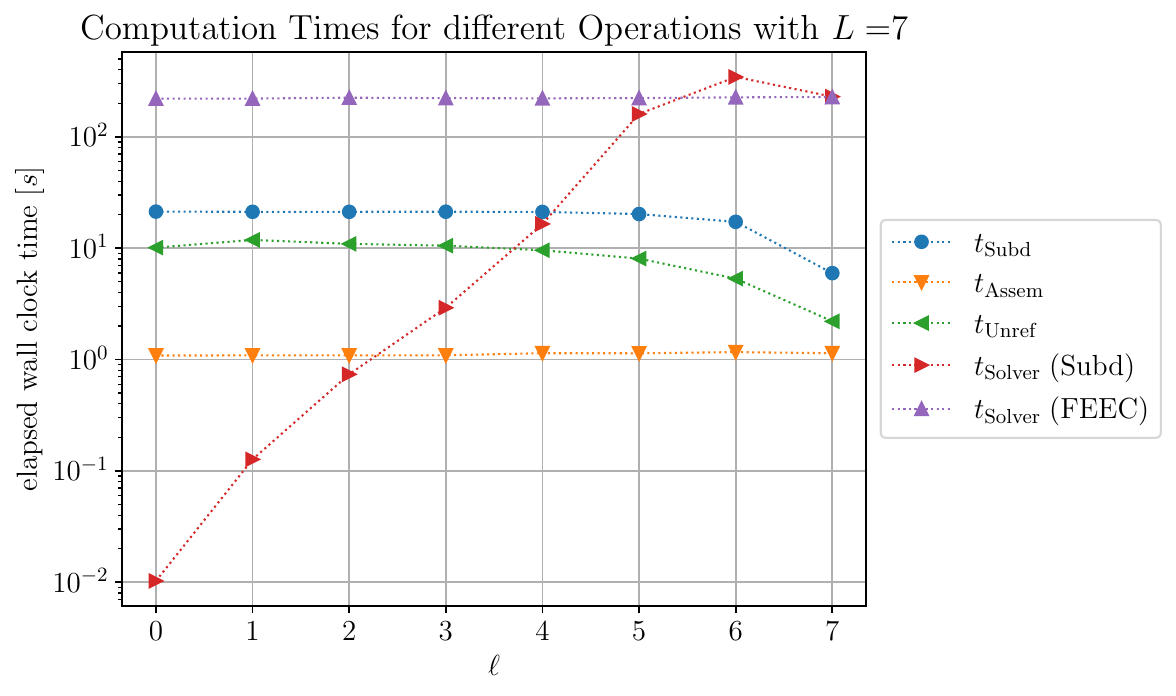}
    \end{center}
    \vspace{-8mm}
    \caption{Wall clock time of certain parts of the subdivision method for $\L=7$.}
    \label{fig:computational_times}
\end{figure}

Figure~\ref{fig:computational_times} displays the wall clock times for solving the Maxwell eigenvalue problem using the spaces $\SDFspaceZero{1}{\l}{7}\big( \T_{\lzero}\big)$ for a fixed finest level $\L =7$ and varying intermediate levels $\l$ and for the FEEC space $\FEECspaceZero{1}{7}$.
Let us first look at the duration $t_{\text{Solver}}$(FEEC) of the eigensolver for the FEEC spaces. The runtime is constant for different levels $\l$ because $\l$ is not a parameter of the FEEC spaces. Similarly, the assembly time $t_{\text{Assem}}$ is constant but about $100$ times smaller than the runtime of the eigensolver.
As mentioned above, this computationally cheap assembly step is shared by the subdivision method.

The computationally more expensive steps that have to be considered when using the subdivision spaces $\SDFspaceZero{1}{\l}{7}$ are the unrefinement and subdivision steps for which the measured durations $t_{\text{Unref}}$ and $t_{\text{Subd}}$, respectively, do depend on level $\l$,
as shown in Figure~\ref{fig:computational_times}. These times become smaller as $\l$ approaches $\L$ because the costs for unrefinement and subdivision are accumulated across fewer levels. These durations are significantly larger than the assembly time but, fortunately, these steps can be precomputed. Compared to the runtime of the FEEC solver,
both durations are at least 10 times smaller.

The last crucial step of the subdivision method (after the preprocessing steps) is running the subdivision eigensolver. The duration $t_{\text{Solver}}$(Subd) increases proportionally to level $\l$ as long as $\l \leq \L\!-\!1$. Notably, for $\l = \L \! - \! 1$, the eigensolver takes longer than for the FEEC method using $\FEECspaceZero{1}{\L}$. In this particular case, the benefit of reducing the matrix dimension does not compensate for the additional computational costs due to the increased density of the subdivision FE matrices. Finally, for $\l = \L$, we observe almost identical wall clock times because $\SDFspaceZero{1}{\L}{\L} = \FEECspaceZero{1}{\L}$.

Overall, for all $\l \leq \L-2$, we observe that subdivision methods save computational time while almost preserving the accuracy (recall Figures~\ref{fig:approx_errors_k_0_with_nr},~\ref{fig:approx_errors_k_1}, and~\ref{fig:approx_errors_k_2}).

\smallskip

Table~\ref{tab:compute_times} provides more details about the relative errors and computational times of the subdivision methods. 
The relative error $\vert\vert{e_{50}}\vert\vert$ measures the average deviation of the first $50$ eigenvalues of $\SDFspaceZero{1}{\l}{7}$ compared to $\FEECspaceZero{1}{7}$:
\begin{equation}
    \label{eq:average_error}
    \vert\vert e_{50} \vert\vert \coloneq \tfrac{1}{50} \; \sum\limits_{i = 1}^{50} \left\vert \frac{(\lambda_h)_i \big(\SDFspaceZero{1}{\l}{7}\big)}{(\lambda_h)_i  \big(\FEECspaceZero{1}{7}\big)} - 1 \right\vert.
\end{equation}
Since this error is defined relative to the FEEC space $\FEECspaceZero{1}{7}$ on the finest mesh $\T_\L$, both $\SDFspaceZero{1}{\l}{7}$ and $\FEECspaceZero{1}{7}$ have the same geometric domain error. This ensures that the eigenvalues computed with $\SDFspaceZero{1}{\l}{7}$ converge to those computed with $\FEECspaceZero{1}{7}$
and, hence, the eigenvalue errors $\vert\vert e_{50} \vert\vert$ according to Eq.~\eqref{eq:average_error} converge to $0$ as $\l$ increases. In fact, $\vert\vert e_{50} \vert\vert = 0$ for $\SDFspaceZero{1}{7}{7} = \FEECspaceZero{1}{7}$.

\renewcommand{\arraystretch}{1.4}
\begin{table}
\begin{footnotesize}

    \centering
    \begin{tabular}{|c|cc|ccc|c|cc|}
        \hline
        \multirow{2}{*}{space} & \multicolumn{2}{c|}{approximation properties} & \multicolumn{3}{c|}{pre-processing times $[s]$} & runtime $[s]$ & \multicolumn{2}{c|}{speed-ups} \\
         & $\left\vert\left\vert {e}_{50} \right\vert\right\vert$ & $\mathrm{dim}\big(\text{space}\big)$ & $t_{\text{Subd}}$ & $t_{\text{Assem}}$ & $t_{\text{Unref}}$ & $t_{\text{Solver}}$ & $\eta_{\text{run}}$ & $\eta_{\text{tot}}$ \\ \hline 
        $\FEECspaceZero{1}{7}$ & --- & $2701312$ & --- & $1.14$ & --- & $228.0$ & --- & --- \\ \hline
        $\SDFspaceZero{1}{0}{7}$ & $6.57 \cdot 10^{-3}$ & $149$ & $21.30$ & $1.09$ & $1.01$ & $0.01$ & $22118$ & $6.91$ \\
        $\SDFspaceZero{1}{1}{7}$ & $7.70 \cdot 10^{-4}$ & $628$ & $21.22$ & $1.09$ & $1.19$ & $0.13$ & $1797$ & $6.56$ \\
        $\SDFspaceZero{1}{2}{7}$ & $1.93 \cdot 10^{-4}$ & $2576$ & $21.22$ & $1.09$ & $1.09$ & $0.74$ & $309.3$ & $6.67$ \\
        $\SDFspaceZero{1}{3}{7}$ & $5.32 \cdot 10^{-5}$ & $10432$ & $21.14$ & $1.09$ & $1.05$ & $2.92$ & $78.1$ & $6.38$ \\
        $\SDFspaceZero{1}{4}{7}$ & $1.49 \cdot 10^{-5}$ & $41984$ & $21.17$ & $1.14$ & $0.96$ & $16.50$ & $13.81$ & $4.40$ \\
        $\SDFspaceZero{1}{5}{7}$ & $4.12 \cdot 10^{-6}$ & $168448$ & $20.26$ & $1.14$ & $0.81$ & $160.5$ & $1.420$ & $1.52$ \\
        $\SDFspaceZero{1}{6}{7}$ & $9.95 \cdot 10^{-7}$ & $674816$ & $17.25$ & $1.17$ & $0.53$ & $345.3$ & $0.621$ & $0.79$ \\
        $\SDFspaceZero{1}{7}{7}$ & $0$ & $2701312$ & $5.967$ & $1.14$ & $0.22$ & $229.5$ & $1.0$ & $0.96$ \\ \hline
    \end{tabular}
    \caption{Overview of computational times and approximation properties of the spaces $\SDFspaceZero{1}{\l}{\L}$ with $\L \leq 7$ compared to a reference solution obtained with $\FEECspaceZero{1}{7} = \SDFspaceZero{1}{7}{7}$. The row called $\left\vert\left\vert {e}_{50} \right\vert\right\vert$ contains the average deviation of the first 50 eigenvalues compared to $\FEECspaceZero{1}{7}$ as defined in Eq.~\eqref{eq:average_error}. The respective dimension of the approximation space can be found in the column $\mathrm{dim}(\text{space})$. The computational times are divided into two categories depending on whether they are considered a pre-processing step or happen at runtime of the numerical algorithm. The recorded wall clock times (in seconds) are the subdivision time $t_{\text{Subd}}$, the assembly time $t_{\text{Assem}}$, the unrefinement time $t_{\text{Unref}}$ and the solver time $t_{\text{Solve}}$. The last column shows speed-up ratios $\eta_{\text{run}}$ and $\eta_{\text{tot}}$ of the runtime and total time, respectively, as they were defined in Eq.~\eqref{eq:speed-ups}.}
    \label{tab:compute_times}
\end{footnotesize}
\end{table}

The table also presents the measured \emph{runtime speed-up} and \emph{total speed-up}, respectively,
\begin{equation}
    \label{eq:speed-ups}
    \eta_{\text{run}} \coloneqq \frac{t_{\text{Solver}}\text{(FEEC)}}{t_{\text{Solver}}\text{(Subd)}} \qquad \text{and} \qquad \eta_{tot} = \frac{t_{\text{Assem}} + t_{\text{Solver}}\text{(FEEC)}}{t_{\text{Subd}} + t_{\text{Assem}} + t_{\text{Unref}} + t_{\text{Solver}}\text{(Subd)}}.
\end{equation}
The table confirms that speed-ups can be achieved since subdivision spaces with intermediate mesh level $\l$ possess fewer DoFs, leading to a reduction in runtime of the eigensolver. A good trade-off between maintaining accuracy and speeding up the computations seems to be around $\l = \L-3$ or $\l=\L-4$ (i.e. a coarsening of $3$ or $4$ levels, respectively), where we achieve a total speed-up factor of $4$ to $6$ and an accuracy of still $10^{-5}$ (sacrificing only about an order of magnitude in accuracy compared to the finest FEEC solutions).
This sweet spot is consistent with the findings in Section~\ref{subsec:projection}.

With a factor of more than $10$, the runtime speed-up is even higher due to the significantly reduced number of degrees of freedom. Especially when solving problems that require many time steps, the computational time required for the preprocessing gets less important and thus the total speed-up approaches the runtime speed-up.
This suggests that the spaces of subdivision $k$-forms can significantly speed up the simulation of dynamical systems which are often integrated over a long time horizon.

\section{Conclusion}\label{sec_conclusion}

This paper introduced a sequence of compatible subdivision $k$-form spaces on irregular triangular meshes that feature enhanced subdivision-induced regularity over conventional $k$-forms. This sequence of spaces, which is built upon standard finite element (FE) spaces, is considered compatible if the spaces comprise a discrete de Rham complex. Therefore, discretisations of PDEs using such sequence of subdivision spaces are structure-preserving, stable and consistent while preserving important vector identity properties \citep{Arnold.2006,Arnold.FEEC}.
The subdivision-induced regularity is a consequence of applying a subdivision algorithm to construct refined basis functions along a hierarchy of increasingly fine meshes. While the subdivision $k$-form spaces do not possess higher Sobolev regularity, the basis functions of the subdivision spaces behave similarly to spline basis functions in the sense that their supports are larger, that they appear visually smoother, and that they exhibit enhanced convergence rates compared to the lowest-order FEEC complex.

The abstract framework around this approach applies to several known instances of compatible $k$-form subdivision schemes. That is, our framework for compatible subdivision $k$-form spaces on arbitrary triangulations can be applied to multiple other subdivision schemes besides the Discrete Exterior Calculus-compatible schemes established in \cite{Wang.2006}; for instance, to the second sequence of Wang schemes for quadrilateral faces, the $\sqrt{3}$-subdivision based sequence in \cite{Huang2012}, and the halfedge-based subdivision scheme in \cite{Custers.2020}. Further, we formalised the process of incorporating vanishing trace boundary conditions and showed that the resulting spaces also comprise discrete de Rham complexes because they are isomorphic to the lowest-order FEEC complex and thus to the relative simplicial complex.

We analysed the properties of the proposed subdivision $k$-form spaces numerically through a convergence study of the projection error of the subdivision $k$-forms against some analytical differential $k$-forms and a study of the Maxwell eigenvalue problem. We found that the observed convergence rate depends on the smoothing number, i.e. the number of levels between the level that defines the degrees of freedom (DoFs) and the level to which we refine the associated basis functions. There exists a critical smoothing number above which the approximation accuracy starts to stagnate. For the Wang schemes used in this work, these critical smoothing numbers are between $3$ and $5$ depending on the form degree $k$. Similarly, we can consider how often the mesh can be coarsened before defining the DoFs without losing too much accuracy. The Maxwell eigenvalue test case allowed coarsening of $3$ to $4$ levels.

These results imply that the subdivision $k$-form spaces can be employed for two purposes: either a coarse FE space is smoothed up to the critical smoothing number to increase the accuracy of the approximation without introducing additional DoFs, or a fine FE space can be coarsened up to the critical smoothing number to decrease the number of DoFs by up to two orders of magnitude while almost preserving the accuracy of the corresponding fine FE space.

The Maxwell eigenvalue problem was used to verify that the subdivision $k$-form spaces indeed comprise discrete de Rham complexes. We report convergence towards the correct spectrum and observe no spurious modes. Further, we were able to show that the subdivision $k$-form spaces also have the potential to reduce the wall clock time needed to simulate the Maxwell eigenvalue problem by a factor of up to $6$.

The results presented in this paper can be extended in several ways. One interesting avenue of future research would be to directly use the limit functions for simulations. This would lead to differential form spaces with increased smoothness. For example, the limit functions of Loop subdivision are $\mathcal{C}^2$ everywhere except for at the irregular vertices where they are $\mathcal{C}^1$. As a result, we would be able to construct FE spaces that discretise complexes with higher order derivative operators. The key problem that arises from using the limit functions is the numerical integration. The limit functions of the Loop subdivision scheme we used for $0$-forms are infinite sequences of polynomial patches that become increasingly dense around irregular vertices. There exist approaches to integrate such functions using quadrature and a truncation of the sequence of polynomial patches, see for example \cite{Juettler.2016}. However, it is not clear how these ideas extend to our $1$- and $2$-form subdivision scheme and how the truncation can be done consistently across the $k$-form spaces to maintain their compatibility.

Another future research direction could attempt to leverage our novel subdivision $k$-form spaces to construct hierarchical spaces. Since the basis functions satisfy a refinement relation, i.e., a coarse basis function can exactly be expressed as a linear combination of finer basis functions, one can consider spaces constructed by iteratively refining the function space by replacing coarse basis functions by a selection of finer basis functions. Doing this consistently across the $k$-form spaces may lead to nested hierarchical spaces that maintain a complex. Similar ideas exist for B-spline spaces \citep{Evans.2020}, where it was shown that the cohomology of the resulting complex is a rather subtle issue. For this reason, we would expect similarly interesting results for analogous hierarchical subdivision $k$-form spaces. Given that such conditions can always be met, the hierarchical spaces could provide a promising avenue to compatible, adaptive simulations on arbitrary triangulations.

\paragraph{Acknowledgements.}
RP acknowledges support from a PGR Studentship funded by the Engineering and Physical Sciences Research Council (EPSRC). The authors are grateful to C. Lessig and M. Desbrun for their support and insightful discussions during the early stages of this project.

\appendix

\renewcommand{\theHsection}{A\arabic{section}}

\section{Additional details about subdivision $k$-form spaces}
\label{app:subdiv_spaces}

This appendix provides additional details regarding the construction and properties of the subdivision $k$-form spaces introduced in Section~\ref{subsec:subdiv_k_forms_through_subdiv_algorithm}. Eq.~\eqref{eq:choice_of_feec_space} fixed the FE spaces upon which the subdivision $k$-form spaces are built. Definition~\ref{def:subdiv_operator} introduces the $k$-form subdivision operators and their coordinate expressions in the form of subdivision matrices $\Smat{k}{\l}{\L}$. The entries of these matrices are determined by the $k$-form subdivision rules.

\begin{definition}[$k$-form subdivision rule]
    \label{def:k_form_subdivrule}
    Analogous to the mesh subdivision rules derived in Section~\ref{sec_subdivalg_formal}, a $k$-form subdivision rule $\mathfrak{S}^k$ is a map
    \begin{equation}
    \begin{split}
        \mathfrak{S}^k: \; \K^k_{\l+1} \to \big[ \K^k_\l \times \mathbb{R} \big]^M, \qquad \simplex{k}{\l+1}{\fineIdxone} \mapsto \Big\{ \big( \simplex{k}{\l}{\coarseIdxone_1}, \; (w^k_\l)_{\coarseIdxone_1} \big), \hdots, \big( \simplex{k}{\l}{\coarseIdxone_M}, \; (w^k_\l)_{\coarseIdxone_M} \big) \Big\} .
    \end{split}
    \end{equation}
    As for mesh subdivision rules, the number $M$ and the indices $\coarseIdxone_\coarseIdxtwo, \; \coarseIdxtwo = 1, \hdots, M$ of the coarse simplices and their weights $(w^k_\l)_{\coarseIdxone_\coarseIdxtwo}$ depend on the chosen $k$-form subdivision scheme and the properties of the fine simplex $\simplex{k}{\l+1}{\fineIdxone}$.
\end{definition}

\begin{remark}
    \label{remark:mesh_and_0_form_subdiv_scheme_need_to_be_same}
    We can think of carrying out mesh subdivision as constructing a chart for the manifold mesh on which the subdivision $k$-forms are defined. For reasons of consistency, we always assume that the mesh subdivision scheme and the $0$-form scheme are the same, i.e. $\mathfrak{S}\big(\vertex{\l+1}{\fineIdxone}) = \mathfrak{S}^0\big(\vertex{\l+1}{\fineIdxone})$.
\end{remark}

A collection of $k$-form subdivision rules, one for each $k$-simplex, constitutes a \emph{$k$-form subdivision scheme}. The individual $k$-form subdivision rules from Definition~\ref{def:k_form_subdivrule} provide the entries of the $k$-form subdivision matrices defined in Definition~\ref{def:subdiv_operator} such that
\begin{equation}
    \label{eq:subdiv_matrix_from_subdiv_rule}
    \Smatij{k}{\l}{\l+1}{\fineIdxone}{\coarseIdxone_\coarseIdxtwo } = \begin{cases}
        (w^k_\l)_{\coarseIdxone_\coarseIdxtwo} \quad &\text{if} \quad \big( \simplex{k}{\l}{\coarseIdxone_\coarseIdxtwo}, \; (w^k_\l)_{\coarseIdxone_\coarseIdxtwo} \big) \in \mathfrak{S}^k \big(\simplex{k}{\l+1}{\fineIdxone} \big), \\
        0, &\text{otherwise}.
    \end{cases}
\end{equation}

The map induced by $\Aop{k}{\l}{\L}$ is \emph{not onto} whenever $\l < \L$ since it maps from the low-dimensional space $\FEECspace{k}{\l}$ to the higher-dimensional space $\FEECspace{k}{\L}$. This observation motivated Proposition~\ref{prop:subspace_of_Vk} which is proven next.

\begin{proof}[Proof of Proposition~\ref{prop:subspace_of_Vk}]
\label{proof:subspace_of_Vk}
Statement (i) follows trivially from Definition~(\ref{def:uniform_subdiv_spaces}) with $\Aop{k}{\l}{\l}$ the identity.
To prove statement (ii), we note that the space $\SDFspace{k}{\L}{\l}$ is induced by the operator $\Aop{k}{\l}{\L}$ with coordinate expression $\Amat{k}{\l}{\L}$. By Eq.~\eqref{eq:accumulated_subdiv_matrix}, $\Amat{k}{\l}{\L}$ has the dimensions $\dimlk{\L}{k} \times \dimlk{\l}{k}$ with
$\dimlk{\l}{k} < \dimlk{\L}{k}$ and thus
\begin{equation}
    \mathrm{dim}\big( \SDFspace{k}{\L}{\l} \big) = \mathrm{dim}\Big( \mathrm{img} \big( \Aop{k}{\l}{\L} \big) \Big)
    = \mathrm{rank}\big(  \Amat{k}{\l}{\L}  \big) = \dimlk{\l}{k} < \dimlk{\L}{k}
\end{equation}
because subdivision matrices (typically) have full column rank.
\end{proof}

\smallskip

In the main part of this paper, we construct a basis for the subdivision $k$-form spaces, see Section~\ref{subsec:basis_for_subdiv_spaces}. At the end of that section, we state Proposition~\ref{prop:support} that describes the support of the constructed basis functions $\constrQ{\SDFbasis}{k}{\l}{\lone}{\coarseIdxone}$. We omitted the proof in the main part but present it here.

\begin{numberedproof}[Proof of Proposition~\ref{prop:support}]
    \label{proof:support}

    The proof proceeds in two stages. First, we are going to show that the support of the basis function $\constrQ{\SDFbasis}{k}{\l}{\lone}{\coarseIdxone}$ cannot be an arbitrary set of fine faces but has to be a set of refined coarse faces. Once that has been established, we will identify the coarse faces whose refinements constitute the support of the basis function $\constrQ{\SDFbasis}{k}{\l}{\lone}{\coarseIdxone}$.

    \smallskip

    For the first stage of the proof, note that the support of any refined basis function $\constrQ{\SDFbasis}{k}{\l}{\lone}{\coarseIdxone}$ has to be a set of coarse faces, refined from level $\l$ to $\lone$, because every coefficient of $\auxincl \big(\constrQ{\SDFbasis}{k}{\l}{\lone}{\coarseIdxone}\big) \in \FEECspace{k}{\lone}$ (on $\lone$) is a linear combination of coefficients (on $\l$) associated to the coarse simplices $\simplex{k}{\l}{\coarseIdxone}$ according to Eq.~\eqref{eq:subdiv_rule_def}.
    Thus, if a fine face $f_{\lone} \in \FFltoL{\l}{\lone}\big( \face{\l}{\coarseIdxtwo}\big)$ is contained in the support of $\constrQ{\SDFbasis}{k}{\l}{\lone}{\coarseIdxone}$, then all fine faces in $\FFltoL{\l}{\lone}\big( \face{\l}{\coarseIdxtwo}\big) \subset \F_{\lone}$ associated to the same coarse face $\face{\l}{\coarseIdxtwo}$ have to be contained in the support of $\constrQ{\SDFbasis}{k}{\l}{\lone}{\coarseIdxone}$, cf. Figure~\ref{fig:basis_function_plots} in the main part.
    This fact allows us to write the support of the fine basis function as a
    set of refined coarse faces. We denote this set of coarse faces as $(\simplexsubset{2}{\l})_\coarseIdxone \subset \K^2_{\l}$ and we write
    \begin{equation}
        \support\big( \constrQ{\SDFbasis}{k}{\l}{\lone}{\coarseIdxone} \big) = \mathrm{Int}\big( \FFltoL{\l}{\lone} \big( (\simplexsubset{2}{\l})_\coarseIdxone \big) \big),
    \end{equation}
    where Int denotes the interior of a set and guarantees that the support is an open set.

    \smallskip

    It remains to identify all coarse faces of the set $(\simplexsubset{2}{\l})_\coarseIdxone \subset \K^2_{\l}$ that are included in $(\simplexsubset{2}{\l})_\coarseIdxone$. From a careful investigation of the Wang subdivision rules, see for example the supplemental material of \cite{Goes.2016}, follows that the support of the basis functions can not propagate over the two-ring $\FFltoL{\l}{\lone} \big(\FVl{\l} \circ \VFl{\l} \circ \F\K^k_{\l} \big( \simplex{k}{\l}{\coarseIdxone}\big)$ of refined faces (see Eqs.~\eqref{eq:adjacency_for_sets} and~\eqref{facesplitting_union} for definitions of the operators). This is because all subdivision rules only use the immediate neighbours of any simplex with non-zero coefficient. Repeating the process of taking neighbours and refining the mesh then results in the fine two-ring of faces. The corresponding coarse two-ring of faces constitutes the set $(\simplexsubset{2}{\l})_\coarseIdxone \subset \K^2_{\l}$ which finalizes the proof.
    \qed
\end{numberedproof}

\section{Subdivision k-form spaces with vanishing trace boundary conditions}
\label{app:vanishing_boundary_complex}

This appendix contains additional details regarding the definition of subdivision $k$-form spaces with vanishing trace boundary conditions (BC) introduced in Section~\ref{subsec:complex_with_vanishing_trace}. The appendix culminates in the proof of Theorem~\ref{theo:complex_with_vanishing_trace}, where we show that the spaces still comprise de Rham complexes after we enforce the vanishing trace BC.

Let $\M$ be a manifold with boundary $\partial \M$ and denote the inclusion map by $j: \partial\M \to \M$. The trace operator $\trace$ is defined as the pullback along the inclusion map, i.e. $\trace \coloneq j^*$ (cf.~\cite{Arnold.2006}). The trace operator restricts a differential $k$-form to the manifold boundary in a meaningful way. For any $\alpha \in \Omega^0(\M)$, $\beta \in \Omega^1(\M)$ and $\gamma \in \Omega^2(\M)$, the trace operator becomes
\begin{equation}
    \trace(\alpha) = \alpha \vert_{\partial \M}, \qquad \trace(\beta) = \vec{\beta} \vert_{\partial \M} \cdot \vec{t}, \qquad \trace(\gamma) = 0,
\end{equation}
where $\vec{\beta}$ and $\vec{t}(x)$ denote the proxy vector field of $\beta$ and the tangent vector of $\partial \M$ at the point $x \in \partial \M$, respectively. Enforcing vanishing trace BC for the spaces of the $L^2$ de Rham complex from Eq.~\eqref{eq:infinite_dim_de_Rham} effectively restricts the $0$- and $1$-form spaces to the subspaces
\begin{align}
    \mathring{H}^1(\M) &= \big\{ \alpha \in H^1(\M) \;\; \colon \;\; \trace(\alpha) = 0 \big\}, \\
    \HocurlM{\M} &= \big\{ \beta \in \HcurlM{\M} \;\; \colon \;\; \trace(\beta) = 0 \big\}.
\end{align}
The Sobolev spaces with vanishing trace BC comprise a \emph{relative} de Rham complex
\begin{equation}
    \label{eq:infinite_dim_relative_de_Rham}
    \begin{tikzcd}[row sep=small, column sep=normal]
        0 \arrow[r] & \mathring{H}^1(\M) \arrow[r, "\extd"] & \HocurlM{\M} \arrow[r, "\extd"] & L^2(\M) \arrow[r] & 0.
    \end{tikzcd}
\end{equation}
In our context, the relative de Rham complex is what we obtain after applying (vanishing trace) BC to the usual (absolute) de Rham complex in Eq.~\eqref{eq:infinite_dim_de_Rham}. Note that the cohomology groups $\mathcal{H}^k_{\mathrm{rel}}$ are different from the cohomology groups $\mathcal{H}^k$ of the absolute de Rham complex. For example, for a simply connected domain, the relative cohomology groups are $\mathcal{H}^0_{\mathrm{rel}} = \mathcal{H}^1_{\mathrm{rel}} = \{ 0\}$ and $\mathcal{H}^2_{\mathrm{rel}} \cong \mathbb{R}$ instead of $\mathcal{H}^0 \cong \mathbb{R}$ and $\mathcal{H}^1 = \mathcal{H}^2 = \{ 0 \}$ for the absolute de Rham complex.

The remainder of this section will first introduce the lowest-order (relative) FEEC complex with vanishing trace BC, and then build the subdivision $k$-form spaces with vanishing trace BC on top of them. The FEEC spaces with vanishing trace BC are defined as:
\begin{equation}
    \label{eq:FEEC_spaces_with_vanishing_trace}
    \FEECspaceZero{k}{\l} \coloneqq \big\{ \freeVec{\omega}{k}{\l} \in \FEECspace{k}{\l} \; \colon \;\; \trace\big( \freeVec{\omega}{k}{\l} \big) = 0 \big\}, \quad \text{for} \; k \in \{0, 1, 2\},
\end{equation}
according to FEEC~\citep{Arnold.FEEC}. 

Eq.~\eqref{def:interior_simplices} defines interior simplices $\Kint_\l$ and allows us to describe the spaces $\FEECspaceZero{k}{\l}$ in terms of interior basis functions:
\begin{equation}
    \label{eq:BCFEEC_spaces_in_basis}
    \FEECspaceZero{0}{\l} = \mathrm{span} \big( \bigcup\limits_{\vertex{\l}{\coarseIdxone} \in \Vint_{\l}} \freeQ{\FEECbasis}{0}{\l}{\coarseIdxone} \big), \qquad \FEECspaceZero{1}{\l} = \mathrm{span} \big( \bigcup\limits_{\edge{\l}{\coarseIdxtwo} \in \Eint_{\l}} \freeQ{\FEECbasis}{1}{\l}{\coarseIdxtwo} \big), \qquad \FEECspaceZero{2}{\l} = \FEECspace{2}{\l}.
\end{equation}

In other words, the spaces $\FEECspaceZero{k}{\l}$ are obtained by discarding the basis function that are associated to boundary vertices and edges, i.e., these spaces are spanned by all basis functions that have vanishing traces, according to Eq.~\eqref{eq:FEEC_spaces_with_vanishing_trace}. This means that we treat the vanishing trace BC as essential BC and incorporate them into the approximation spaces directly. Note that the resulting complex $(\FEECspaceZero{k}{\l}, \extd)$ is a sub-complex of the relative de Rham complex in Eq.~\eqref{eq:infinite_dim_relative_de_Rham} and thus, their cohomology groups are isomorphic.

The remainder of this section sets up subdivision $k$-form spaces that mimic the structure of the spaces $\FEECspaceZero{k}{\l}$. First, we need to highlight how subdivision schemes, in particular the Wang schemes, behave close to the boundary of the domain.

\begin{assumption}
    \label{ass:subdiv_preserves_vanishing_trace}
    In the following, we assume that the subdivision operators $\Aop{k}{\l}{\L}$ preserve vanishing traces, i.e.,
    \begin{equation}
        \label{eq:subdiv_preserves_vanishing_trace}
        \Aop{k}{\l}{\L} : \FEECspaceZero{k}{\l} \to \FEECspaceZero{k}{\L}.
    \end{equation}
\end{assumption}
This is a rather mild assumption because most subdivision schemes are designed to satisfy this property, for example by adopting 1D subdivision schemes for the boundary simplices, see remark~\ref{remark:remark_vanishing_trace} for more information.

\begin{remark}[Regarding the assumption in Eq.~\eqref{eq:subdiv_preserves_vanishing_trace}]
    \label{remark:remark_vanishing_trace}
    Different subdivision schemes behave differently close to the boundary of the mesh. Their boundary rules can vary with regards to the selected coarse simplices and their associated weights. In particular, there are two important cases for the purpose of enforcing boundary conditions: a fine boundary simplex can either depend only on coarse boundary simplices or on both boundary and interior simplices of the coarse mesh.

    The subdivision rules for boundary simplices are often chosen to reproduce e.g. one-dimensional B-spline subdivision schemes and thus only depend on coarse boundary simplices. From a subdivision perspective, this behaviour is desirable because two domains with a common boundary will still share the refined boundary after applying subdivision to each of the domains. Such contact between two objects happens abundantly in computer graphics. From a finite elements perspective, we can think of that as preserving vanishing traces under subdivision of $0$-forms. The commutation relations in Lemma~\ref{lemma:compatibility_condition_in_coordinates} imply that the $1$-form subdivision scheme cannot re-introduce any dependence of fine boundary edges on coarse interior edges.
\end{remark}

The following proposition shows that the Wang schemes for triangular meshes satisfy the assumption in Eq.~\eqref{eq:subdiv_preserves_vanishing_trace}.

\begin{proposition}
    \label{prop:no_support_at_bdry}
    Any Wang-scheme induced basis function $\constrQ{\SDFbasis}{k}{\l}{\lone}{\coarseIdxone}$ associated to an interior simplex $\simplex{k}{\l}{\coarseIdxone} \in \Kint_{\l}$ from Eq.~\eqref{def:interior_simplices} has a vanishing trace on the boundary $\partial\T_{\lone}$ of the finer mesh on level $\lone > \l$, i.e.
    \begin{equation}
        \trace \;\big( \constrQ{\SDFbasis}{k}{\l}{\lone}{\coarseIdxone} \big) = 0 \qquad \forall \simplex{k}{\l}{\coarseIdxone} \in \Kint_{\l}.
    \end{equation}
\end{proposition}
\begin{proof}
    \label{proof:no_support_at_bdry}
    The claim follows from the subdivision rules. The subdivision rules for boundary simplices only depend on boundary simplices (cf. \cite{Goes.2016}) and thus they can only pick up fine non-zero coefficients if there was a non-zero coefficient associated to a coarse boundary simplex. As this is not the case for non-boundary basis functions, the claim follows.
\end{proof}

\begin{proposition}
    \label{prop:Wang_subdiv_preserves_zero_trace}
    The Wang subdivision schemes satisfy Eq.~\eqref{eq:subdiv_preserves_vanishing_trace}.
\end{proposition}
\begin{proof}
    For any $\freeVec{\omega}{k}{\l} \in \FEECspaceZero{k}{\l}$, we compute
    \begin{equation}
        \Aop{k}{\l}{\L} \big( \freeVec{\omega}{k}{\l} \big) = \Aop{k}{\l}{\L} \big( \sum\limits_{\simplex{k}{\l}{\coarseIdxone} \, \in \, \Kint_{\l}} \; \freeQ{c}{k}{\l}{\coarseIdxone} \, \freeQ{\FEECbasis}{k}{\l}{\coarseIdxone} \big) = \sum\limits_{\simplex{k}{\l}{\coarseIdxone} \, \in \, \Kint_{\l}} \; \freeQ{c}{k}{\l}{\coarseIdxone} \, \constrQ{\SDFbasis}{k}{\l}{\L}{\coarseIdxone}.
    \end{equation}
    Using Proposition~\ref{prop:no_support_at_bdry}, we find $\trace \big( \Aop{k}{\l}{\L} ( \freeVec{\omega}{k}{\l} ) \big) = 0$ because the sum only contains the basis functions associated to an interior simplex.
\end{proof}

\smallskip

Proposition~\ref{prop:no_support_at_bdry} implies that the only basis functions with non-zero trace are the ones associated with boundary simplices. Thus, the simplest way of enforcing vanishing trace BC is to omit these basis functions from the spaces. This is what lead to Definition~\ref{def:subdiv_spaces_with_vanishing_trace} in the main part. The following is a corollary of that definition.

\begin{corollary}
    $\SDFspaceZero{k}{\l}{\L} = \SDFspace{k}{\l}{\L} \cap \FEECspaceZero{k}{\L}$.
\end{corollary}

The remainder of this appendix sets up operators and derives properties of the subdivision $k$-form spaces $\SDFspaceZero{k}{\l}{\L}$ with vanishing trace BC analogously to the regular spaces $\SDFspace{k}{\l}{\L}$. One of the main differences between these spaces are their respective numbers of dimensions.

\begin{proposition}
    \label{prop:dimension_of_reduced_spaces}
    The spaces $\SDFspaceZero{k}{\l}{\L}$ have the dimensions
    \begin{equation}
        \mathrm{dim} \big( \SDFspaceZero{k}{\l}{\L} \big) = \big\vert \Kint_{\l} \big\vert.
    \end{equation}
\end{proposition}
\begin{proof}
    The claim follows immediately from Definition~\ref{def:subdiv_spaces_with_vanishing_trace} and the fact that all basis functions in $\constrVec{\SDFbasesZero}{k}{\l}{\L}$ are linearly independent.
\end{proof}

The next definition establishes the subdivision operator for the spaces $\SDFspaceZero{k}{\l}{\L}$.

\begin{definition}
    \label{def:subdiv_op_for_BC_spaces}
    The restricted subdivision operator $\BCAop{k}{\l}{\L}$ is defined as
    \begin{equation}
        \BCAop{k}{\l}{\L} \coloneq \Aop{k}{\l}{\L} \Big|_{\FEECspaceZero{k}{\l}} \;\; : \FEECspaceZero{k}{\l} \to \SDFspaceZero{k}{\L}{\l}
    \end{equation}
    Its matrix expression $\BCAmat{k}{\l}{\L}$ can be obtained by deleting the columns of $\Amat{k}{\l}{\L}$ that correspond to boundary simplices $\K^k(\partial\T_\l)$.
\end{definition}
The definition above hinges on Assumption~\ref{ass:subdiv_preserves_vanishing_trace}. If the assumption did not hold, restricting the regular subdivision operator $\Aop{k}{\l}{\L}$ to $\FEECspaceZero{k}{\l}$ would not be sufficient to ensure that $\BCAop{k}{\l}{\L}$ maps to $\SDFspaceZero{k}{\L}{\l} \subset \FEECspaceZero{k}{\L}$ because it would introduce a non-zero trace.

The restricted subdivision operator allows us to define the basis exchange operator for subdivision $k$-form spaces with vanishing trace BC.

\begin{definition}
    \label{def:BE_for_BC_spaces}
    The restricted basis exchange operator $\BCBEop{k}{\l}{\l}{\L}$ is defined as
    \begin{equation}\label{rbe_iso}
        \BCBEop{k}{\l}{\l}{\L} \coloneq \dualAop{k}{\l}{\l}{\L} \Big|_{\FEECspaceZero{k}{\l}} \;\; : \FEECspaceZero{k}{\l} \to \SDFspaceZero{k}{\l}{\L}
    \end{equation}
    Its matrix expression is given by a $\vert \Kint_{\l} \vert \times \vert \Kint_{\l} \vert$ identity matrix.
\end{definition}

Note that $\BCBEop{k}{\l}{\l}{\L}$ introduces the subdivision matrix $\BCAmat{k}{\l}{\L}$ analogous to Theorem~\ref{theorem:consistency_of_BE_op} for $\dualAop{k}{\l}{\l}{\L}$ when we consider its action on the bases of $\FEECspaceZero{k}{\l}$ and $\FEECspaceZero{k}{\L}$.
Based on the restricted BE operator, we can next show that the spaces  $\SDFspaceZero{k}{\l}{\L}$ are the images of $\BCBEop{k}{\l}{\l}{\L}$.

\begin{proposition}
    \label{prop:BC_Aop_is_isomorphism}
    The spaces $\SDFspaceZero{k}{\l}{\L}$ satisfy $\SDFspaceZero{k}{\l}{\L} = \BCBEop{k}{\l}{\l}{\L} \big[ \FEECspaceZero{k}{\l} \big]$. Further, $\BCBEop{k}{\l}{\l}{\L}$ is an isomorphism
    in Eq.~\eqref{rbe_iso}.
\end{proposition}
\begin{proof}
    The first claim follows immediately from Definitions~\ref{def:adjoint_subdiv_operator} and~\ref{def:subdiv_spaces_with_vanishing_trace}. Further, by Eq.~\eqref{eq:BCFEEC_spaces_in_basis} and Proposition~\ref{prop:dimension_of_reduced_spaces}, the spaces $\FEECspaceZero{k}{\l}$ and $\SDFspaceZero{k}{\l}{\L}$ have the same number of dimensions. Since $\BCBEop{k}{\l}{\l}{\L}$ acts by an identity matrix, it has no kernel and is an isomorphism.
\end{proof}

The spaces $\SDFspaceZero{k}{\l}{\L}$ share essential features with the spaces $\SDFspace{k}{\l}{\L}$. In particular, they reproduce the closedness under subdivision from Proposition~\ref{prop:subdiv_operator_refined} and the nestedness from Proposition~\ref{prop:uniform_spaces_nested}. The proofs typically follow the pattern of restricting $\Aop{k}{\l}{\L}$ to $\FEECspaceZero{k}{\l}$, repeating the argument of the proof for $\FEECspace{k}{\l}$ or $\SDFspace{k}{\l}{\L}$, and finally using Eq.~\eqref{eq:subdiv_preserves_vanishing_trace} to deduce that the trace still vanishes.

\smallskip

We have now established all the necessary prerequisites to show that the spaces $\SDFspaceZero{k}{\l}{\L}$ comprise complexes. The next lemmas present properties of $\auxextd$ that will help us in Proof~\ref{proof:complex_with_vanishing_trace} of Theorem~\ref{theo:complex_with_vanishing_trace}.

\begin{lemma}
    \label{lemma:auxextd_with_BCs}
    The derivative $\auxextd$ is a map
    \begin{equation}
        \auxextd: \SDFspaceZero{k}{\l}{\L} \to \SDFspaceZero{k+1}{\l}{\L}
    \end{equation}
    and satisfies $\auxextd \circ \auxextd = 0$.
\end{lemma}
\begin{proof}
    The claim follows from the observation that the excluded basis functions associated to boundary simplices do not occur in the sum after taking derivatives, see Lemma~\ref{lemma:d_tilde_action_in_coordinates}. The derivative of a basis function associated to an interior simplex can be represented using only basis functions associated to interior edges. Derivatives of basis functions associated to interior edges are not problematic since we regard all faces as being interior. The property $\auxextd \circ \auxextd = 0$ follows from inclusion into $\FEECspaceZero{k}{\L}$.
\end{proof}

\begin{lemma}
    \label{lemma:commutativity_with_BCs}
    The restricted basis exchange operators $\BCBEop{k}{\l}{\l}{\L}$ satisfy
    \begin{equation}
        \Big(\auxextd \circ \BCBEop{k}{\l}{\l}{\L}\Big) \Big|_{\FEECspaceZero{k}{\l}} = \Big(\BCBEop{k+1}{\l}{\l}{\L} \circ \extd\Big) \Big|_{\FEECspaceZero{k}{\l}}
    \end{equation}
     on the domain $\FEECspaceZero{k}{\l}$.
\end{lemma}
\begin{proof}
    From Definition~\ref{def:BE_for_BC_spaces} follows that the claim is equivalent to
    \begin{equation}
        \Big( \auxextd \circ \dualAop{k}{\l}{\l}{\L} \Big|_{\FEECspaceZero{k}{\l}} \Big) \Big|_{\FEECspaceZero{k}{\l}} = \Big( \dualAop{k+1}{\l}{\l}{\L} \Big|_{\FEECspaceZero{k+1}{\l}} \circ \extd \Big) \Big|_{\FEECspaceZero{k}{\l}}.
    \end{equation}
    On the left, we can simply drop the restriction of $\dualAop{k}{\l}{\l}{\L}$ because the restriction already happened outside of the brackets. The right-hand side can be simplified using $\extd [ \FEECspaceZero{k}{\l} ] \subset \FEECspaceZero{k+1}{\l}$ since the FEEC spaces $\FEECspaceZero{k}{\l}$ form a differential complex. As a result, the image of $\extd$ is always in $\FEECspaceZero{k+1}{\l}$ and we can drop the restriction of $\dualAop{k+1}{\l}{\l}{\L}$.
    After both restrictions have been dropped from $\dualAop{k}{\l}{\l}{\L}$, the claim follows from Lemma~\ref{lemma:d_tilde_commuted_with_d}.
\end{proof}

The main Theorem~\ref{theo:complex_with_vanishing_trace} regarding the subdivision $k$-form spaces with vanishing trace BC claims that the spaces comprise discrete de Rham complexes. We now present the proof of the theorem.

\begin{numberedproof}[Proof of Theorem~\ref{theo:complex_with_vanishing_trace}]
    \label{proof:complex_with_vanishing_trace}
    The proof of this theorem is analogous to the one of Theorem~\ref{theo:uniform_de_rham_complex}. First, Lemma~\ref{lemma:auxextd_with_BCs} implies that the spaces $\SDFspaceZero{k}{\l}{\L}$ satisfy a differential complex.

    Next, we show that the differential complex is a relative de Rham complex. By Proposition~\ref{prop:BC_Aop_is_isomorphism} and Lemma~\ref{lemma:commutativity_with_BCs}, the restricted basis exchange operators $\BCBEop{k}{\l}{\l}{\L}$ are cochain isomorphisms between the complexes $\big( \FEECspaceZero{k}{\l}, \extd\big)$ and $\big( \SDFspaceZero{k}{\l}{\L}, \auxextd\big)$. Thus, $\BCBEop{k}{\l}{\l}{\L}$ induces an isomorphism on the cohomology groups of those complexes. The cohomology groups of the complex $\big( \SDFspaceZero{k}{\l}{\L}, \auxextd\big)$ are then isomorphic to the cohomology groups of the relative de Rham complex in Eq.~\eqref{eq:infinite_dim_relative_de_Rham}. Thus, we conclude that the subdivision $k$-form spaces with vanishing trace BC constitute a de Rham complex.
    \qed
\end{numberedproof}

\FloatBarrier

\section{Approximation of the physical domain}
\label{app:domain_approximation}

This appendix provides detailed information regarding the approximation of the domain described in Section~\ref{subsec:domain_approximation}. The algorithm enables us to account for the shrinking of the domain under subdivision as presented in Figure~\ref{subfig:Loop_refined_physical_domain}. The appendix focuses on the implementation of an algorithm that achieves a sufficiently good approximation of the domain $\M$ through Eq.~\eqref{mesh_construct}, i.e.
\begin{equation}
    \label{mesh_construct_appendix}
    \T^{\Loop}_\L \coloneqq \AopLoop{\lzero}{\L} \big( \T_{\lzero} \big) \approx \M,
\end{equation}
where $\T_\lzero$ denotes the first mesh in the hierarchy and is the output of the algorithm. The following notation was introduced in Section~\ref{subsec:domain_approximation}.
The mesh $\TFE$ is a mesh that approximates $\M$ as we would use it in standard FE methods.
The mesh hierarchy we are interested in starts at $\T_{\lzero}$ and ends with a finest mesh $\T_\L$ that approximates $\M$. $\AopLoop{\l}{\L}$ represents the Loop subdivision scheme \citep{Loop.1987} as discussed above. 

The algorithm also utilises the Whitney subdivision scheme.
The key feature of this subdivision scheme is that its invariants are exactly the piece-wise linear basis functions of FEEC, i.e. $\Aop{k, \text{Whit}}{\l}{\L}\freeQ{\FEECbasis}{k}{\l}{\coarseIdxone} = \freeQ{\FEECbasis}{k}{\l}{\coarseIdxone}$. This means that Whitney subdivision merely re-expresses the basis function $\freeQ{\FEECbasis}{k}{\l}{\coarseIdxone}$ in the finer basis $\freeVec{\FEECbases}{k}{\L}$ without changing it.
Since the barycentric coordinate functions are also piece-wise linear, Whitney mesh subdivision carries out barycentric quadrisection, i.e. new vertices are introduced exactly at the midpoints of edges, and so the boundary of the mesh does not change under Whitney refinement. Denoting the Whitney mesh subdivision operator by $\AopWhit{\l}{\L}$
(see Remark~\ref{remark_othersubdivschemes}), we can express this fact as
\begin{equation}
    \TFE \approx \M \;\; \implies \;\; \T_\L^{\Whit} \coloneqq \AopWhit{\lzero}{\L} \big(\TFE\big) \approx \M.
\end{equation}

A suitable initial mesh $\T_{\lzero}$ can be obtained using a weighted least squares method that minimizes the distance of the vertex positions of the fine Whitney mesh $\T_\L^{\Whit}$ and fine Loop mesh $\T^{\Loop}_\L$.
First, denote the positions of the vertices of a mesh $\T_\l$ by $\Vec{x}_\l \in \mathbb{R}^{\vert\V_\l \vert \times d}$, where $d$ is the dimension of the ambient space. In the following, we will assume $d=3$. We can then write
\begin{equation}
    \Vec{x}_\l = \Big[ \vertexpos{\l}{1} \;\; \vertexpos{\l}{2} \;\; \hdots \;\; \vertexpos{\l}{\vert \V_\l \vert}\Big]^T = \begin{bmatrix}
        (x^1_{\l})_{1} & (x^1_{\l})_{2} & \hdots & (x^1_{\l})_{\vert \V_\l \vert} \\[4pt]
        (x^2_{\l})_{1} & (x^2_{\l})_{2} & \hdots & (x^2_{\l})_{\vert \V_\l \vert} \\[4pt]
        (x^3_{\l})_{1} & (x^3_{\l})_{2} & \hdots & (x^3_{\l})_{\vert \V_\l \vert}
    \end{bmatrix}^T,
\end{equation}
where $\vertexpos{\l}{\coarseIdxone} = \big[ x^1_\coarseIdxone \;\; x^2_\coarseIdxone \;\; x^3_\coarseIdxone \big]$ denotes the position of the vertex $\vertex{\l}{\coarseIdxone}$. Further, we denote the vectors storing only the $n^{\text{th}}$ coordinate of the vertices by $\Vec{x}^{\;n}_{\l} = \big[ (x^n_{\l})_{1} \;\; (x^n_{\l})_{2} \;\; \hdots \;\; (x^n_{\l})_{\vert \V_\l \vert} \big]^T$ for $n\in\{ 1,2,3\}$.

We now proceed to set up the least squares method. To introduce the weighting, we define the matrix-induced norm
\begin{equation}
    \vvert{\vvert{\Vec{x}}}_W \coloneqq \sqrt{\Vec{x}^T \cdot W \cdot \Vec{x}},
\end{equation}
where $W$ is a diagonal matrix specified below. Let $\Vec{x}_{\lzero}$ and $\Vec{x}^{\,\Whit}_{\L}$ be the vertices of $\T_{\lzero}$ and $\T_{\L}^{\Whit}$, respectively.
The error $E$ of our mesh approximation is then measured through
\begin{equation}
    \label{eq:lsq_error}
    E\big(\Vec{x}^{\;n}_{\lzero}) \coloneqq \big\vert \big\vert \Vec{x}^{\,\Whit, n}_\L - \Amat{\Loop}{\lzero}{\L} \cdot \Vec{x}^{\;n}_{\lzero} \big\vert \big\vert^2_W \, ,
\end{equation}
where $\Amat{\Loop}{\lzero}{\L} $ is the matrix representation of $\AopLoop{\lzero}{\L}$.
This error achieves a minimum value iff
\begin{equation}
    \label{eq:normal_equation}
    \frac{\partial E}{\partial \Vec{x}^{\;n}_{\lzero}} = 0 \quad \Leftrightarrow \quad \big( \Amat{\Loop}{\lzero}{\L}\big)^T \cdot W \cdot \Amat{\Loop}{\lzero}{\L} \cdot \Vec{x}^{\;n}_{\lzero} = \big( \Amat{\Loop}{\lzero}{\L}\big)^T \cdot W \cdot \Vec{x}^{\,\Whit,n}_\L.
\end{equation}
By solving this equation for $n\in\{ 1,2,3 \}$, we obtain the coordinates $x^{\;n}_{\lzero}$ of the vertices of a mesh $\T_{\lzero}$ for which $\T_\L^\Loop \approx \M$. Note that we can carry out the least squares fit individually for each $n$ since subdivision also averages the coordinates $x^{\,n}$ separately.

\begin{figure}[]

    \sbox\twosubbox{%
      \resizebox{\dimexpr.95\textwidth-1em}{!}{%
        \includegraphics[height=1.05cm]{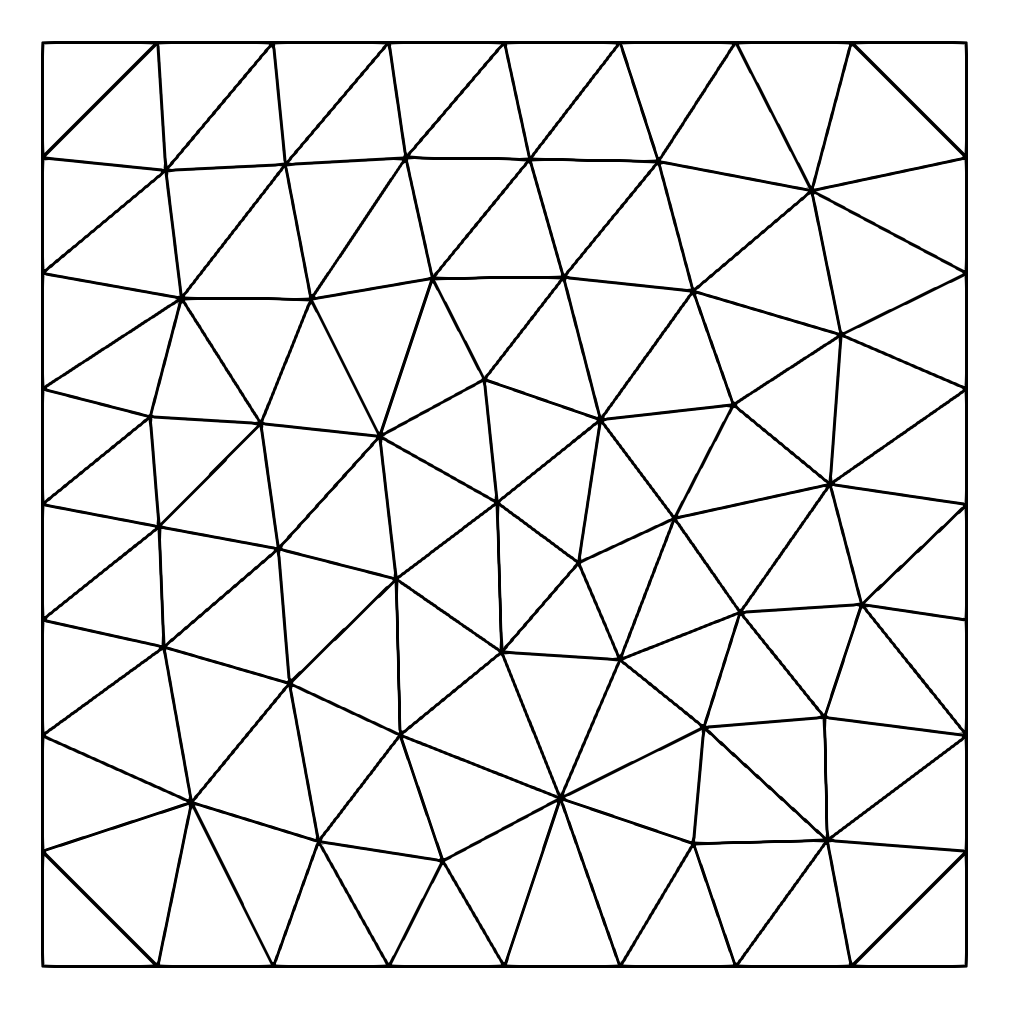}%
        \includegraphics[height=1.05cm]{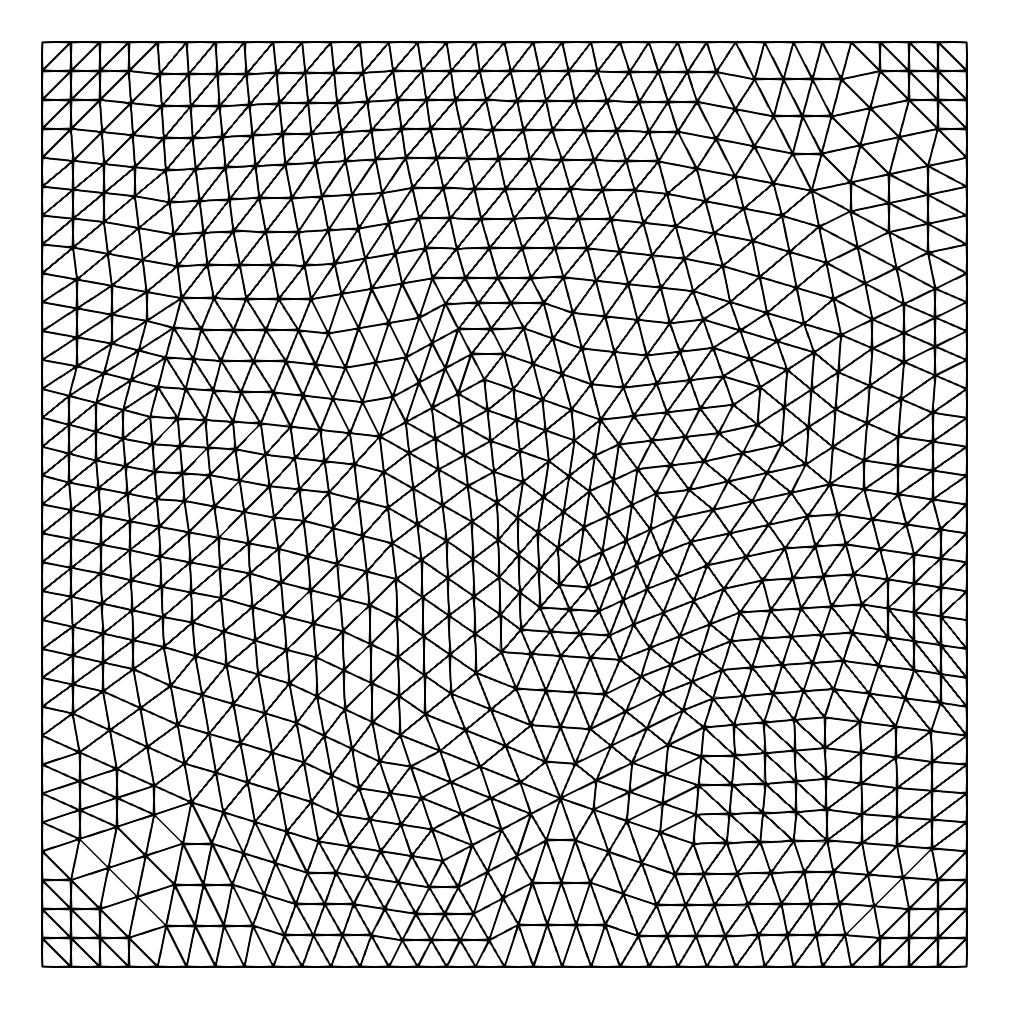}%
        \includegraphics[height=1.05cm]{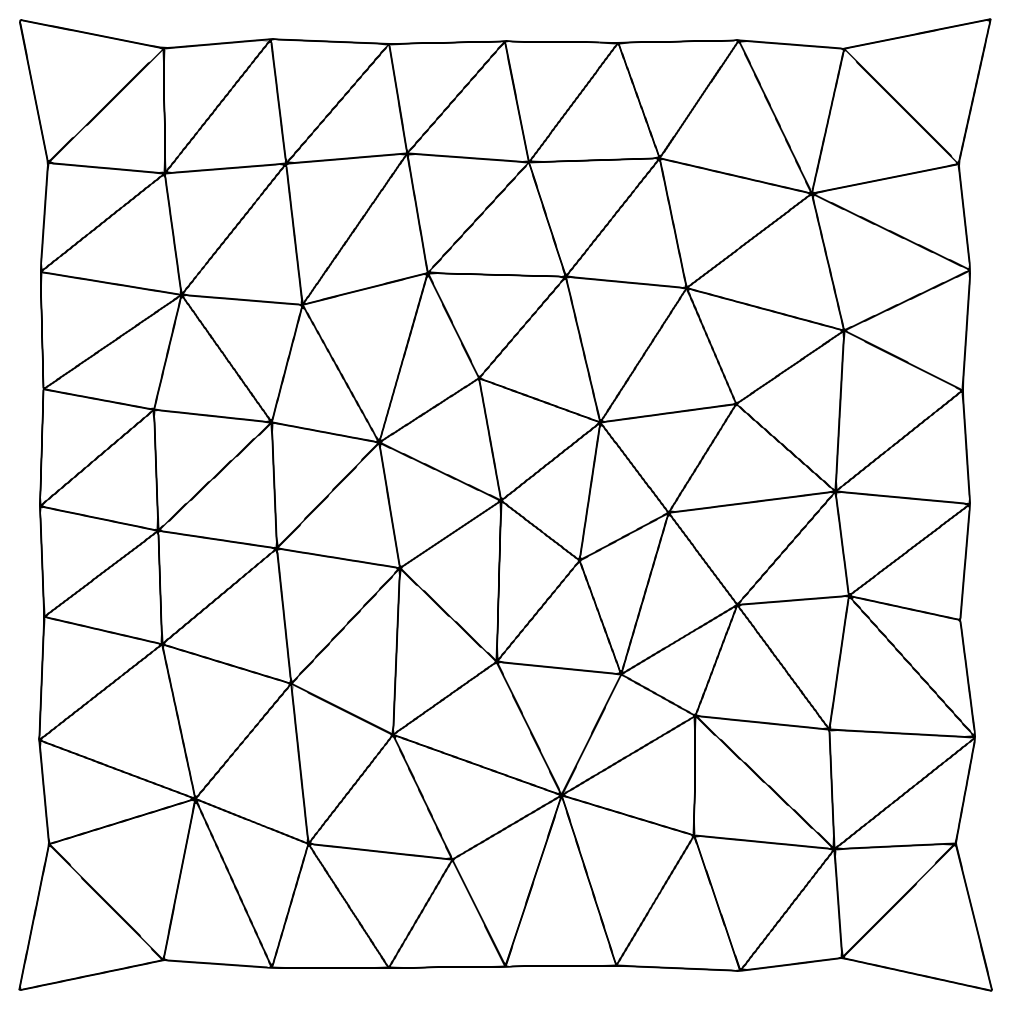}%
        \includegraphics[height=1.05cm]{Loop_refined_new_domain_surface_with_old_domain_mesh.png}
      }%
    }
    \setlength{\twosubht}{\ht\twosubbox}


    \centering

    \subcaptionbox{Reference mesh $\TFE$ of $\M$.\label{subfig:0_form_limit}}{%
      \includegraphics[height=\twosubht]{domain_mesh.png}%
    }\quad
    \subcaptionbox{Whitney refined mesh $\T^{\Whit}_{\L}$.\label{subfig:withney_form_limit}}{%
      \includegraphics[height=\twosubht]{Withney_refined_domain_mesh_wireframe.png}%
    }\quad
    \subcaptionbox{Least-squares optimal mesh $\T_{\lzero}$.\label{subfig:1_form_limit}}{%
      \includegraphics[height=\twosubht]{new_domain_wireframe.png}
    }\quad
    \subcaptionbox{Loop refined adjusted mesh $\T_{\L}^{\Loop}$.\label{subfig:2_form_limit}}{%
      \includegraphics[height=\twosubht]{Loop_refined_new_domain_surface_with_old_domain_mesh.png}
    }
    \caption{Depiction of the results of the steps involved in finding the least squares approximation of $\M$ on level $\L$, from left to right (cf. Alg.~\ref{algo:domain_approximation}).}
    \label{fig:domain_approximation_steps}
\end{figure}

The weights in $W$ are chosen to prioritise boundary vertices as their positions are particularly crucial to the overall shape of the mesh domain. For example, $\T_{\lzero}$ in Figure~\ref{subfig:1_form_limit} has been obtained by setting
\begin{equation}
    W_{\coarseIdxone\coarseIdxone} = \begin{cases}
        w_1 \;\;\ \text{if } \vertex{\L}{\coarseIdxone} \text{ is a corner vertex}, \\
        w_2 \;\;\ \text{if } \vertex{\L}{\coarseIdxone} \text{ is a boundary vertex but not at a corner}, \\
        w_3 \;\;\ \text{if } \vertex{\L}{\coarseIdxone} \text{ is an interior vertex created by refining a coarse vertex}, \\
        w_4 \;\;\ \text{if } \vertex{\L}{\coarseIdxone} \text{ is a newly created interior vertex},
    \end{cases}
\end{equation}
with $w_1 = 10$, $w_2 = 1$, $w_3 = 1$, $w_4 = 0.01$. These weights can be changed to capture different domains $\M$ more effectively.

Algorithm~\ref{algo:domain_approximation} summarises the steps involved in the approximation of the physical domain.
The result of each step is illustrated in Figure~\ref{fig:domain_approximation_steps}. In the following,
we will approximate $\M$ using Algorithm~\ref{algo:domain_approximation} such that $\T^{\Loop}_\L \approx \M$.

\begin{algorithm}[htb]
    \BlankLine
    \KwData{Physical domain $\M$ on which the PDE is defined}
    \KwResult{Initial mesh $\T_{\lzero}$ such that $\AopLoop{\lzero}{\L} \big(\T_{\lzero}\big) \approx \M$}
    \BlankLine
    \TitleOfAlgo{}
    \quad - obtain standard FE reference mesh $\TFE$ of $\M$ \\
    \quad - perform Whitney refinement: $\T_\L^{\Whit} = \AopWhit{\lzero}{L} \big(\TFE\big)$ \\
    \quad - find vertex positions $\Vec{x}_{\lzero}$ that minimise the error in Eq.~\eqref{eq:lsq_error} using Eq.~\eqref{eq:normal_equation} \\
    \quad - obtain $\T_{\lzero}$ by copying topology of $\TFE$ and setting the vertex positions to $\Vec{x}_{\lzero}$
    \caption{Approximation of the physical domain (as illustrated in Figure~\ref{fig:domain_approximation_steps})}
    \label{algo:domain_approximation}
\end{algorithm}

\begin{remark}
There are other approaches one could consider to obtain a mesh $\T_{\lzero}$ for which $\AopLoop{\lzero}{\L}\big(\T_{\lzero}\big) \approx \M$. For example, we could use Whitney subdivision for the mesh and Loop subdivision for the differential $0$-forms on that mesh. However, there is an issue with that. Let $\constrVec{X}{}{\lzero}{\L}$ be the coordinate function of $\T_\L^{\Whit} = \AopWhit{\lzero}{\L}\big(\T\big)$ and $\constrVec{\varphi}{}{\lzero}{\L}$ an associated chart. Then, consider
\begin{equation}
    \constrVec{\omega}{0}{\lzero}{\L} \circ \constrVec{\varphi}{}{\lzero}{\L} = \Big( \Aop{\Loop}{\lzero}{\L} \freeVec{\omega}{0}{\lzero} \Big) \circ \Big( \AopWhit{\lzero}{\L} \freeVec{\varphi}{}{\lzero} \Big).
\end{equation}
Letting $\L \to \infty$, we find that $\constrVec{\varphi}{}{\lzero}{\L}$ will stay in $\CG_1 \subset \mathcal{C}^0$ while the intention of subdivision is that the limit functions are in $\mathcal{C}^k$ for some $k > 0$. Thus, the Whitney parameterisation is not smooth enough and we sacrifice differentiability of the limit functions, cf. \citep{Loop.1987}). However, we experienced problematic behaviour in our simulations already for small $\L$.

Another approach is to follow the ideas of standard FE methods to find a geometry map in the space of Loop subdivision $0$-forms. This leads to the following weak problem formulation: Find an approximate coordinate function $\constrVec{X}{}{\lzero}{\L} \in \SDFspace{0}{\lzero}{\L}$ such that
\begin{equation}
    \big( \constrVec{X}{}{\lzero}{\L}, \, v \big) = \big( X, \, v \big) \qquad \forall v \in \SDFspace{0}{\lzero}{\L},
\end{equation}
where $X$ denotes the global coordinate function of the physical domain.
However, that also fails for a subtle reason: the basis functions of the subdivision space inherently rely on a fixed mesh $\T_{\lzero}$. Since this mesh is now the independent variable, our basis functions are also unknown during the projection and we cannot compute the inner products.
\end{remark}

\end{document}